\documentclass[11pt]{amsart}
\usepackage{amssymb}
\usepackage{amsmath}
\usepackage{bbm}
\usepackage{graphicx} 
\usepackage[english]{babel} 
\usepackage{float}
\usepackage{mathtools}
\usepackage{hyperref}
\usepackage[noadjust]{cite}
\usepackage[margin=3.2cm]{geometry}
\hypersetup{
    colorlinks=true,       
    linkcolor=blue,          
    citecolor=blue,        
    filecolor=blue,      
    urlcolor=blue           
}

\usepackage{tikz}
\usetikzlibrary{
  knots,
  hobby,
  decorations.pathreplacing,
  shapes.geometric,
  calc
}

\tikzset{
  knot diagram/every strand/.append style={
    ultra thick,
    red
  },
  show curve controls/.style={
    postaction=decorate,
    decoration={show path construction,
      curveto code={
        \draw [blue, dashed]
        (\tikzinputsegmentfirst) -- (\tikzinputsegmentsupporta)
        node [at end, draw, solid, red, inner sep=2pt]{};
        \draw [blue, dashed]
        (\tikzinputsegmentsupportb) -- (\tikzinputsegmentlast)
        node [at start, draw, solid, red, inner sep=2pt]{}
        node [at end, fill, blue, ellipse, inner sep=2pt]{}
        ;
      }
    }
  },
  show curve endpoints/.style={
    postaction=decorate,
    decoration={show path construction,
      curveto code={
        \node [fill, blue, ellipse, inner sep=2pt] at (\tikzinputsegmentlast) {}
        ;
      }
    }
  }
}

\usepackage[pdf]{pstricks} 
\usepackage{epsfig}
\usepackage{pst-grad} 
\usepackage{pst-plot} 
\usepackage[space]{grffile} 
\usepackage{etoolbox} 
\makeatletter 
\patchcmd\Gread@eps{\@inputcheck#1 }{\@inputcheck"#1"\relax}{}{}
\makeatother

\usepackage[all]{xy}
\usepackage{mathrsfs}
\usepackage{graphics}
\usepackage{amsthm}

\theoremstyle{plain}\newtheorem{theorem}{Theorem}[section]\newtheorem{Theorem}{Theorem}\newtheorem{Corollary}{Corollary}\newtheorem{proposition}[theorem]{Proposition}\newtheorem{lemma}[theorem]{Lemma}\newtheorem{corollary}[theorem]{Corollary}
\def\eq{\coloneqq}

\theoremstyle{definition}\newtheorem{definition}[theorem]{Definition}\newtheorem{example}[theorem]{Example}\newtheorem{notation}[theorem]{Notation}\newtheorem{remark}[theorem]{Remark}

\def\C{\mathbb{C}}\def\N{\mathbb{N}}\def\Z{\mathbb{Z}}\def\R{\mathbb{R}}\def\T{\mathbb{T}}

\def\ZZ2{\mathbb{\Z/ 2\Z}}
\def\sb{\subset}\def\su{\subset}
\def\lb{\langle}\def\rb{\rangle}\def\ot{\otimes}\def\t{\times}\def\sm{\setminus}

\def\c{\gamma}

\def\a{\alpha}\def\b{\beta}\def\d{\delta}\def\e{\epsilon}\def\s{\sigma}\def\De{\Delta}\def\La{\Lambda}\def\o{\omega}\def\la{\lambda}\def\la{\lambda}\def\p{\partial}\def\S{\Sigma}
\def\mb{\mathbb}\def\mf{\mathfrak}\def\mc{\mathcal}\def\ov{\overline}\def\wt{\widetilde}

\def\gl{\mathfrak{gl}}
\def\sl2{\mathfrak{sl}_2}
\def\su2{\mathfrak{su}(2)}
\def\Aut{\text{Aut}\,}
\def\id{\text{id}}\def\Ker{\text{Ker}\,}
\def\Im{\text{Im}\,}\def\Hom{\text{Hom}\,}
\def\char{\text{char}\,}\def\aug{\text{aug}\,}
\def\inte{\text{int} \,}
\def\Ham{\text{Ham}}

\def\Spinc{\text{Spin}^c}

\def\TH3{\Theta_3^{H}}

\def\x{\textbf{x}}\def\y{\textbf{y}}
\def\aa{\boldsymbol\a}\def\bb{\boldsymbol\b}
\def\HD{(\S,\aa,\bb)}\def\HH{\mc{H}}

\def\Tab{\T_{\a}\cap\T_{\b}}

\def\Sa{\S\sm\aa}\def\Sb{\S\sm\bb}

\def\ss{\mf{s}}

\def\L{\mc{L}}

\def\Uq{U_q(\mf{sl}_2)}\def\Usl2{U_q(\sl2)}\def\usl2{\wt{U}_q(\sl2)}\def\Uqgl11{U_q(\mathfrak{gl}(1|1))}\def\Uqsl11{U_q(\mathfrak{sl}(1|1))}\def\Uq{U_q(\mf{sl}_2)}

\def\Rep{\text{Rep}}

\def\II1{\text{II}_1}

\def\mod2{\ (mod \ 2)}
\def\L2{L^2}\def\l2{l^2(G)}\def\Cn2{C_n^{(2)}}\def\Hn2{H_n^{(2)}}\def\bn2{b_n^{(2)}}\def\pn2{\p_n^{(2)}}
 
\def\W1p{W^{1,p}}\def\Wd1p{W_{\d}^{1,p}}

\def\L{\mathcal{L}}

\def\ZZ{\mathcal{Z}}

\def\kk{\mathbb{K}}\def\gl{\mathfrak{gl}}\def\gl11{\mathfrak{gl}(1|1)}

\def\wtmu{\wt{\mu}}

\def\bolda{\boldsymbol{a}}
\def\aaa{\aa\cup\bolda}

\def\aaa{\aa^e}

\def\HDD{(\S,\aaa,\bb)}
\def\wtmu{\mu}

\def\Ua{U_{\a}}\def\Ub{U_{\b}}

\def\nequiv{\not\equiv}\def\gg{\mathfrak{g}}\def\Hn{H(n)}

\def\Ha{H_{\a}}\def\Ham{H_{\a^{-1}}}\def\Haot{H_{\a_1\a_2}}

\def\Sa{S_{\a}}\def\Sam{S_{\a^{-1}}}

\def\pia{G}

\def\Dea{\De_{\a}}

\def\maot{m_{\a_1,\a_2}}

\def\ca{c_{\a}}

\def\Haot{H_{\a_1\a_2}}\def\HaoHat{H_{\a_1}\ot H_{\a_2}}\def\Dea{\De_{\a}}\def\ea{\e_{\a}}

\def\ca{\coint_{\a}}

\def\ia{i_{\a}}

\def\Aut{\text{Aut}}

\def\uH{\underline{H}} \def\wtH{\wt{H}}\def\uHH{\uH=\{\Ha\}_{\a\in G}}

\def\gda{g^*_{\a}}

\def\lH{r_H}

\def\rhoc{\rho\ot h}
\def\ZHrho{Z_H^{\rho}}

\def\rhoc{\rho\ot h}
\def\ZHrho{Z_H^{\rho}}

\def\LaV{\La(V)}

\def\ovbx{\ov{\b}_x}

\def\FM{F_M}\def\kkM{\kk_M}\def\rhoc{\rho\ot h}\def\HM{H_M}

\def\coint{\boldsymbol{c}}
\def\int{\boldsymbol{\mu}}

\def\LaV{\La(V)}

\def\ZHrho{Z^{\rho}_H}\def\rhophi{\rho_{\phi}}

\def\rhoc{\rho\ot h}\def\Der{\De^{\rhoc}}

\def\ovtau{\tau_0}

\begin{document}

\title[Twisting Kuperberg invariants via Fox calculus]{Twisting Kuperberg invariants via Fox calculus and Reidemeister torsion}
\author{Daniel L\'opez Neumann}
\address{Institut de Math\'ematiques de Jussieu - Paris Rive Gauche, Universit\'e Paris Diderot, F-75013 Paris, France}
\email{lopezd@imj-prg.fr}

\maketitle

\begin{abstract}


We study Kuperberg invariants for sutured manifolds in the case of a semidirect product of an involutory Hopf superalgebra $H$ with its automorphism group $\Aut(H)$. These are topological invariants of balanced sutured 3-manifolds endowed with a homomorphism of the fundamental group into $\Aut(H)$ and possibly with a $\Spinc$ structure and a homology orientation. We show that these invariants are computed via a form of Fox calculus and that, if $H$ is $\N$-graded, they can be extended in a canonical way to polynomial invariants. When $H$ is an exterior algebra, we show that this invariant specializes to a refinement of the twisted relative Reidemeister torsion of sutured 3-manifolds. We also give an explanation of our Fox calculus formulas in terms of a particular Hopf group-algebra.
\end{abstract}

\tableofcontents

\section[Introduction]{Introduction}

Topological invariants of knots and 3-manifolds may be built, essentially, from two methods. 
On the one hand, one can define topological invariants using the tools of classical algebraic topology. These are referred as {\em classical invariants} and include the Alexander polynomial of knots \cite{Alexander:topological} and, more generally, the Reidemeister torsion of 3-manifolds \cite{Reidemeister}. These invariants contain deep topological information and, as initiated by Lin in the 90's, they can be further strengthened by twisting with a (non-abelian) representation of the fundamental group \cite{Lin:representations}. The resulting twisted invariants turn out to be extremely powerful, for instance, they sometimes detect mutation and non-invertibility of some knots \cite{Wada:twisted, KL:twisted, KL:mutation} and, when taken all together, they detect the Seifert genus of a knot and whether a 3-manifold fibers over the circle \cite{FV:Thurston, FV:fibered}.

\medskip



On the other hand, one can build invariants using the tools of {\em quantum topology}, a vast domain that relates low dimensional topology to quantum field theory and representation theory. 
These are the so called {\em quantum invariants}, which were introduced during the 80's through the pioneering works of Jones, Witten, and Reshetikhin-Turaev \cite{Jones:polynomial, Witten:quantum, RT1, RT2}. 
Mathematically speaking, the construction of quantum invariants relies on the theory of (braided) monoidal categories, Hopf algebras and, in particular, quantum groups. For instance, the Jones polynomial of links and the Witten-Reshetikhin-Turaev invariants (WRT) of closed 3-manifolds are built through the monoidal category of representations of the quantum group $\Uq$ \cite{Turaev:BOOK1}. One can also build invariants of closed 3-manifolds directly from an arbitrary finite dimensional Hopf algebra $H$, through the theory of Hopf algebra (co)integrals. This method was introduced by Kuperberg \cite{Kup1, Kup2} and Hennings \cite{Hennings:invariants}, in different settings. These two approaches are now known to be essentially equivalent \cite{CC:ontwoinvariants} and they relate to WRT when $H=\Uq$ at a root of unity \cite{CKS:relation-WRT-Henn}. 
There exists also a theory of 
quantum invariants of pairs $(M,\rho)$, where $M$ is either a closed 3-manifold or a link complement and $\rho$ is a homomorphism of $\pi_1(M)$ into some group $G$. These are obtained by extending the previous methods to monoidal categories or Hopf algebras graded by $G$ \cite{Turaev:homotopy, Turaev:BOOK-HQFT}.



\medskip

It turns out that the above two families of invariants are not disjoint, even though they are built from very different methods. More precisely, some classical invariants can be realized as quantum invariants, in general through the representation theory of an appropiate Hopf algebra $H$. For instance, the Alexander polynomial of links in the three-sphere can be obtained through such methods if $H=\Uqgl11$ \cite{Reshetikhin:supergroup, RS:Alexander} or $\Uq$ at $q=i$ \cite{Murakami:Alexander}. The abelian Reidemeister torsion of closed 3-manifolds has been shown to be obtained via an ``unrolled" version of $\Uq$ at $q=i$ \cite{BCGP}. 
More recently, a special instance of the $SL(2,\C)$-twisted torsion of the complement of a link $L\sb S^3$ has been obtained through the representation theory of a graded object associated to $\Uq$ (in its unrestricted version) \cite{McPhail-Snyder:holonomy}. In another direction, the author showed that the abelian relative torsion of balanced sutured 3-manifolds can be obtained through a generalization of the Hopf algebraic approach of Kuperberg, specialized to the Borel part of $\Uqgl11$ \cite{LN:Kup}. However, neither of these works capture the more general aspects of Reidemeister torsion theory, such as $GL(n,\C)$-twisted Alexander polynomials, torsion of links in arbitrary homology spheres, etc. Understanding Reidemeister torsion as part of quantum topology seems an important issue in order to find topological applications of quantum invariants, which often are weaker than their classical counterparts. 

\medskip

In this paper, we take a step in this direction by showing that several aspects of Reidemeister torsion theory, such as Fox calculus, $\Spinc$ refinements and twisted polynomials, are general Hopf algebra constructions and therefore belong to the realm of quantum topology. The twisted relative Reidemeister torsion of balanced sutured 3-manifolds, hence also twisted Alexander polynomials of links, is shown to be a special case of such construction. We achieve this through our sutured manifold extension of involutory Kuperberg invariants \cite{LN:Kup} restricted to a semidirect product $\kk[\Aut(H)]\ltimes H$. The key idea is to consider this semidirect product relative to $\kk[\Aut(H)]$, or equivalently, as a Hopf group-algebra graded by $\Aut(H)$. 
Thus, our results indicate that Reidemeister torsion naturally belongs to the world of homotopy quantum field theory (HQFT) of Turaev \cite{Turaev:homotopy, Turaev:BOOK-HQFT}. 

\medskip

\subsection*{Background} To describe our results in detail we recall a few notions and previous work. First, recall that a balanced sutured 3-manifold is a pair $(M,\c)$ where $M$ is a 3-manifold with non-empty boundary and $\c$ is a collection of annuli in $\p M$ dividing the boundary into two homeomorphic pieces $R_{-}(\c)$ and $R_+(\c)$ \cite{Gabai:foliations}. For instance, the complement of a link in an arbitrary closed 3-manifold can be considered as a sutured manifold by letting $\c$ consist on two meridians on each boundary component. A fundamental topological invariant of balanced sutured manifolds is the twisted relative Reidemeister torsion $\tau^{\rho}(M,R_-(\c))\in\kk$, which depends on $(M,\c)$ together with a homomorphism $\rho:\pi_1(M)\to GL(V)$ for some finite dimensional vector space $V$ over a field $\kk$ \cite{Turaev:BOOK2}. If $h:\pi_1(M)\to H_1(M)$ is the Hurewicz map, one can extend the torsion to $\tau^{\rho\ot h}(M,R_-(\c))\in\kk[H_1(M)]$. This extension generalizes the twisted Alexander polynomials of links and for $\rho\equiv 1$, it is the Euler characteristic of a Floer homology invariant of sutured manifolds \cite{FJR11}. The torsion is actually defined up to a $\pm\det\rho(g)$ indeterminacy with $g\in\pi_1(M)$, but this can be corrected by picking a $\Spinc$ structure $\ss$ and a homology orientation $\o$ (i.e. an orientation of the vector space $H_*(M,R_-(\c);\R)$) \cite{Turaev:BOOK2, FJR11} so it has the form $$\tau^{\rho}(M,R_-(\c),\ss,\o)\in\kk.$$ 

\medskip


Recall also that given an arbitrary finite dimensional Hopf (super)algebra $H$ over a field $\kk$, Kuperberg defines a topological invariant $I^{Kup}_H(Y,f)\in\kk$ of a closed oriented 3-manifold $Y$ endowed with a framing $f$ of its tangent bundle \cite{Kup2}. This relies on the Heegaard diagrammatic presentation of closed 3-manifolds along with the theory of Hopf algebra (co)integrals (e.g. \cite{Radford:BOOK}). When $H$ is the Borel of $\Uq$ at a root of unity, this invariant is related to WRT invariants \cite{CC:ontwoinvariants, CKS:relation-WRT-Henn}. 
However, the appearance of framings makes the computation of this invariant quite challenging. If $H$ is involutory, i.e. $S^2=\id_H$ where $S$ is the antipode of $H$, Kuperberg's invariant is independent of the framing \cite{Kup1} hence we denote it by $$I_H^{Kup}(Y)\in\kk.$$ 
The involutory invariant admits an extension to pairs $(Y,\rho)$, where $Y$ is a closed 3-manifold and $\rho:\pi_1(Y)\to G$ is a homomorphism into some group $G$ \cite{Virelizier:flat}. This relies on a finite-type involutory Hopf $G$-coalgebra $\uHH$ (which amounts to a Hopf algebra graded by $G$) \cite{Turaev:homotopy, Virelizier:Hopfgroup} and is denoted $$I^{\rho}_{\uH}(Y)\in\kk.$$ 
This specializes to $I_{H_1}^{Kup}(Y)$ if $\rho$ is trivial. 
It has to be noted that the assumptions of \cite{Kup1} and \cite{Virelizier:flat} imply semisimplicity of the relevant Hopf algebras \cite{LR:cosemisimplechar0}. Recently, a candidate for an unframed version of Kuperberg's invariant has been proposed only assuming unimodularity of $H$ \cite{CGPT}. However, neither of these unframed approaches can be directly related to WRT, since the Borel of $\Uq$ at a root of unity is non-involutory and non-unimodular.

\medskip


Now, in \cite{LN:Kup}, the author generalized Kuperberg's involutory invariant to balanced sutured 3-manifolds. This construction relied on sutured Heegaard diagrams \cite{Juhasz:holomorphic} along with relative versions of the Hopf algebra (co)integrals. More precisely, the algebraic input consisted essentially of an involutory Hopf superalgebra $\wt{H}$ (possibly of infinite dimension) relative to a pair of Hopf subalgebras $A,B\sb \wt{H}$, where $A$ is the domain of the cointegral and $B$ is the target of the integral. When $B=\kk$ the output is an invariant $$\wt{I}_{\wt{H}}^{\rho}(M,\c,\ss,\o)\in\kk$$ where $(M,\c)$ is a balanced sutured 3-manifold, $\rho:\pi_1(M)\to G(A)$ is a homomorphism into the group-likes of $A$, $\ss\in\Spinc(M,\c)$ and $\o$ is a homology orientation. In contrast to the above unframed approaches \cite{Kup1, Virelizier:flat, CGPT}, here $\wt{H}$ may be non-unimodular and the appearance of $\ss$ and $\o$ is related to this, indeed, $\wt{I}^{\rho}_{\wt{H}}$ depends on $\ss,\o$ up to $\pm a^*(\rho(g))$ where $a^*:G(A)\to\kk^{\t}$ is the distinguished group-like of $\wt{H}$ rel $A$ and $g\in\pi_1(M)$. When $\wt{H}$ is the Borel of $\Uqgl11$ with an appropriate relative integral, we showed that $\wt{I}_{\wt{H}}^{\rho}$ is a refinement of the abelian torsion $\tau^{\rho}(M,R_-(\c))$, i.e. $\rho:H_1(M)\to\kk^{\t}$, 
and we did this by relating the coproduct of $\wt{H}$ to Fox calculus.


\subsection*{Main results} In the present work, we show that a semidirect product $\wt{H}=\kk[\Aut(H)]\ltimes H$ fits into the setting of \cite{LN:Kup} with $A=\kk[\Aut(H)]$ and $a^*=\lH:\Aut(H)\to\kk^{\t}$ (where $\lH$ is as in Definition \ref{def: rH}) and we study the resulting invariant 
$\wt{I}_{\kk[\Aut(H)]\ltimes H}^{\rho}(M,\c,\ss,\o)\in\kk$, where $\rho:\pi_1(M)\to\Aut(H)$ is a homomorphism and $\ss,\o$ are as above. 
We show that when the definining formula of this invariant is rewritten only in terms of the structure tensors of $H$ and the homomorphism $\rho$, then one gets a formula very similar to the original one of Kuperberg \cite{Kup1} but in which $\rho$ twists the structure tensors via Fox calculus. Since the algebraic input only depends on $H$, we denote this formula by $I_H^{\rho}(M,\c,\ss,\o)$ and call it a {\em twisted Kuperberg invariant}. We develop the formula for $I_H^{\rho}$ independently of \cite{LN:Kup} in Section \ref{sect: the invariant}. Thus, we can state our first main result as follows.

\medskip



\begin{Theorem}
\label{Theorem: I for sd product is Fox calculus}
Let $H$ be a finite dimensional involutory Hopf superalgebra over a field $\kk$ with a two-sided cointegral and integral. Let $(M,\c)$ be a balanced sutured 3-manifold, $\rho:\pi_1(M)\to\Aut(H)$ a group homomorphism, $\ss\in\Spinc(M,\c)$ and $\o$ an orientation of $H_*(M,R_-(\c);\R)$. Then the invariant $\wt{I}^{\rho}_{\wt{H}}$ at a semidirect product coincides with the Fox calculus invariant $I_H^{\rho}$:

$$\wt{I}_{\kk[\Aut(H)]\ltimes H}^{\rho}(M,\c,\ss,\o)= I_H^{\rho}(M,\c,\ss,\o).$$

\end{Theorem}

We also provide an explanation of the Fox calculus formula that defines $I_H^{\rho}$ in terms of Hopf group-algebras (the dual notion of a Hopf group-coalgebra). We show that the semidirect product $\kk[\Aut(H)]\ltimes H$ determines a Hopf $\Aut(H)$-algebra $\uH$ for which  $$I^{\rho}_{\uH}(Y)=I_{H}^{\rho}(M_0,\c_0),$$ 
where the left hand side is the (dual version of the) invariant of Virelizier \cite{Virelizier:flat} and the right hand side is ours for $M_0=Y\sm B^3$, where $B^3$ is an open 3-ball embedded in $Y$ and $\c_0$ is a single suture in $\p M_0$, see Proposition \ref{prop: Virelizier at sd product is ours}. 
Moreover, we relate the relative integral approach of \cite{LN:Kup} to the Hopf group-(co)algebra approach. Therefore, our construction can also be understood as a generalization to sutured manifolds of \cite{Virelizier:flat} restricted to a particular Hopf group-algebra (which in our case may be non-semisimple). In addition, the Hopf group-algebra approach clarifies the appearance of $\Spinc$ structures and homology orientations, see Remark \ref{remark: Spinc and o vs non-unimodular}.
\medskip

Our construction has the additional feature that it can be extended to define polynomial invariants provided $H$ is $\N$-graded. Indeed, in such a case, any $\rho:\pi_1(M)\to\Aut(H)$ can be combined with the Hurewicz map $h:\pi_1(M)\to H_1(M;\Z)$ to define a homomorphism $\rho\ot h:\pi_1(M)\to\Aut(H\ot_{\kk}\kk[H_1(M)])$. Hence, there is an invariant $$I_H^{\rho\ot h}(M,\c,\ss,\o)\in\kk[H_1(M)]$$
which we call a {\em twisted Kuperberg polynomial} and this specializes to $I_H^{\rho}(M,\c,\ss,\o)$ via the augmentation map $\aug:\kk[H_1(M)]\to\kk$. 
In particular, this procedure defines (Hopf algebraic) twisted multivariable polynomial invariants of links in arbitrary homology spheres. When $\rho\equiv 1$, this link invariant can be considered as a ``polynomial deformation" of the Kuperberg invariant of the underlying closed 3-manifold, see Corollary \ref{corollary: multivariable link polynomial}.


\medskip

In our second main theorem we show that all the above procedures generalize existing ones in Reidemeister torsion theory. Indeed, let $\La(V)$ be the exterior algebra of an $n$-dimensional vector space $V$ over a field $\kk$. This is an $\N$-graded involutory Hopf superalgebra with $\Aut(\La(V))\cong GL(V)$ and the distinguished group-like of the associated semidirect product is $r_{\La(V)}=\det:GL(V)\to\kk^{\t}$.

\begin{Theorem}
\label{Theorem: I at exterior is torsion}
For an arbitrary $(M,\c)$, $\rho:\pi_1(M)\to GL(V)$, $\ss\in\Spinc(M,\c)$ and $\o$ as above we have
\begin{align*}
I^{\rho}_{\LaV}(M,\c,\ss,\o)=\ovtau^{\rho^{-T}}(M,\c,\ss,\o)
\end{align*}
where $\ovtau$ denotes the refinement of the twisted Reidemeister torsion $\tau^{\rho^{-T}}(M,R_-(\c))$ of Subsection \ref{subs: normalizing tau via Spinc} \footnote{The right hand side is computed at the inverse-transpose because we follow the conventions of \cite{FV:survey} for the torsion.}. Similarly, we have
\begin{align*}
I^{\rhoc}_{\LaV}(M,\c,\ss,\o)=\s(\ovtau^{\rho^{-T}\ot h}(M,\c,\ss,\o))\in \kk[H_1(M)]
\end{align*}
where $\s:\kk[H_1(M)]\to\kk[H_1(M)]$ is the $\kk$-linear map characterized by $\s(f)=f^{-1}, f\in H_1(M)$.
\end{Theorem}

This theorem generalizes to the non-abelian setting the result of \cite{LN:Kup} mentioned above. It turns out that there is a much simpler proof of the general case, relying on the universal property of the Hopf superalgebra $\La(V)$. It has to be noted that the dual of the Borel of $\Uqgl11$ at a root of unity of order $n$ considered in \cite{LN:Kup} is a semidirect product $\kk[\Z/n\Z]\ltimes \La(\kk)$ as in the present work (where $\Z/n\Z$ acts over $\kk$ by multiplication by an $n$-th root of unity). For link complements, the right hand side is equivalent to the twisted Alexander polynomials (see Corollary \ref{corollary: sutured torsion recovers TWISTED ALEX}), hence we get:

\begin{Corollary}
\label{Corollary: I at ext of link is twisted Alex poly}
Let $L$ be an ordered, oriented $m$-component link in an homology sphere $Y$ and $\rho:\pi_1(M_L)\to GL(V)$ a homomorphism, where $M_L=Y\sm L$. Then the twisted Kuperberg polynomial of $L$ at an exterior algebra is equivalent to the twisted Alexander polynomial: $$\s(I_{\La(V)}^{\rho^{-T}\ot h}(M_L,\c_L,\ss,\o))\dot{=}\frac{\prod_{i=1}^m\det(t_i\rho(a^*_i)-I_n)}{\De_{L,0}^{\rho}}\cdot \De_L^{\rho}(t_1,\dots,t_m)\in\kk[t_1^{\pm 1},\dots,t_m^{\pm 1}]$$
where $\De_{L,0}^{\rho}$ is the 0-th twisted Alexander polynomial of $L$ (which is always non-zero). Here $\dot{=}$ stands for equality up to multiplication by $\pm \det(\rho(g))\cdot t_1^{n_1}\cdots t_m^{n_m}$ for some $g\in\pi_1(M_L)$ and $n_i\in\Z$, and $\s$ is defined as above.
\end{Corollary}

\subsection*{Future directions} As mentioned above, Kuperberg invariants of closed 3-manifolds can be defined out of an arbitrary finite dimensional Hopf algebra, but without the involutory condition, the 3-manifolds need to be framed. Finding an involutory Hopf algebra for which our extension of Kuperberg invariants gives something different than Reidemeister torsion seems unlikely (see Subsection \ref{sect: going beyond R torsion} for concrete reasons). Thus, our results should be generalized to non-involutory Hopf algebras. 

\medskip

On the other hand, one knows from the works of Kirk-Livingston \cite{KL:mutation} and Friedl-Vidussi \cite{FV:Thurston, FV:fibered} that twisted Alexander polynomials are an extremely powerful invariant and contain deep topological information. 
It would be interesting to see whether some of these results extend to our twisted Kuperberg polynomials for other ($\N$-graded) Hopf algebras. Of course, for this to capture topological information beyond torsion one should use non-involutory Hopf algebras. More generally, it would be interesting to study possible topological applications of quantum invariants obtained through group-graded objects (Hopf group-coalgebras, modular $G$-categories, etc), and ultimately, of Turaev's homotopy quantum field theories.


\subsection*[Structure of the paper]{Structure of the paper}

Sections \ref{section: Some Hopf algebra theory} and \ref{section: sutured manifolds} consist of background material on Hopf (super)algebras, sutured manifolds and (twisted) Reidemeister torsion.  In Section \ref{sect: the invariant} we give the Fox calculus formula for $I_H^{\rho}(M,\c,\ss,\o)$ and study some of its properties, notably, that it extends to $\kk[H_1(M)]$ for $\N$-graded Hopf superalgebras and that it recovers our refinement of torsion when $H$ is an exterior algebra (Theorem \ref{Theorem: I at exterior is torsion}). This section is independent of our previous work \cite{LN:Kup}. In Section \ref{section: Fox calculus and semidirect products} we prove Theorem \ref{Theorem: I for sd product is Fox calculus} and relate our construction to the Hopf group-algebra approach of \cite{Virelizier:flat}. Finally, we devote to an Appendix some (rather trivial) technical details concerning the case of non-abelian $\rho$, which was not considered in \cite{LN:Kup}. 

\subsection*[Acknowledgments]{Acknowledgments}

I would like to thank my PhD supervisor, Christian Blanchet, for all his support and suggestions during the course of my PhD. I would also like to thank Anna Beliakova, Andr\'es Fontalvo Orozco, Krzysztof Putyra and Alexis Virelizier for many interesting conversations. Finally, I would like to thank the referee for valuable comments. This project has received funding from the European Union's Horizon 2020 Research and Innovation Programme under the Marie Sklodowska-Curie Grant Agreement No. 665850.

\section{Hopf superalgebras}
\label{section: Some Hopf algebra theory}

In this section, we recall the notions of the theory of Hopf (super)algebras we need. For more details, see \cite{Radford:BOOK}. In what follows we assume all vector spaces are defined over a field $\kk$.

\subsection{Basic notions and notation}\label{subs: basic notions} 

By a super-vector space, we mean a vector space $V$ endowed with a direct sum decomposition $V=V_0\oplus V_1$. A vector $v\in V$ is said to be homogeneous if $v\in V_0$ or $v\in V_1$. If $v\in V_i$, we say that $v$ has degree $i$, denoted $|v|=i$ (mod 2). If $V,W$ are super-vector spaces, then $V\ot W$ is a super-vector space with
\begin{align*}
(V\ot W)_i\eq\bigoplus_j V_j\ot W_{i-j}
\end{align*}
where the indices are taken mod 2. The category of super-vector spaces and their degree zero linear maps forms a symmetric monoidal category with symmetry map
\begin{align*}
\tau_{V,W}:V\ot W&\to W\ot V\\
v\ot w&\mapsto (-1)^{|v||w|}w\ot v.
\end{align*}

\medskip

We will use tensor network notation for the morphisms of this category (see \cite{Kup1, Kup2}), that is, we denote the tensor factors of the domain of a linear map as incoming arrows and those of the target as outcoming arrows with the convention that the left to right direction in a tensor $V_1\ot V_2\ot\dots \ot V_n$ corresponds to the top to bottom direction in tensor network notation. Moreover, since the (super)vector spaces will usually be clear from the context, we drop them from the notation. For instance, a tensor $T:V\ot V\to W\ot W\ot W$ is denoted by
\begin{figure}[H]
\centering
\begin{pspicture}(0,-0.30805147)(1.7040216,0.30805147)
\rput[bl](0.6040216,-0.1){$T$}
\psline[linecolor=black, linewidth=0.018, arrowsize=0.05291667cm 2.0,arrowlength=0.8,arrowinset=0.2]{->}(0.0040216064,0.3)(0.40402162,0.1)
\psline[linecolor=black, linewidth=0.018, arrowsize=0.05291667cm 2.0,arrowlength=0.8,arrowinset=0.2]{->}(1.0040216,0.1)(1.4040216,0.3)
\psline[linecolor=black, linewidth=0.018, arrowsize=0.05291667cm 2.0,arrowlength=0.8,arrowinset=0.2]{->}(1.0040216,-0.1)(1.4040216,-0.3)
\psline[linecolor=black, linewidth=0.018, arrowsize=0.05291667cm 2.0,arrowlength=0.8,arrowinset=0.2]{->}(0.0040216064,-0.3)(0.40402162,-0.1)
\psline[linecolor=black, linewidth=0.018, arrowsize=0.05291667cm 2.0,arrowlength=0.8,arrowinset=0.2]{->}(1.0040216,0.0)(1.4040216,0.0)
\end{pspicture}
\end{figure}

\noindent where the rightmost arrows correspond, from top to bottom, to the tensor factors of $W\ot W\ot W$ read from left to right (and similarly for the leftmost arrows). A vector $v\in V$ , a covector $v^*\in V^*$ and the identity map $\id_V:V\to V$ are denoted by
\begin{figure}[H]
\centering
\begin{pspicture}(0,-0.155)(4.07,0.155)
\psline[linecolor=black, linewidth=0.018, arrowsize=0.05291667cm 2.0,arrowlength=0.8,arrowinset=0.2]{->}(0.32,-0.005)(0.72,-0.005)
\psline[linecolor=black, linewidth=0.018, arrowsize=0.05291667cm 2.0,arrowlength=0.8,arrowinset=0.2]{->}(1.62,-0.005)(2.02,-0.005)
\rput[bl](0.0,-0.075){$v$}
\rput[bl](2.15,-0.085){$v^*$}
\psline[linecolor=black, linewidth=0.018, arrowsize=0.05291667cm 2.0,arrowlength=0.8,arrowinset=0.2]{->}(3.42,-0.005)(3.82,-0.005)
\rput[bl](0.87,-0.155){,}
\rput[bl](2.57,-0.155){,}
\rput[bl](4.02,-0.105){.}
\end{pspicture}
\end{figure}
\noindent Composition of tensors is denoted by joining the corresponding outcoming/incoming arrows, the tensor product is denoted by stacking one figure over another and the symmetry $\tau_{V,W}$ of the category is denoted by a crossing pair of arrows. For instance, if $T_1:V\to V\ot V$ and 
$T_2:V\ot V\to V$ then $(T_2\ot\id_V)\circ (\id_V\ot T_1)$, $T_1\ot T_2$ and $\tau_{V,V}\circ T_1$ are respectively denoted by
\begin{figure}[H]
\centering
\begin{pspicture}(0,-0.8500985)(8.85,0.8500985)
\rput[bl](0.6,-0.44204703){$T_1$}
\psline[linecolor=black, linewidth=0.018, arrowsize=0.05291667cm 2.0,arrowlength=0.8,arrowinset=0.2]{->}(0.0,-0.24204704)(0.4,-0.24204704)
\psline[linecolor=black, linewidth=0.018, arrowsize=0.05291667cm 2.0,arrowlength=0.8,arrowinset=0.2]{->}(1.1,-0.14204705)(1.5,0.057952955)
\psline[linecolor=black, linewidth=0.018, arrowsize=0.05291667cm 2.0,arrowlength=0.8,arrowinset=0.2]{->}(1.1,-0.34204704)(1.5,-0.542047)
\psline[linecolor=black, linewidth=0.018, arrowsize=0.05291667cm 2.0,arrowlength=0.8,arrowinset=0.2]{->}(1.1,0.45795295)(1.5,0.25795296)
\rput[bl](1.7,-0.042047042){$T_2$}
\psline[linecolor=black, linewidth=0.018, arrowsize=0.05291667cm 2.0,arrowlength=0.8,arrowinset=0.2]{->}(2.2,0.15795296)(2.6,0.15795296)
\rput[bl](4.6,0.45795295){$T_1$}
\psline[linecolor=black, linewidth=0.018, arrowsize=0.05291667cm 2.0,arrowlength=0.8,arrowinset=0.2]{->}(4.0,0.55795294)(4.4,0.55795294)
\psline[linecolor=black, linewidth=0.018, arrowsize=0.05291667cm 2.0,arrowlength=0.8,arrowinset=0.2]{->}(5.1,0.65795296)(5.5,0.85795295)
\psline[linecolor=black, linewidth=0.018, arrowsize=0.05291667cm 2.0,arrowlength=0.8,arrowinset=0.2]{->}(5.1,0.45795295)(5.5,0.25795296)
\psline[linecolor=black, linewidth=0.018, arrowsize=0.05291667cm 2.0,arrowlength=0.8,arrowinset=0.2]{->}(4.0,-0.84204704)(4.4,-0.64204705)
\psline[linecolor=black, linewidth=0.018, arrowsize=0.05291667cm 2.0,arrowlength=0.8,arrowinset=0.2]{->}(4.0,-0.24204704)(4.4,-0.44204703)
\rput[bl](4.6,-0.7420471){$T_2$}
\psline[linecolor=black, linewidth=0.018, arrowsize=0.05291667cm 2.0,arrowlength=0.8,arrowinset=0.2]{->}(5.1,-0.542047)(5.5,-0.542047)
\rput[bl](7.5,-0.26204705){$T_1$}
\psline[linecolor=black, linewidth=0.018, arrowsize=0.05291667cm 2.0,arrowlength=0.8,arrowinset=0.2]{->}(6.9,-0.14204705)(7.3,-0.14204705)
\psbezier[linecolor=black, linewidth=0.018, arrowsize=0.05291667cm 2.0,arrowlength=0.8,arrowinset=0.2]{->}(8.0,-0.042047042)(8.3,-0.042047042)(8.4,-0.14204705)(8.5,-0.44204704284667967)
\psbezier[linecolor=black, linewidth=0.018, arrowsize=0.05291667cm 2.0,arrowlength=0.8,arrowinset=0.2]{->}(8.0,-0.24204704)(8.3,-0.24204704)(8.4,-0.14204705)(8.5,0.1579529571533203)
\rput[bl](5.7,-0.14204705){,}
\rput[bl](8.8,-0.19204704){.}
\rput[bl](2.8,-0.14204705){,}
\end{pspicture}
\end{figure}

\medskip


\subsection{Hopf superalgebras}\label{subs: Hopf superalgebras} A {\em Hopf superalgebra} is a super-vector space $H$ endowed with degree zero tensors 
\begin{figure}[H]
\centering
\begin{pspicture}(0,-0.31162995)(9.832368,0.31162995)
\psline[linecolor=black, linewidth=0.018, arrowsize=0.05291667cm 2.0,arrowlength=0.8,arrowinset=0.2]{->}(4.3058095,0.0)(4.705809,0.0)
\rput[bl](4.8658094,-0.1){$\Delta$}
\psline[linecolor=black, linewidth=0.018, arrowsize=0.05291667cm 2.0,arrowlength=0.8,arrowinset=0.2]{->}(5.3058095,0.1)(5.705809,0.3)
\psline[linecolor=black, linewidth=0.018, arrowsize=0.05291667cm 2.0,arrowlength=0.8,arrowinset=0.2]{->}(5.3058095,-0.1)(5.705809,-0.3)
\psline[linecolor=black, linewidth=0.018, arrowsize=0.05291667cm 2.0,arrowlength=0.8,arrowinset=0.2]{->}(0.005809326,-0.3)(0.4058093,-0.1)
\psline[linecolor=black, linewidth=0.018, arrowsize=0.05291667cm 2.0,arrowlength=0.8,arrowinset=0.2]{->}(0.005809326,0.3)(0.4058093,0.1)
\rput[bl](0.5658093,-0.05){$m$}
\psline[linecolor=black, linewidth=0.018, arrowsize=0.05291667cm 2.0,arrowlength=0.8,arrowinset=0.2]{->}(1.0058093,0.0)(1.4058093,0.0)
\rput[bl](2.4458094,-0.1){$1$}
\psline[linecolor=black, linewidth=0.018, arrowsize=0.05291667cm 2.0,arrowlength=0.8,arrowinset=0.2]{->}(2.8058093,0.0)(3.2058094,0.0)
\psline[linecolor=black, linewidth=0.018, arrowsize=0.05291667cm 2.0,arrowlength=0.8,arrowinset=0.2]{->}(6.8058095,0.0)(7.205809,0.0)
\rput[bl](7.4058094,-0.06){$\epsilon$}
\rput[bl](9.145809,-0.1){$S$}
\psline[linecolor=black, linewidth=0.018, arrowsize=0.05291667cm 2.0,arrowlength=0.8,arrowinset=0.2]{->}(9.50581,0.0)(9.905809,0.0)
\psline[linecolor=black, linewidth=0.018, arrowsize=0.05291667cm 2.0,arrowlength=0.8,arrowinset=0.2]{->}(8.605809,0.0)(9.00581,0.0)
\rput[bl](1.5258093,-0.1){,}
\rput[bl](3.3258092,-0.1){,}
\rput[bl](5.725809,-0.1){,}
\rput[bl](7.625809,-0.1){,}
\end{pspicture}
\end{figure}
\noindent where each arrow corresponds to $H$ (that is, $m:H\ot H\to H, 1\in H$, etc). These tensors satisfy the usual Hopf algebra axioms, except that the algebra property for the coproduct $\De$ involves the symmetry $\tau_{H,H}$ of super-vector spaces (denoted by a crossing pair of arrows):
\begin{figure}[H]
\centering
\begin{pspicture}(0,-0.42)(5.9058094,0.42)
\psline[linecolor=black, linewidth=0.018, arrowsize=0.05291667cm 2.0,arrowlength=0.8,arrowinset=0.2]{->}(0.005809326,-0.32)(0.4058093,-0.12)
\psline[linecolor=black, linewidth=0.018, arrowsize=0.05291667cm 2.0,arrowlength=0.8,arrowinset=0.2]{->}(0.005809326,0.28)(0.4058093,0.08)
\rput[bl](0.5558093,-0.1){$m$}
\psline[linecolor=black, linewidth=0.018, arrowsize=0.05291667cm 2.0,arrowlength=0.8,arrowinset=0.2]{->}(1.0058093,-0.02)(1.4058093,-0.02)
\rput[bl](1.5658094,-0.12){$\Delta$}
\psline[linecolor=black, linewidth=0.018, arrowsize=0.05291667cm 2.0,arrowlength=0.8,arrowinset=0.2]{->}(2.0058093,0.08)(2.4058094,0.28)
\psline[linecolor=black, linewidth=0.018, arrowsize=0.05291667cm 2.0,arrowlength=0.8,arrowinset=0.2]{->}(2.0058093,-0.12)(2.4058094,-0.32)
\rput[bl](2.6858094,-0.08){=}
\psline[linecolor=black, linewidth=0.018, arrowsize=0.05291667cm 2.0,arrowlength=0.8,arrowinset=0.2]{->}(3.2058094,0.28)(3.6058092,0.28)
\rput[bl](3.7658093,0.18){$\Delta$}
\psline[linecolor=black, linewidth=0.018, arrowsize=0.05291667cm 2.0,arrowlength=0.8,arrowinset=0.2]{->}(4.205809,0.28)(4.705809,0.28)
\rput[bl](4.855809,0.2){$m$}
\psline[linecolor=black, linewidth=0.018, arrowsize=0.05291667cm 2.0,arrowlength=0.8,arrowinset=0.2]{->}(5.3058095,0.28)(5.705809,0.28)
\psline[linecolor=black, linewidth=0.018, arrowsize=0.05291667cm 2.0,arrowlength=0.8,arrowinset=0.2]{->}(3.2058094,-0.32)(3.6058092,-0.32)
\rput[bl](3.7658093,-0.42){$\Delta$}
\psline[linecolor=black, linewidth=0.018, arrowsize=0.05291667cm 2.0,arrowlength=0.8,arrowinset=0.2]{->}(4.205809,-0.32)(4.705809,-0.32)
\rput[bl](4.855809,-0.4){$m$}
\psline[linecolor=black, linewidth=0.018, arrowsize=0.05291667cm 2.0,arrowlength=0.8,arrowinset=0.2]{->}(5.3058095,-0.32)(5.705809,-0.32)
\psline[linecolor=black, linewidth=0.018, arrowsize=0.05291667cm 2.0,arrowlength=0.8,arrowinset=0.2]{->}(4.205809,-0.22)(4.705809,0.18)
\psline[linecolor=black, linewidth=0.018, arrowsize=0.05291667cm 2.0,arrowlength=0.8,arrowinset=0.2]{->}(4.205809,0.18)(4.705809,-0.22)
\rput[bl](5.855809,-0.42){.}
\end{pspicture}
\end{figure}
\noindent Though we mostly use tensor network notation, in some places we also use Sweedler's notation for the coproduct, that is, we write $\De(h)=h_{(1)}\ot h_{(2)}$ (omitting the sum sign) for any $h\in H$. Any (ungraded) Hopf algebra can be seen as a Hopf superalgebra concentrated in degree zero. In what follows we reserve the term {\em Hopf algebra} exclusively for the ungraded case. If $H$ is a Hopf superalgebra, then we define $H^{op}=(H,m^{op},1,\De,\e,S^{-1})$ and $H^{cop}=(H,m,1,\De^{op},\e,S^{-1})$ where $m^{op}=m\circ\tau_{H,H}$ and $\De^{op}=\tau_{H,H}\circ\De$. The dual $H^*$ is also a Hopf superalgebra in the usual way. We denote by $\Aut(H)$ the group of Hopf superalgebra automorphisms of $H$.


\medskip

A Hopf superalgebra is said to be {\em involutory} if $S\circ S=\id_H$. It is {\em commutative} if $m=m^{op}$ and it is {\em cocommutative} if $\De=\De^{op}$. A commutative or cocommutative Hopf superalgebra is always involutory, see e.g. \cite[Corollary 7.1.11]{Radford:BOOK}. A Hopf superalgebra $H$ is $\N$-graded if it has a direct sum decomposition $H=\oplus_{n\in\N}H_n$ such that
\begin{align*}
m(H_i\ot H_j)&\sb H_{i+j}, & \De(H_n)&\sb \sum_{i+j=n}H_i\ot H_j, & S(H_n)\sb H_n,
\end{align*}
for each $i,j,n\in\N$ and such that the $\N$-degree refines the mod 2 degree, that is, $H_i=\oplus_{n\geq 0}H_{2n+i}$ for each $i=0,1 \pmod 2$. To distinguish from mod 2 degree, we set $|x|_0\eq n\in\N$ for nonzero $x\in H_n$.


\def\Prim{\text{Prim}}

\begin{example}
Let $V$ be a finite dimensional vector space. The {\em exterior algebra} $\LaV$ on $V$ is the quotient of the tensor algebra $T(V)=\oplus_{n\geq 0}V^{\ot n}$ (where $V^{\ot 0}=\kk$) by the ideal generated by the elements of the form $v\ot w+w\ot v$ with $v,w\in V$. This becomes an $\N$-graded superalgebra by letting $V\sb \LaV$ be in degree one and it is a Hopf superalgebra if we set 
\begin{align*}
\De(v)&=1\ot v+v\ot 1, & \e(v)&=0, & S(v)=&-v
\end{align*}
for any $v\in V$ and extend $\De,\e$ (resp. $S$) by letting them be superalgebra homomorphisms (resp. antihomomorphism). This is a commutative, cocommutative (hence involutory) Hopf superalgebra. Since $\Prim(\La(V))=V$ (where $\Prim(H)=\{h\in H \ | \ \De(h)=h\ot 1+1\ot h\}$) it follows that the group of automorphisms of $\La(V)$ is isomorphic to $GL(V)$.
\end{example}

Given a subgroup $G\sb\Aut(H)$, the {\em semidirect-product} $\kk[G]\ltimes H$ is a Hopf superalgebra defined as follows. As a super-coalgebra, it is the tensor product $\kk[G]\ot H$, where $\kk[G]$ is the group-algebra of $G$ (this is a Hopf algebra with $\De(\a)=\a\ot\a, \e(\a)=1$ and $S(\a)=\a^{-1}$ for each $\a\in G$). Its product is defined by
\begin{align*}
(\a_1\ot h_1)\cdot (\a_2\ot h_2)\eq\a_1 \a_2\ot \a_2^{-1}(h_1)h_2,
\end{align*}
\noindent and the antipode by $S(\a_1\ot h_1)\eq\a_1^{-1}\ot \a_1^{-1}(S_H(h_1))$ for any $\a_1,\a_2\in G$ and $h_1,h_2\in H$. It has to be noted that for semisimple $H$ (i.e. semisimple as a $\kk$-algebra), $\Aut(H)$ is a finite group by \cite{Radford:automorphisms}, hence $\kk[\Aut(H)]\ltimes H$ is finite-dimensional. However, for non-semisimple $H$, $\kk[\Aut(H)]\ltimes H$ is infinite-dimensional in general (for instance if $H=\La(V)$).


\subsection{Integrals and cointegrals}\label{subs: integrals and cointegrals} Let $H=(H,m,1,\De,\e,S)$ be a finite dimensional Hopf superalgebra over a field $\kk$. A {\em right cointegral} in $H$ is an element $\coint_r\in H$ such that
\begin{figure}[H]
\centering
\begin{pspicture}(0,-0.3558149)(5.05,0.3558149)
\psline[linecolor=black, linewidth=0.018, arrowsize=0.05291667cm 2.0,arrowlength=0.8,arrowinset=0.2]{->}(0.5,-0.3441851)(0.9,-0.14418511)
\psline[linecolor=black, linewidth=0.018, arrowsize=0.05291667cm 2.0,arrowlength=0.8,arrowinset=0.2]{->}(0.5,0.2558149)(0.9,0.055814896)
\rput[bl](1.06,-0.09418511){$m$}
\psline[linecolor=black, linewidth=0.018, arrowsize=0.05291667cm 2.0,arrowlength=0.8,arrowinset=0.2]{->}(1.5,-0.044185106)(1.9,-0.044185106)
\rput[bl](2.2,-0.044185106){=}
\psline[linecolor=black, linewidth=0.018, arrowsize=0.05291667cm 2.0,arrowlength=0.8,arrowinset=0.2]{->}(2.8,-0.044185106)(3.2,-0.044185106)
\rput[bl](3.4,-0.104185104){$\epsilon$}
\psline[linecolor=black, linewidth=0.018, arrowsize=0.05291667cm 2.0,arrowlength=0.8,arrowinset=0.2]{->}(4.4,-0.044185106)(4.8,-0.044185106)
\rput[bl](5.0,-0.14418511){.}
\rput[bl](0.0,0.1558149){$\coint_r$}
\rput[bl](3.9,-0.14418511){$\coint_r$}
\end{pspicture}

\end{figure}
\noindent A left cointegral of $H$ is defined as a right cointegral of $H^{op}$. A {\em right integral} is an  element $\int_r\in H^*$ such that
\begin{figure}[H]
\centering
\begin{pspicture}(0,-0.34439287)(5.05,0.34439287)
\psline[linecolor=black, linewidth=0.018, arrowsize=0.05291667cm 2.0,arrowlength=0.8,arrowinset=0.2]{->}(0.0,-0.065607145)(0.4,-0.065607145)
\rput[bl](0.56,-0.16560715){$\Delta$}
\psline[linecolor=black, linewidth=0.018, arrowsize=0.05291667cm 2.0,arrowlength=0.8,arrowinset=0.2]{->}(1.0,0.034392852)(1.4,0.23439285)
\psline[linecolor=black, linewidth=0.018, arrowsize=0.05291667cm 2.0,arrowlength=0.8,arrowinset=0.2]{->}(1.0,-0.16560715)(1.4,-0.36560714)
\rput[bl](1.6,0.13439286){$\int_r$}
\rput[bl](2.2,-0.065607145){=}
\rput[bl](4.04,-0.16560715){$1$}
\psline[linecolor=black, linewidth=0.018, arrowsize=0.05291667cm 2.0,arrowlength=0.8,arrowinset=0.2]{->}(4.4,-0.065607145)(4.8,-0.065607145)
\psline[linecolor=black, linewidth=0.018, arrowsize=0.05291667cm 2.0,arrowlength=0.8,arrowinset=0.2]{->}(2.8,-0.065607145)(3.2,-0.065607145)
\rput[bl](5.0,-0.16560715){.}
\rput[bl](3.4,-0.16560715){$\int_r$}
\end{pspicture}
\end{figure}
\noindent Equivalently, a right integral over $H$ is a right cointegral of the dual Hopf superalgebra $H^*$.

\medskip

If $H$ is finite dimensional then there is a non-zero right cointegral and it is unique up to scalar \cite[Theorem 10.2.2]{Radford:BOOK}, the same holds for left or right integrals as well. 

\begin{remark}
\label{remark: unimodularity}
It is not always true that left (co)integrals are also right (co)integrals. If $H$ is an (ungraded) Hopf algebra, then one says that $H$ is {\em unimodular} if this is the case, but this definition is not completely appropriate in the super-case. Indeed, one can define unimodularity for finite tensor categories, e.g. the category of representations of a Hopf (super)algebra (see \cite[Definition 6.5.7]{EGNO:BOOK}). Then one shows that for Hopf algebras, unimodularity of $\Rep(H)$ is equivalent to the cointegral being two-sided. Nevertheless, for Hopf superalgebras, unimodularity of its category of representations is equivalent to the cointegral being two-sided and of mod 2 degree zero. For instance, an exterior algebra over an odd dimensional vector space is non-unimodular, even though cointegrals are two-sided.
\end{remark}

Now let $\coint_r$ be a non-zero right cointegral in $H$ and let $\a\in\Aut(H)$. Clearly, $\a(\coint_r)$ is a right cointegral so by uniqueness, there is a scalar $\la\in\kk^{\t}$ (depending on $\a$) such that $\a(\coint_r)=\la \coint_r.$ Similarly, if $\int_r$ is a right integral on $H$, uniqueness of integrals implies that $\int_r\circ\a=\la'\int_r$ for some $\la'\in\kk^{\t}$. Since $\int_r(\coint_r)\neq 0$ \cite[Theorem 10.2.2, b)]{Radford:BOOK} it follows that $\la'=\la$.

\begin{definition}
\label{def: rH}
For any $\a\in\Aut(H)$, we denote $\lH(\a)\eq\la$, where $\la$ is the scalar defined above. This is a group homomorphism $\lH:\Aut(H)\to \kk^{\t}$. 
\end{definition}

Note that a Hopf algebra is semisimple if and only if $\e(\coint_r)\neq 0$ by \cite[Theorem 10.3.2]{Radford:BOOK} and so $\lH\equiv 1$ in this case. Hence, the homomorphism $\lH$ is relevant only for non-semisimple $H$.


\begin{example}
\label{example: rH of exterior algebra}
Let $\LaV$ be the exterior algebra on a finite dimensional vector space $V$, recall that $\Aut(\La(V))=GL(V)$. If $X_1,\dots,X_n$ is a basis of $V$, then the (two-sided) cointegral of $\LaV$ is the product
$\coint=X_1\dots X_n$ so that $$\a(\coint)=\a(X_1)\dots \a(X_n)=\det(\a)X_1\dots X_n=\det(\a)\coint$$
for any $\a\in GL(V)$. Thus, $r_{\LaV}:GL(V)\to \kk^{\t}$ is just the determinant.
\end{example}

\def\HR{H_R}

Now suppose $H$ is $\N$-graded and $R$ is a commutative $\kk$-algebra with unit (we do not suppose $R$ is a domain). Consider the Hopf superalgebra $\HR\eq H\ot_{\kk}R$ over $R$. Given an element $\a\in \Aut(H)$ and $f\in R$ we define a $R$-linear Hopf endomorphism $\a\ot f$ of $\HR$ by 
\begin{align}
\label{eq: alpha tensor f}
\a\ot f(h\ot f')\eq \a(h)\ot (f^{|h|_0}\cdot f')
\end{align}
where $h\in H$ is homogeneous and $f'\in R$. This is an automorphism of $\HR$ if $f\in R^{\t}$. Now let $\coint, \int$ be respectively a right cointegral  and integral for $H$. Then $\coint_R\eq\coint\ot 1$ and $\int_R\eq \int\ot \id_R$ are respectively a cointegral and integral for $\HR$ respectively. If $\a\ot f$ is defined as above one has
\begin{align*}
\a\ot f(\coint\ot 1)&=\a(\coint)\ot f^{|\coint|_0}\\
&= \lH(\a)\coint\ot f^{|\coint|_0}\\
&=\lH(\a)f^{|\coint|_0} (\coint\ot 1).
\end{align*}
This implies that the homomorphism $r_{\HR}:\Aut(\HR)\to R^{\t}$ satisfies
\begin{equation}
\label{lemma: lHM homomorphism}
r_{\HR}(\a\ot f)\eq \lH(\a)\cdot f^{|\coint|_0}.
\end{equation}



\subsection{Hopf $G$-algebras}\label{subs: Hopf G-algebras} 

\def\ucoint{\underline{\coint}}
\def\bg{\boldsymbol{g}}

Let $G$ be a group. A {\em Hopf $G$-algebra} is a family of coalgebras $\uH=\{(\Ha,\Dea,\ea)\}_{\a\in\pia}$ indexed by $G$ endowed with coalgebra morphisms $$\maot:\HaoHat\to\Haot$$
for each $\a_1,\a_2\in G$, a unit $1\in H_{1_G}$ and maps $\Sa:\Ha\to \Ham$ for each $\a\in G$ satisfying graded versions of the associativity, unitality and antipode axioms (see \cite{Virelizier:Hopfgroup} or \cite{Turaev:homotopy, Turaev:BOOK-HQFT} for more details in the dual setting). We employ similar tensor network notation for Hopf $G$-algebras, the $G$-labels only appearing as subscripts of the structure maps. Note that $H_{1_G}$ is a Hopf algebra in the usual sense. We say that $\uH$ is {\em involutory} if $\Sam\Sa=\id_{\Ha}$ for each $\a\in G$. We say that $\uH$ is of {\em finite type} if each $\Ha$ is finite dimensional. These notions extend in an obvious way to super-vector spaces.
\medskip


Let $\uHH$ be a finite type Hopf $G$-algebra. Since the dual of a Hopf $G$-algebra is a Hopf $G$-coalgebra, the existence and uniqueness theorems of integrals of \cite{Virelizier:Hopfgroup} have analogous statements in the $G$-algebra case. Thus there exists a unique {\em right cointegral} of $\uH$ (up to scalar), that is, a family $\ucoint=\{\ca\}_{\a\in\pia}$, where $\ca\in\Ha$ for each $\a\in G$, satisfying 
\begin{figure}[H]
\centering
\begin{pspicture}(0,-0.42081496)(6.8,0.42081496)
\psline[linecolor=black, linewidth=0.018, arrowsize=0.05291667cm 2.0,arrowlength=0.8,arrowinset=0.2]{->}(0.55,-0.40918502)(0.95,-0.20918503)
\psline[linecolor=black, linewidth=0.018, arrowsize=0.05291667cm 2.0,arrowlength=0.8,arrowinset=0.2]{->}(0.55,0.19081497)(0.95,-0.009185028)
\rput[bl](1.11,-0.23918504){$m_{\alpha_1\alpha_2}$}
\psline[linecolor=black, linewidth=0.018, arrowsize=0.05291667cm 2.0,arrowlength=0.8,arrowinset=0.2]{->}(2.25,-0.109185025)(2.65,-0.109185025)
\rput[bl](-0.07,0.19081497){$\coint_{\alpha_1}$}
\rput[bl](3.0,-0.109185025){=}
\psline[linecolor=black, linewidth=0.018, arrowsize=0.05291667cm 2.0,arrowlength=0.8,arrowinset=0.2]{->}(3.65,-0.109185025)(4.05,-0.109185025)
\psline[linecolor=black, linewidth=0.018, arrowsize=0.05291667cm 2.0,arrowlength=0.8,arrowinset=0.2]{->}(6.15,-0.109185025)(6.55,-0.109185025)
\rput[bl](4.2,-0.20918503){$\epsilon_{\alpha_2}$}
\rput[bl](5.2,-0.20918503){$\coint_{\alpha_1\alpha_2}$}
\rput[bl](6.75,-0.109185025){.}
\end{pspicture}

\end{figure}

\noindent Moreover, there exists a unique family $\bg^*=\{\gda\}_{\a\in G}$ where $\gda\in\Ha^*$ satisfying
\begin{figure}[H]
\centering
\begin{pspicture}(0,-0.40581498)(6.8,0.40581498)
\psline[linecolor=black, linewidth=0.018, arrowsize=0.05291667cm 2.0,arrowlength=0.8,arrowinset=0.2]{->}(0.55,-0.20581497)(0.95,-0.005814972)
\psline[linecolor=black, linewidth=0.018, arrowsize=0.05291667cm 2.0,arrowlength=0.8,arrowinset=0.2]{->}(0.55,0.39418504)(0.95,0.19418503)
\rput[bl](1.11,-0.03581497){$m_{\alpha_1\alpha_2}$}
\psline[linecolor=black, linewidth=0.018, arrowsize=0.05291667cm 2.0,arrowlength=0.8,arrowinset=0.2]{->}(2.25,0.094185024)(2.65,0.094185024)
\rput[bl](-0.03,-0.40581498){$\coint_{\alpha_2}$}
\rput[bl](3.0,0.094185024){=}
\psline[linecolor=black, linewidth=0.018, arrowsize=0.05291667cm 2.0,arrowlength=0.8,arrowinset=0.2]{->}(3.65,0.094185024)(4.05,0.094185024)
\psline[linecolor=black, linewidth=0.018, arrowsize=0.05291667cm 2.0,arrowlength=0.8,arrowinset=0.2]{->}(6.15,0.094185024)(6.55,0.094185024)
\rput[bl](4.2,-0.055814972){$g^*_{\alpha_1}$}
\rput[bl](5.2,-0.005814972){$\coint_{\alpha_1\alpha_2}$}
\end{pspicture}

\end{figure} 
 
\noindent and $g^*_{\a_1\a_2}\circ m_{\a_1\a_2}=g^*_{\a_1}\ot g^*_{\a_2}$ for each $\a_1,\a_2\in G$. We call $\bg^*$ the {\em comodulus} of $\uH$. 

\medskip

\begin{example}
\label{example: Hopf group-algebra from semidirect product}
Let $H$ be a finite dimensional Hopf (super)algebra and let $G=\Aut(H)$. Consider the semidirect product Hopf (super)algebra $\kk[G]\ltimes H$. Let $$\Ha\eq \{h\cdot\a \ | \ h\in H\}\sb \kk[G]\ltimes H$$
for each $\a\in G$. Then $\uH\eq\{\Ha\}_{\a\in G}$ is a Hopf $G$-(super)algebra with the structure morphisms induced from the semidirect product. A right cointegral is given by $\ca\eq \coint\cdot\a\in\Ha$ where $\coint$ is a right cointegral of $H$. The comodulus of $\uH$ is given by $$\gda(h\cdot\a)\eq\lH(\a)g^*(h)$$
for $h\in H, \a\in G$, where $g^*\in H^*$ is the comodulus of $H$ and $\lH$ is the homomorphism of Definition \ref{def: rH}. This example is dual to the Hopf group-coalgebra of \cite[Subsection 1.2]{Turaev:BOOK-HQFT}.
\end{example}

\section{Sutured manifolds and Reidemeister torsion}
\label{section: sutured manifolds}

In this section we recall a few concepts from the theory of sutured manifolds \cite{Gabai:foliations, Juhasz:holomorphic, JTZ:naturality} and twisted Reidemeister torsion \cite{Turaev:BOOK1, FV:survey} as well as some considerations from our previous work \cite{LN:Kup}. In what follows, all 3-manifolds will be assumed to be compact and oriented. Homology will always be taken with integral coefficients, unless specified otherwise.

\subsection[Sutured manifolds]{Sutured manifolds}
\label{subs: sutured manifolds} A {\em sutured manifold} is a (compact, oriented) 3-manifold-with-boundary $M$ endowed with a collection $\c$ of pairwise disjoint annuli in $\p M$ subject to the following properties:
\begin{enumerate}
\item Each annuli is the closed tubular neighborhood of an oriented simple closed curve called a {\em suture}. The collection of sutures is denoted by $s(\c)$.
\item The surface $R\eq\p M\sm \inte(\c)$ is oriented and each (oriented) component of $\p R$ is oriented-parallel to some suture.
\end{enumerate}
We denote by $R_+(\c)$ (resp. $R_-(\c)$) the union of the components of $R$ whose orientation coincide (resp. is opposite) with the induced orientation on $\p M$. We say that $(M,\c)$ is {\em balanced} if $M$ has no closed components, each component of $\p M$ has at least one suture and $\chi(R_-(\c))=\chi(R_+(\c))$.

\begin{example}
\label{example: sutured from closed}
Let $Y$ be a connected closed oriented 3-manifold. Let $B^3\sb Y$ be an embedded open 3-ball. Then $M_0=Y\sm B^3$ is a balanced sutured 3-manifold if we let $\c_0$ be a single annulus in $\p M_0$. The subsurfaces $R_{\pm}(\c_0)$ are both disks in this case.
\end{example}

\begin{example} 
\label{example: sutured from link complement}
Let $L$ be a link in a closed oriented 3-manifold $Y$. Let $N(L)$ be an open tubular neighborhood of $L$ in $Y$. Then $M_L=Y\sm N(L)$ is a balanced sutured 3-manifold if we let $\c=\c_L$ consists of a pair of oppositely oriented meridians, one pair for each component of $\p M_L$. For each component $T\sb \p M_L$, $T\sm (T\cap \c_L)$ consists of two annuli, one in $R_-(\c_L)$ and the other in $R_+(\c_L)$.
\end{example}

\subsection[Extended Heegaard diagrams]{Extended Heegaard diagrams} Let $(M,\c)$ be a sutured 3-manifold. An {\em (embedded) sutured Heegaard diagram} is a tuple $\HH=(\S,\aa,\bb)$ consisting of the following data:
\begin{enumerate}
\item A compact oriented surface-with-boundary $\S$ embedded in $M$ with $\p\S=s(\c)$ as oriented 1-manifolds.
\item A set $\aa$ (resp. $\bb$) of pairwise disjoint simple closed curves in $\S$ bounding disks to the negative (resp. positive) side of $\S$.
\end{enumerate}
We will usually identify the set $\aa$ with the 1-submanifold $\cup_{\a\in\aa}\a\sb\S$, and similarly for $\bb$. We assume that $\aa$ and $\bb$ are transverse submanifolds of $\S$. We require that the surface $\S[\aa]$(resp. $\S[\bb]$) obtained by compressing $\S$ along the disks corresponding to $\aa$ (resp. $\bb$) results in a surface isotopic to $R_-(\c)$ (resp. $R_+(\c)$) relative to $\c$. In other words, $\S$ splits $M$ into two handlebodies $\Ua$ and $\Ub$, where $\Ua$ is obtained by attaching one handles to $R_-\t [-1,1]$ along $R_-\t \{1\}$ with belt circles the curves in $\aa$ while $\Ub$ is obtained by attaching one handles to $R_+\t [-1,1]$ along $R_+\t \{-1\}$ with belt circles the curves in $\bb$. We say that $\HH$ is {\em balanced} if $|\aa|=|\bb|$ and every component of $\S\sm \aa$ and $\S\sm\bb$ contains a component of $\p\S$.  From now on, we will refer to (balanced) embedded sutured Heegaard diagrams only as (balanced) Heegaard diagrams.
\medskip

An {\em extended Heegaard diagram} of $(M,\c)$ is a Heegaard diagram $\HD$ endowed with a set $\bolda$ of pairwise disjoint properly embedded arcs in $\S\sm\aa$ that forms a cut system of the surface $\S[\aa]$, that is, for any component $R'$ of $\S[\aa]$, $R'\sm R'\cap N(\bolda)$ is homeomorphic to a disk, where $N(\bolda)$ is an open tubular neighborhood of $\bolda$ in $\S[\aa]$. We denote $\aaa=\aa\cup\bolda$ and if $|\aa|=d,|\bolda|=l$, then we will denote $\aa=\{\a_1,\dots,\a_d\}$ and $\bolda=\{\a_{d+1},\dots,\a_{d+l}\}$.
\medskip

As is well-known, any sutured manifold $(M,\c)$ has a Heegaard diagram and $(M,\c)$ is balanced if and only if any (and hence all) of its Heegaard diagrams is balanced \cite{Juhasz:holomorphic}. The Reidemeister-Singer theorem extends to sutured manifolds: any two Heegaard diagrams of a given sutured manifold are related by Heegaard moves \cite{Juhasz:holomorphic}, see also \cite[Proposition 2.36]{JTZ:naturality} for a precise version in the embedded case. A Heegaard diagram can always be extended as above and any two extended Heegaard diagrams of $(M,\c)$ are related by {\em extended Heegaard moves} \cite{LN:Kup}: usual Heegaard moves of $\HD$ with the condition that the $\a$'s are isotoped or handleslided in the complement of the arcs in $\bolda$, isotopies of arcs, and sliding an arc over a curve in $\aa$ or over another arc. Note that the definition of extended Heegaard diagram of \cite{LN:Kup} includes also a cut system of $\S[\bb]$, we do not require this in the present paper.

\subsection{A presentation of $\pi_1$}
\label{subs: HDs extended}

\medskip
\def\arc{a}

Suppose now that $R_-(\c)$ is connected. We describe a presentation of $\pi_1(M)$ out of an extended Heegaard diagram $\HH=\HDD$ of $(M,\c)$ (see also \cite[Section 4.1]{FJR11}). We assume $M$ has a basepoint $p\in s(\c)$ (note that this is fixed when doing Heegaard moves since our diagrams have fixed boundary $\p\S=s(\c)$). Let $\aaa=\aa\cup\bolda$ where $\aa=\{\a_1,\dots,\a_d\}$ are the closed curves and $\bolda=\{\a_{d+1},\dots,\a_{d+l}\}$ are the arcs. Write $M=\Ua\cup\Ub$ where $\Ua,\Ub$ are as above and think of $\Ua$ as constructed from $R_-\t [-1,1]$ by attaching 3-dimensional 1-handles with belt circles the closed $\a$ curves. The cocores of these 1-handles are disks $D_1,\dots,D_d$ with $\p D_i=\a_i$ for each $i=1,\dots,d$. We also have disks $D_{i}=\a_{i}\t [-1,1]\sb R_-\t [-1,1]$ for each $i=d+1,\dots,d+l$. Since we supposed $R_-(\c)$ is connected, the complement of a sufficiently small neighborhood of $D_1\cup\dots\cup D_{d+l}$ in $\Ua$ is homeomorphic to a single 3-ball. This implies that for each $i=1,\dots,d+l$ there is a unique unoriented loop $\a^*_i$ based at $p$ contained in $\Ua$ that intersects $D_i$ in a single point and is disjoint from $D_j$ for $j\neq i$. If each curve and arc in $\aaa$ is oriented, then we orient $\a^*_i$ so that $\a_i\cdot\a^*_i=+1$ in $\S$ (when $\a^*_i$ is homotoped to lie in $\S$), see Figure \ref{fig: dual curves in pi1} (right). From now on, we consider $\a^*_i$ as an element of $\pi_1(M,p)$.

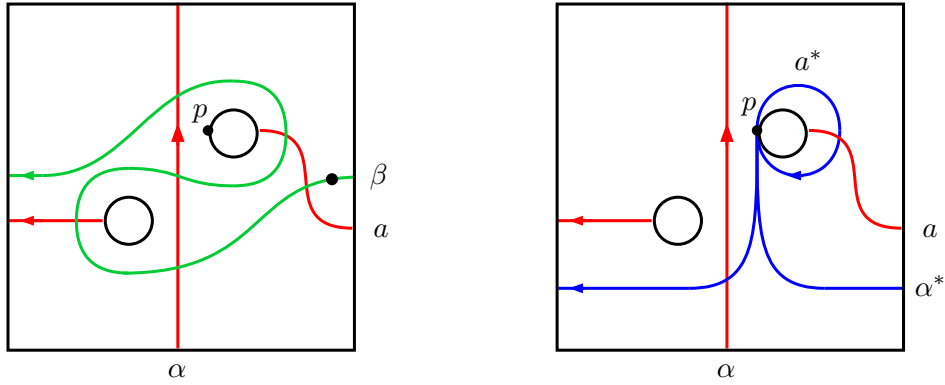
\begin{figure}[t]
\centering
\begin{pspicture}(0,-2.4892063)(12.407761,2.4892063)
\definecolor{colour1}{rgb}{0.0,0.8,0.2}
\psbezier[linecolor=blue, linewidth=0.04](9.97776,0.7871429)(9.97776,-0.01285712)(11.077761,-0.01285712)(11.077761,0.7871428801521517)
\psbezier[linecolor=blue, linewidth=0.04](9.97776,0.7871429)(9.97776,1.5871428)(11.077761,1.5871428)(11.077761,0.7871428801521517)
\psline[linecolor=red, linewidth=0.04, arrowsize=0.06cm 3.0,arrowlength=1.4,arrowinset=0.0]{<-}(2.2777605,0.8871429)(2.2777605,0.5871429)
\psline[linecolor=red, linewidth=0.04](2.2777605,2.4871428)(2.2777605,-2.112857)
\psbezier[linecolor=colour1, linewidth=0.04](1.6259203,0.28260043)(2.3227432,0.28260043)(2.3227432,0.050326124)(3.019566,0.050326123382021706)
\pscircle[linecolor=black, linewidth=0.04, dimen=outer](1.6259203,-0.4142225){0.34841147}
\psline[linecolor=colour1, linewidth=0.04](0.4777606,0.18714288)(0.02776062,0.18714288)
\rput[bl](2.1453185,-2.4892063){$\alpha$}
\rput[bl](4.831306,-0.0038712155){$\beta$}
\rput[bl](4.8777604,-0.6302757){$\arc$}
\psline[linecolor=red, linewidth=0.04](0.02776062,-0.41285712)(1.2777606,-0.41285712)
\psline[linecolor=red, linewidth=0.04, arrowsize=0.05291667cm 2.0,arrowlength=1.4,arrowinset=0.0]{<-}(0.17776062,-0.41285712)(0.4777606,-0.41285712)
\psline[linecolor=colour1, linewidth=0.04, arrowsize=0.05291667cm 2.0,arrowlength=1.4,arrowinset=0.0]{->}(0.37776062,0.18714288)(0.17776062,0.18714288)
\pscircle[linecolor=black, linewidth=0.04, dimen=outer](3.019566,0.74714905){0.34841147}
\psbezier[linecolor=red, linewidth=0.04](4.5977607,-0.51285714)(3.4777606,-0.51285714)(4.477761,0.7871429)(3.3677607,0.7871428801521517)
\psbezier[linecolor=colour1, linewidth=0.04](0.4777606,0.18714288)(1.6391321,0.18714288)(1.8869634,1.4881016)(3.0192654,1.4444972438424377)
\psbezier[linecolor=colour1, linewidth=0.04](1.6259203,0.28260043)(1.1613717,0.28260043)(0.92909735,0.050326124)(0.92909735,-0.4142224954434255)(0.92909735,-0.8787711)(1.1613717,-1.1110455)(1.6259203,-1.1110455)
\psbezier[linecolor=colour1, linewidth=0.04](3.019566,1.443972)(3.4841146,1.443972)(3.7163892,1.2116977)(3.7163892,0.7471490516200617)(3.7163892,0.28260043)(3.4841146,0.050326124)(3.019566,0.050326124)
\psbezier[linecolor=colour1, linewidth=0.04](4.6454864,0.16646327)(3.213212,0.14578368)(3.4777606,-1.1128571)(1.5484955,-1.111045423681586)
\psdots[linecolor=black, dotsize=0.16](4.3277607,0.13714288)
\psframe[linecolor=black, linewidth=0.04, dimen=outer](4.6454864,2.4892063)(0.0,-2.1562798)
\psline[linecolor=red, linewidth=0.04, arrowsize=0.06cm 3.0,arrowlength=1.4,arrowinset=0.0]{<-}(9.577761,0.8871429)(9.577761,0.5871429)
\psline[linecolor=red, linewidth=0.04](9.577761,2.4871428)(9.577761,-2.112857)
\rput[bl](9.445318,-2.4892063){$\alpha$}
\rput[bl](12.177761,-0.6302757){$\arc$}
\psline[linecolor=red, linewidth=0.04](7.3277607,-0.41285712)(8.577761,-0.41285712)
\psline[linecolor=red, linewidth=0.04, arrowsize=0.05291667cm 2.0,arrowlength=1.4,arrowinset=0.0]{<-}(7.477761,-0.41285712)(7.7777605,-0.41285712)
\psbezier[linecolor=red, linewidth=0.04](11.89776,-0.51285714)(10.7777605,-0.51285714)(11.7777605,0.7871429)(10.667761,0.7871428801521517)
\psdots[linecolor=black, dotsize=0.14](2.6777606,0.7871429)
\psbezier[linecolor=blue, linewidth=0.04](9.97776,0.73714286)(9.97776,-0.6128571)(9.97776,-1.3128572)(9.077761,-1.3128571198478483)
\psbezier[linecolor=blue, linewidth=0.04](9.97776,0.73714286)(9.97776,-0.6128571)(9.97776,-1.3128572)(10.877761,-1.3128571198478483)
\psline[linecolor=blue, linewidth=0.04](9.077761,-1.3128572)(7.3377604,-1.3128572)
\psline[linecolor=blue, linewidth=0.04](10.877761,-1.3128572)(11.927761,-1.3128572)
\psline[linecolor=blue, linewidth=0.04, arrowsize=0.05291667cm 2.0,arrowlength=1.4,arrowinset=0.0]{->}(10.577761,0.18714288)(10.377761,0.18714288)
\psline[linecolor=blue, linewidth=0.04, arrowsize=0.05291667cm 2.0,arrowlength=1.4,arrowinset=0.0]{->}(7.6777606,-1.3128572)(7.477761,-1.3128572)
\rput[bl](10.47776,1.5871428){$\arc^*$}
\rput[bl](12.077761,-1.4128572){$\alpha^*$}
\rput[bl](2.4777606,0.8871429){$p$}
\rput[bl](9.7777605,0.98714286){$p$}
\psdots[linecolor=black, dotsize=0.14](9.97776,0.7871429)
\pscircle[linecolor=black, linewidth=0.04, dimen=outer](10.319566,0.74714905){0.34841147}
\psframe[linecolor=black, linewidth=0.04, dimen=outer](11.945486,2.4892063)(7.3,-2.1562798)
\pscircle[linecolor=black, linewidth=0.04, dimen=outer](8.9259205,-0.4142225){0.34841147}
\end{pspicture}
\caption{Left: an oriented, $\bb$-based, extended Heegaard diagram of the (sutured) complement of the left trefoil knot  $K\sb S^3$. The square is assumed to be oriented towards the reader and opposite sides have to be identified. Right: the corresponding dual curves $\a^*, a^*\in\pi_1(S^3\sm K,p)$ (drawn over $\S$ in blue).}
\label{fig: dual curves in pi1}
\end{figure}

\medskip

Suppose $\HH$ is {\em oriented}, that is, all the curves and arcs are oriented, and {\em $\bb$-based}, meaning that each curve in $\bb$ has a basepoint in $\S\sm\aaa$. For each $\b\in\bb$ let $\ov{\b}$ be the word in the generators $\a^*_1,\dots,\a^*_{d+l}$ defined as follows: for each crossing $x$ of $\b$ write $(\a^*_x)^{m_x}$ where $\a_x \in\aaa$ is the curve or arc passing through that point. Then multiply these generators from left to right as we follow the orientation of $\b$ starting from its basepoint. Then the elements $\a^* ( \a\in\aaa)$ and the words $\ov{\b} (\b\in\bb)$ define a presentation of the fundamental group of $M$:
\begin{align}
\label{eq: pres of pi1}
\pi_1(M,p)=\lb \a^*_1,\dots,\a^*_{d+l} \ | \ \ov{\b}_1,\dots, \ov{\b}_d\rb.
\end{align}

\noindent For instance, the diagram of Figure \ref{fig: dual curves in pi1} (left) determines the presentation $
\pi_1(S^3\sm K,p)=\lb \a, a \ | \ a\a a^{-1}\a^{-1}a^{-1}\a\rb$ where we denote $\a^*,a^*$ just by $\a,a$ for simplicity.
\medskip

\subsection{$\Spinc$ structures and multipoints}
\label{subs: Spinc} Let $(M,\c)$ be a balanced sutured 3-manifold. Fix a vector field $v_0$ on $\p M$ that points into $M$ (resp. out of $M$) along $R_-$ (resp. $R_+$) and along $\c$ it is equivalent to the gradient of some height function $\c=s(\c)\t [-1,1]\to [-1,1]$. Let $v,w$ be two nowhere-vanishing vector fields on $M$ such that $v|_{\p M}=v_0=w|_{\p M}$. We say that $v$ and $w$ are {\em homologous} if there is an embedded open 3-ball $B\sb \inte(M)$ such that $v,w$ are homotopic rel $\p M$ in $M\sm B$ through nowhere-vanishing vector fields. A {\em $\Spinc$}-structure is an homology class of such vector fields. The set of $\Spinc$ structures on $M$ is denoted by $\Spinc(M,\c)$. The obstruction to homotope a $\Spinc$ structure $\ss_1$ onto another $\ss_2$ lies in $H^2(M,\p M)$, we denote it by $\ss_1-\ss_2\in H^2(M,\p M)$. Hence $\Spinc(M,\c)$ is an affine space over $H^2(M,\p M)$, see \cite{Juhasz:holomorphic} for more details. 
\medskip

Now let $\HD$ be a balanced Heegaard diagram of $(M,\c)$ and let $d=|\aa|=|\bb|$. Then $\Spinc$ structures on $(M,\c)$ can be encoded in a very simple way: there is a map
\begin{align*}
s:\Tab\to \Spinc(M,\c)
\end{align*}
where $\Tab$ is the set of {\em multipoints} of $\HD$, i.e. unordered $d$-tuples $\x=\{x_1,\dots,x_d\}$ with $x_i\in\a_{\s(i)}\cap\b_{i}$ for each $i$ where $\s$ is some permutation in $S_d$. We refer the reader to \cite{Juhasz:holomorphic} for the definition of this map (which is a sutured extension of the map of \cite{OS1} in the closed case). The only thing we will need from this map is that 
\begin{align}
\label{eq: Spinc and exy}
s(\y)-s(\x)=PD[\e(\y,\x)]
\end{align}
where $PD$ denotes Poincar\'e duality and $\e(\y,\x)$ is an homology class which is easily described through an extended Heegaard diagram of $(M,\c)$ (assuming $R_-(\c)$ connected). To do this, we put basepoints on the $\b$ curves of our diagram as in \cite{LN:Kup}.

\begin{definition}
\label{def: basepoints from multipoint}
Let $\HH$ be an oriented Heegaard diagram and let $\x\in\Tab$ be a multipoint in $\HH$, say $\x=\{x_1,\dots, x_d\}$ where $x_i\in\a_{\s(i)}\cap\b_i$ for each $i=1,\dots,d$ and some $\s\in S_d$. For each $i$, we let $q_i=q_i(\x)\in \b_i$ be a basepoint defined as follows: if the crossing $x_i$ is positive (resp. negative), then $q_i$ lies just before $x_i$ (resp. after $x_i$) when following the orientation of $\b_i$, see Figure \ref{fig: basepoint convention}. 

\end{definition}


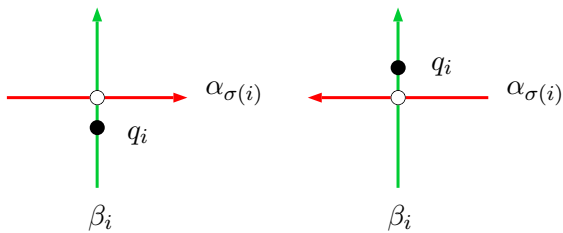
\begin{figure}[h]
\centering
\begin{pspicture}(0,-1.4334792)(6.94,1.52)
\definecolor{colour0}{rgb}{0.0,0.8,0.2}
\psline[linecolor=colour0, linewidth=0.04, arrowsize=0.05291667cm 2.0,arrowlength=1.4,arrowinset=0.0]{<-}(1.2,1.5265208)(1.2,-0.8734791)
\psline[linecolor=red, linewidth=0.04, arrowsize=0.05291667cm 2.0,arrowlength=1.4,arrowinset=0.0]{->}(0.0,0.32652083)(2.4,0.32652083)
\psdots[linecolor=black, fillstyle=solid, dotstyle=o, dotsize=0.2, fillcolor=white](1.2,0.32652083)
\rput[bl](1.08,-1.4334792){$\beta_i$}
\psdots[linecolor=black, dotsize=0.2](1.2,-0.07347915)
\rput[bl](1.6,-0.29347914){$q_i$}
\rput[bl](2.64,0.22652084){$\alpha_{\s(i)}$}
\psline[linecolor=colour0, linewidth=0.04, arrowsize=0.05291667cm 2.0,arrowlength=1.4,arrowinset=0.0]{<-}(5.2,1.5265208)(5.2,-0.8734791)
\psline[linecolor=red, linewidth=0.04, arrowsize=0.05291667cm 2.0,arrowlength=1.4,arrowinset=0.0]{<-}(4.0,0.32652083)(6.4,0.32652083)
\psdots[linecolor=black, dotstyle=o, dotsize=0.2, fillcolor=white](5.2,0.32652083)
\rput[bl](5.06,-1.4334792){$\beta_i$}
\rput[bl](5.64,0.6265209){$q_i$}
\rput[bl](6.64,0.22652084){$\alpha_{\s(i)}$}
\psdots[linecolor=black, dotsize=0.2](5.2,0.72652084)
\end{pspicture}

\caption{Basepoints on $\bb$ coming from $\x\in\Tab$. The surface is oriented towards the reader, so that the left crossing is positive. The white dot represents $x_i\in\a_{\s(i)}\cap\b_i$.}

\label{fig: basepoint convention}
\label{figure: reversing orientation of alpha changes basepoints}
\end{figure}

So let $\HH=\HDD$ be an oriented extended Heegaard diagram of $(M,\c)$ and let $\x,\y\in\Tab$. For each $i=1,\dots,d$ we let $d_i\sb\b_i$ be the arc from $q_i(\x)$ to $q_i(\y)$. Then there is an element $\ov{d}_i\in\pi_1(M,p)$ obtained by multiplying the $(\a^*)^{\pm 1}$'s ($\a\in\aaa$) obtained from the crossings through $d_i$ following its orientation. Then for any $\x,\y\in\Tab$ one has
\begin{align}
\label{eq: eyx via pi1}
\e(\y,\x)=\prod_{i=1}^dh(\ov{d_i})
\end{align}
in $H_1(M)$ (in multiplicative notation) where $h:\pi_1(M)\to H_1(M)$ is the Hurewicz map \cite[Lemma 4.4]{LN:Kup}.

\subsection[Homology orientations]{Homology orientations}
\label{subsection: homology orientation}

A {\em homology orientation} of a balanced sutured manifold $(M,\c)$ is an orientation $\o$ of the vector space $H_*(M,R_-(\c);\R)$. If $\HH=\HD$ is a balanced Heegaard diagram of $(M,\c)$ and $d=|\aa|=|\bb|$, then it is shown in \cite{FJR11} that a homology orientation is equivalent to a {\em sign-ordering} of $\HH$, that is, an orientation of all the curves $\aa\cup\bb$ and an total order of each set $\aa,\bb$ up to some equivalence. We refer the reader to \cite{FJR11} for this correspondence. Given an ordered (i.e. $\aa$ and $\bb$ are totally ordered), oriented Heegaard diagram $\HH$ we denote by $o(\HH)$ the corresponding homology orientation. Note that if $H_*(M,R_-(\c);\R)=0$ (e.g. when $M$ is the complement of a link in a rational homology 3-sphere) there is a canonical sign-ordering of any Heegaard diagram, determined by the condition $\det(A)>0$ where $A=(\a_i\cdot\b_j)_{i,j=1,\dots, d}$ is the intersection matrix of $\HH$. When $Y$ is a closed oriented 3-manifold and $(M_0,\c_0)$ is the associated sutured manifold (Example \ref{example: sutured from closed}), there is also a canonical homology orientation of $H_*(M_0,R_-(\c_0))=H_2(Y)\oplus H_1(Y)$, determined by any basis of $H_1(Y)$ followed by its Poincar\'e dual in $H_2(Y)$ (see \cite[Section 18]{Turaev:BOOK2}).

\subsection{Twisted torsion of sutured manifolds}\label{subs: twisted torsion for sutured manifolds}

\def\rhovc{\rho\ot\ov{h}}
We now briefly introduce the twisted relative Reidemeister torsion of a sutured manifold, for more details see \cite{Turaev:BOOK1, FV:survey}. Let $X$ be a finite connected CW complex and $Y$ a possibly empty subcomplex such that $\chi(X,Y)=0$. Let $\rho:\pi\to GL(V)$ be a representation, where $V$ is a finite dimensional vector space over a field $\kk$ and $\pi=\pi_1(X)$. The {\em twisted Reidemeister torsion} is a topological invariant $\tau^{\rho}(X,Y)\in\kk$ of $(X,Y,\rho)$ defined up to multiplication by $\pm\det\rho(g),g\in\pi$. Briefly, this is defined as follows: consider the universal covering $q:\wt{X}\to X$ of $X$ and let $Y'=q^{-1}(Y)$. Then the cellular complex $C_*(\wt{X},Y')$ is a complex of free left $\Z[\pi]$-modules, which we consider as right $\Z[\pi]$-modules via $c\cdot g\eq g^{-1}(c)$ for each $g\in\pi$ and cell $c$ of $\wt{X}$. Now consider the complex $C_*^{\rho}(X,Y)\eq C_*(\wt{X},Y')\ot_{\Z[\pi]}V$ where $V$ is a left $\Z[\pi]$-module via $\rho$. If we choose a lift to $\wt{X}$ of each cell of $X\sm Y$, then we get a preferred basis of this complex by tensoring these lifts with an arbitrary basis of $V$. If the complex $C_*^{\rho}(X,Y)$ is not acyclic we set $\tau^{\rho}(X,Y)=0$ and if it is acyclic, we define $\tau^{\rho}(X,Y)$ as the algebraic torsion of this based complex, see \cite{Turaev:BOOK1}. The $\pm\det\rho(g)$ indeterminacy of the torsion comes from the choice of lifts as well as their order and orientation. If $X$ has dimension 2 and all the 0-cells are contained in $Y$ (in particular $Y$ is non-empty) then $C_i(\wt{X},Y')=0$ for $i\neq 1,2$ and so the torsion is just the determinant of the boundary map $\p_2^{\rho}:C^{\rho}_2(X,Y)\to C_1^{\rho}(X,Y)$ which is computed via Fox calculus (see e.g. \cite[Claim 16.6]{Turaev:BOOK1}). This implies that we can define the torsion provided $\kk$ is only a commutative ring with unit and $\tau^{\rho}(X,Y)\in\kk$ (instead of the ring obtained by inverting the non-zero divisors of $\kk$). In particular, if $h:\pi_1(X)\to H_1(X)$ is the Hurewicz map and if we define $\rhoc:\pi_1(X)\to GL(V\ot_{\kk}\kk[H_1(X)])$ by
\begin{equation}
\label{eq: def of rhoc}
\rhoc(\d)(v\ot f)\eq \rho(\d)(v)\ot (h(\d)\cdot f)
\end{equation}
where $\d\in\pi_1(X),v\in V, f\in H_1(X)$, then there is a torsion $\tau^{\rhoc}(X,Y)\in\kk[H_1(X)]$ under the above hypothesis for $(X,Y)$. When $\rho\equiv 1$ and $\dim(V)=1$ this is the {\em maximal abelian torsion} of the pair $(X,Y)$.

\medskip

Now suppose $(M,\c)$ is a balanced sutured manifold with $R_-(\c)$ connected and let $\pi=\pi_1(M,p)$ where $p\in s(\c)$ is a basepoint. We now explain how to compute the torsion $\tau^{\rho}(M,R_-)$ from an extended Heegaard diagram $\HH=\HDD$. Suppose $\HH$ is oriented, $\bb$-based and {\em ordered}, that is, each set $\aa$ and $\bb$ is linearly ordered and write $\aa=\{\a_1\dots,\a_d\}, \bb=\{\b_1,\dots,\b_d\}$ in the corresponding orders. Let $\bolda=\{\a_{d+1},\dots,\a_{d+l}\}$ (unordered). Consider the presentation of $\pi_1(M,p)$ given in (\ref{eq: pres of pi1}).
If $F$ is the free group generated by $\a^*_1,\dots,\a^*_{d+l}$, recall that the Fox derivatives $\p/\p\a^*_i:\Z[F]\to\Z[F]$ are characterized by $\p\a^*_j/\p\a^*_i=\d_{ij}$ and $\p(uv)/\p\a^*_i=\p u/\p\a^*_i+u\cdot\p v/\p\a^*_i$ for each $u,v\in F$. From now on, we will consider the Fox derivatives $\p\ov{\b}_j/\p \a^*_i$ as elements of $\Z[\pi]$.

\begin{proposition}
\label{prop: torsion via FOX calculus}
Given an ordered, oriented, based, extended Heegaard diagram $\HDD$ of $(M,\c)$, the twisted torsion of the pair $(M,R_-)$ is computed via Fox calculus by 
\begin{align*}
\tau^{\rho}(M,R_-)=\det\left(\rho\left(\s\left(\frac{\p \ov{\b}_j}{\p\a^*_i}\right)\right)\right)_{i,j=1,\dots,d}
\end{align*}
up to an indeterminacy of the form $\pm\det(\rho(g)), g\in\pi$ where $\s:\Z[\pi]\to\Z[\pi]$ is defined by $\s(g)=g^{-1}$ for $g\in\pi$. The same formula is valid for $\tau^{\rho\ot h}$.
 
\end{proposition}
\begin{proof}
This is the twisted version of \cite[Lemma 3.6]{FJR11}. Think of $(M,\c)$ as obtained from a single 0-handle by attaching $l$-handles (specifying $R_-\t I$), then $d$-handles corresponding to the $\a$'s and then $d$-handles corresponding to the $\b$'s. Now, let $X$ be the CW complex obtained by collapsing each of these handles to its core and let $Y$ be the subcomplex of $X$ corresponding to $R_-\t I$.  By \cite[Corollary 8.5]{Turaev:BOOK2} collapsing a handle to its core does not changes the relative torsion, so $\tau^{\rho}(M,R_-)=\tau^{\rho}(X,Y).$ But $\tau^{\rho}(X,Y)=\det(\p_2^{\rho})$ where $\p_2^{\rho}:C_2^{\rho}(X,Y)\to C_1^{\rho}(X,Y)$ is the boundary map, since $X$ is a 2-complex whose 0-cells are all contained in $Y$. For an appropriate basis of $C_2^{\rho}(X,Y)$ and $C_1^{\rho}(X,Y)$, the boundary map $\p_2^{\rho}$ is represented by the matrix $\rho\left(\s\left(\p \ov{\b}_j/\p\a^*_i\right)\right)_{i,j=1,\dots, d}$ \cite[Claim 16.6]{Turaev:BOOK1}, this implies the desired formula.

\end{proof}


\subsection{Refining $\tau$ using $\Spinc$ via multipoints}\label{subs: normalizing tau via Spinc} We now explain a simple Heegaard-diagrammatic way to remove the indeterminacies of the above formula, relying on the multipoint representation of $\Spinc$. This is an idea from \cite{LN:Kup} which we restate in the case of torsion. Let $\HH$ be an ordered, oriented and $\bb$-based extended Heegaard diagram of $(M,\c)$ as above, then the determinant of the right hand side of Proposition \ref{prop: torsion via FOX calculus} is a well-defined scalar of $\kk$ (i.e., there is no indeterminacy). This depends on the extra structure on $\HH$ (basepoints and orientations) up to $\pm\det(\rho(g))$ as follows. Let $q\in\b$ be the basepoint of a curve $\b\in\bb$ and let $q'\in\b$ be another basepoint. Let $c\sb\b$ be the oriented arc from $q$ to $q'$. Let $\ov{\b},\ov{\b'}$ be the words in the $\a^*$'s obtained from $\b$ but starting from $q$ and $q'$ respectively. These are related by $\ov{\b}=\ov{c}\cdot\ov{\b'}\cdot\ov{c}^{-1}$ where $\ov{c}\in\pi$ is obtained by the multiplying the $(\a^*)^{\pm 1}$ through $c$ when following its orientation. But then the Fox derivatives satisfy $$\frac{\p\ov{\b}}{\p\a^*}=\ov{c}\cdot\frac{\p\ov{\b'}}{\p\a^*}$$ when thought as elements of $\Z[\pi]$, for each $\a\in\aa$. Thus, changing the basepoint of a $\b$-curve has the effect of multiplying the above determinant by $\det\rho(\ov{c})$, where $c$ is the arc from the old basepoint to the new one. There is another source of a $\det\rho(g)$ indeterminacy, coming from the choice of orientations of the closed $\a$ curves. Indeed, reversing a curve $\a_i$ has the effect of multiplying $\p\ov{\b}_j/\p\a^*_i$ by $-\a^*_i$ on the right for each $j$, hence the determinant of Proposition \ref{prop: torsion via FOX calculus} is multiplied by $(-1)^n\det(\rho(\a^*_i)^{-1})$.
\medskip

The above $\det\rho(g)$ indeterminacies can be removed using the rule of Definition \ref{def: basepoints from multipoint}. Thus, pick a multipoint $\x=\{x_1,\dots, x_d\}$ of $\HH$, for simplicity we suppose $x_i\in\a_i\cap\b_i$ for each $i$, and put basepoints on all $\b$ curves as in Definition \ref{def: basepoints from multipoint}. This immediately removes the $\det(\rho(g))$ indeterminacy coming from the orientation of the $\a$'s: if the orientation of an $\a_i$ is reversed, then the basepoint of $\b_i$ is crossed through $\a_i$, and this cancels the above $\det(\rho(\a^*_i))$ factor (only leaving a sign indeterminacy). Now if we change the multipoint, say to $\y\in\Tab$, then the basepoints of all $\b$ curves change. More precisely, if $d_j$ is the arc of $\b_j$ from the $\x$-basepoint to the $\y$-basepoint, then the whole determinant is multiplied by $\prod_{j=1}^d\det\circ\rho(\ov{d_j})$ by the above argument. But by (\ref{eq: eyx via pi1}), this is exactly $\det\circ\rho(\e(\y,\x))$ (note that $\det\circ\rho$ descends to $H_1(M)$). Similarly, if we are given $\ss\in\Spinc(M,\c)$, then by (\ref{eq: Spinc and exy}) the scalar $\det\circ\rho(PD[\ss-s(\x)])$ is multiplied by $\det\circ\rho(\e(\x,\y))$ when going from $\x$ to $\y$. It follows that the expression
\begin{align*}
\det\circ\rho(PD[\ss-s(\x)])\cdot\det\left(\rho\left(\s\left(\frac{\p\ov{\b_j}}{\p\a^*_i}\right)\right)\right)_{i,j=1,\dots, d}
\end{align*}
is independent of the multipoint chosen, where the $\b$ curves have the basepoints coming from $\x$. This still has a sign indeterminacy, but this is easily corrected by further picking an orientation $\o$ of $H_*(M,R_-(\c);\R)$. Indeed, if $\d$ is the sign defined by $o(\HH)=\d\o$ where $o(\HH)$ is the orientation of $H_*(M,R_-(\c);\R)$ induced from the ordering and orientation of $\HH$ \cite{FJR11} and if $n=\dim(V)$, then
\begin{align*}
\d^n\cdot\det\circ\rho(PD[\ss-s(\x)])\cdot\det\left(\rho\left(\s\left(\frac{\p\ov{\b_j}}{\p\a^*_i}\right)\right)\right)_{i,j=1,\dots, d}
\end{align*}
is independent of the extra structure we added to $\HH$ (ordering, orientation and basepoints). One can then show directly that this is invariant under extended Heegaard moves, hence it depends only on $(M,\c)$, $\rho:\pi_1(M,p)\to GL(V)$, $\ss\in\Spinc(M,\c)$ and the orientation $\o$. Thus, the latter expression can be considered as another normalization of the twisted Reidemeister torsion $\tau^{\rho}(M,R_-(\c))$, cf. \cite{FJR11}. 

\begin{definition}
Let $(M,\c)$ be a balanced sutured 3-manifold with connected $R_-(\c)$ and let $\rho,\ss,\o$ be as before. Then we denote by $\ovtau^{\rho}(M,\c,\ss,\o)$ the above refinement of the twisted Reidemeister torsion $\tau^{\rho}(M,R_-(\c))$.
\end{definition}

\subsection{The disconnected case}\label{subs: the disconnected case} Now let $(M,\c)$ be a balanced sutured 3-manifold with disconnected $R_-(\c)$. Attach a 2-dimensional one-handle $h_0$ to two circles of $s(\c)$ corresponding to different components of $R_-(\c)$. Thickening $h_0$ to a 3-dimensional one-handle $h_0\t[-1,1]$ attached along $s(\c)\t[-1,1]=\c$ results in a sutured manifold $(M',\c')$ in which $R_-(\c')$ has one component less than $R_-(\c)$. Therefore, after sufficiently many such handle attachments, we can ensure that the resulting sutured manifold, still denoted $(M',\c')$, has connected $R_-(\c')$. Let $\ss\in\Spinc(M,\c)$ and suppose $v$ is a vector field representing $\ss$. One can extend $v$ to a non-vanishing vector field $v'$ on any such $M'$ by letting $v'|_M=v$ and letting $v'$ be the vertical vector field over each one-handle $h_0\t[-1,1]$ attached to $M$. The $\Spinc$ structure on $M'$ defined by $v'$ is denoted $i(\ss)$. This defines an affine map $i:\Spinc(M,\c)\to\Spinc(M',\c')$, that is, it satisfies that $$j^*(i(\ss_1)-i(\ss_2))=\ss_1-\ss_2$$ for any $\ss_1,\ss_2\in\Spinc(M,\c)$, where $j^*:H^2(M',\p M')\to H^2(M,\p M)$ is induced by inclusion, see \cite[Proposition 5.4]{Juhasz:polytope}. a homology orientation $\o$ of $(M,\c)$ induces a homology orientation $\o'$ for $(M',\c')$ since $H_*(M,R_-;\R)\cong H_*(M',R'_-;\R)$. Then we set $$\ovtau^{\rho}(M,\c,\ss,\o)\eq \ovtau^{\rho'}(M',\c',i(\ss),\o')$$ 
where $(M',\c')$ is any sutured manifold with connected $R_-(\c')$ constructed as above and $\rho':\pi_1(M')\to GL(V)$ is an arbitrary extension of $\rho$ (note that $\pi_1(M')=\pi_1(M)*F$ where $F$ is a free group with as many generators as handles were attached to $M$ to get $M'$). It is also a consequence of \cite{LN:Kup} (via Theorem \ref{Theorem: I for sd product is Fox calculus} and Theorem \ref{Theorem: I at exterior is torsion}) that this is well-defined, i.e., independent of the $(M',\c')$ chosen and it is a consequence of \cite[Lemma 3.20]{FJR11} that this is indeed a refinement of the torsion $\tau^{\rho}(M,R_-(\c))$. Note that if we take $\rho\ot h$, then a priori $\ovtau^{\rho'\ot h'}(M',\c',i(\ss),\o')\in\kk[H_1(M')]$. This actually belongs to $\kk[H_1(M)]$ as a consequence of Lemma \ref{lemma: disconnected case twisted poly belongs to H1M} below together with Theorem \ref{Theorem: I at exterior is torsion}.
\medskip

\subsection{Twisted torsion of sutured link complements}\label{subs: Twisted torsion of sutured link complements} Now let $L$ be a link in a closed oriented 3-manifold $Y$, with components $L_1,\dots, L_m$. Let $(M_L,\c_L)$ be the associated sutured manifold (see Example \ref{example: sutured from link complement}) and let $\rho:\pi_1(M_L)\to GL(V)$ be a representation into a finite dimensional vector space $V$ over a field $\kk$. For simplicity, we will assume $Y$ is an homology sphere, so $H_1(M_L)\cong\Z^{\oplus m}$. Let $h:\pi_1(M)\to H_1(M_L)$ be the projection, recall that $\rho$ can be combined with $h$ to define a representation $\rhoc$ over $V'=V\ot_{\kk}\kk[H_1(M_L)]$ as in (\ref{eq: def of rhoc}). If $\wt{M_L}$ is the universal covering of $M_L$, then $C_*^{\rhoc}(\wt{M_L})=C_*(\wt{M_L})\ot_{\Z[\pi]}V'$ is a complex of free $\Z[H_1(M_L)]$-modules ($V'$ is a $\Z[\pi]$-module via $\rhoc$). For each $i\geq 0$, the $i$-th homology group of this complex is a finitely presented module over the Noetherian UFD $\kk[H_1(M_L)]$, so we can define its order, which is an element of $\kk[H_1(M_L)]$ defined up to multiplication by a unit. This is the $i$-th {\em twisted Alexander polynomial} of $L$, denoted $\De_{L,i}^{\rho}$, the first twisted Alexander polynomial is denoted just by $\De^{\rho}_L$. If $L$ is ordered and oriented and if $t_i\in H_1(M_L)$ is the class of a meridian around $L_i$ oriented so that $lk(t_i,L_i)=+1$ for each $i=1,\dots,m$ then $\kk[H_1(M_L)]\cong \kk[t_1^{\pm 1},\dots,t_m^{\pm 1}]$ canonically, and so we can think of the twisted Alexander polynomials of $L$ as elements of $\kk[t_1^{\pm 1},\dots, t_m^{\pm 1}]$ defined up to multiplication by a unit. See \cite{Turaev:BOOK1, FV:survey} for more details.
\medskip

We now relate these polynomials to the torsion of $(M_L,R_-)$. If $R_-=R_1\cup\dots\cup R_m$ where $R_i$ is an annulus around $L_i$, let $y_i\in R_i$ be a basepoint and $a^*_i$ be the generator of $\pi_1(R_i,y_i)$ corresponding to $t_i$. Then $a^*_i$ is defined in $\pi_1(M_L,p)$ up to conjugation, so $\det(t_i\rho(a^*_i)-I_n)$ is well-defined and non-zero. Note that the $0$-th Alexander polynomial $\De_{L,0}^{\rho}$ is always non-zero by \cite[Proposition 3.2, (1)]{FV:survey}.

\def\Der{\De^{\rho}}

\begin{proposition}
\label{corollary: sutured torsion recovers TWISTED ALEX}
Let $(M_L,\c_L)$ be the sutured manifold associated to the complement of an $m$-component (ordered, oriented) link $L\sb Y$ where $Y$ is an homology sphere. Let $\rho:\pi_1(M_L)\to GL(V)$ be a homomorphism, where $V$ is an $n$-dimensional vector space over a field $\kk$. Then
\begin{align*}
\tau^{\rho\ot h}(M_L,R_-(\c_L))\dot{=}
\frac{\prod_{i=1}^m \det(t_i\rho(a^*_i)-I_n)}{\De^{\rho}_{L,0}}\cdot \De^{\rho}_L(t_1,\dots,t_m)\in\kk[t_1^{\pm 1},\dots,t_m^{\pm 1}].
\end{align*}

\end{proposition}
\begin{proof}
Let $M=M_L$, since $\tau^{\rho\ot h}(R_-)=\prod_{i=1}^m \det(t_i\rho(a^*_i)-I_n)^{-1}$ is non-zero, we have $$\tau^{\rho\ot h}(M,R_-)=\tau^{\rho\ot h}(R_-)^{-1}\tau^{\rho\ot h}(M)$$ by multiplicativity of the torsion. If $\De_L^{\rho}=0$, then $C_*^{\rho\ot h}(M)$ is not acyclic and so $\tau^{\rho\ot h}(M)=0$. Thus $\tau^{\rho\ot h}(M,R_-)$ is zero and the above equation holds. Suppose now that $\De_L^{\rho}\neq 0$, then $\De_{L,2}^{\rho}=1$ by \cite[Proposition 3.2]{FV:survey} (since $M$ has non-empty boundary) and since $\De_{L,3}^{\rho}=1$ and $\De_{L,0}^{\rho}\neq 0$ always hold, it follows that $C_*^{\rho\ot h}(M)$ is acyclic. By \cite[Theorem 4.7]{Turaev:BOOK1} we get
\begin{align*}
\tau^{\rho\ot h}(M,R_-)\dot{=}
\prod_{i=1}^m \det(t_i\rho(a^*_i)-I_n)\cdot\frac{\Der_{L,1}}{\Der_{L,0}}
\end{align*}
as desired. 

\end{proof}

Note that $\De_{L,0}^{\rho}=1$ whenever $m>1$ (and $\rho$ is arbitrary) or if $m=1$ and $\rho$ is irreducible with $\rho|_{\Ker(h)}\neq 1$, while if $m=1$ and $\rho\equiv 1$ (with $\dim(V)=1$), then $\De_{L,0}=t-1$ \cite[Proposition 3.2]{FV:survey}.

\section{Twisted Kuperberg invariants}
\label{sect: the invariant}

\def\rhoa{\rho_A}\def\rhob{\rho_B}


In this section, we define the invariant $I_H^{\rho}(M,\c,\ss,\o)$ and study some of its properties. We begin by defining a scalar $Z_H^{\rho}(\HH)$ by twisting Kuperberg's formulas via Fox calculus using the homomorphism $\rho$, where $\HH$ is an extended Heegaard diagram of a balanced sutured 3-manifold $(M,\c)$ with connected $R_-(\c)$. This is not necessarily a topological invariant of $(M,\c,\rho)$ as it may have some indeterminacies of the form $\pm\lH(\rho(g)), g\in\pi_1(M)$. In Subsection \ref{subs: general definition w Spinc and hom or} we fix these indeterminacies using homology orientations and $\Spinc$ structures, leading to $I_H^{\rho}(M,\c,\ss,\o)$ and we treat the case of disconnected $R_-(\c)$ as well. In Subsection \ref{subs: twisted Kup polynomials} we explain how to extend such numerical invariants to polynomial invariants when $H$ is $\N$-graded. In Subsection \ref{subs: computation of ZHn}, we prove Theorem \ref{Theorem: I at exterior is torsion}. Finally, we discuss in Subsection \ref{sect: going beyond R torsion} the constraints imposed by the involutory condition of Hopf superalgebras.

\medskip

Throughout all this section, we let $(H,m,1,\De,\e,S)$ be a finite dimensional involutory Hopf superalgebra over a field $\kk$. We suppose $H$ admits a two-sided cointegral $\coint\in H$ and a two-sided integral $\int\in H^*$, and we assume $\int(\coint)=1$. 
We also let $(M,\c)$ be a connected balanced sutured 3-manifold and let $\rho:\pi_1(M,p)\to \Aut(H)$ be a group homomorphism, where $p\in s(\c)$ is a basepoint.

\subsection{Twisted tensors associated to Heegaard diagrams}\label{subs: tensors and HDs}

\def\mbb{m_{\bb}}\def\Ibb{\mc{I}_{\bb}} 
\def\umbb{\underline{m}_{\bb}}\def\uDeaa{\underline{\De}_{\aa}}
\def\uSaa{\underline{S}_{\aa}}\def\uDeaaa{\underline{\De}_{\aaa}}
\def\uSaaa{\underline{S}_{\aaa}}\def\uSbb{\underline{S}_{\bb}}\def\umb{\underline{m}_{\b}}\def\uDea{\underline{\De}_{\a}}
\def\um{\underline{m}}\def\uDe{\underline{\De}}\def\uS{\underline{S}}
\def\mbb{m_{\bb}}\def\ovbx{\ov{\b}_x}
\def\ka{\kappa_{\a}}\def\kai{\kappa_{\a_i}}\def\mb{m_{\b}}\def\Sb{S_{\b}}\def\Sbb{S_{\bb}}\def\kap{\kappa} 

We begin by supposing that $R_-(\c)$ is connected. Let $\HH=(\S,\aaa,\bb)$ be an extended Heegaard diagram of $(M,\c)$ and let $\aaa=\aa\cup\bolda$ with $\aa=\{\a_1,\dots,\a_d\}$ the closed curves and $\bolda=\{\a_{d+1},\dots,\a_{d+l}\}$ the arcs. From now on, we suppose $\HH$ is {\em ordered}, that is, both sets $\aa$ and $\bb$ are totally ordered, and that $\HH$ is oriented and $\bb$-based, so all the curves and arcs are oriented and each curve in $\bb$ has a basepoint. With this additional structure, we will define a scalar $Z_H^{\rho}(\HH)\in\kk$ by contracting some tensors associated to the closed circles and the crossings.

\medskip
 These tensors are defined in the following way. First, to each closed $\a\in\aa$ we associate the tensor

\def\Ibb{\mc{I}_{\bb}}


\begin{figure}[H]
\centering
\begin{pspicture}(0,-0.5572612)(1.5572608,0.5572612)
\psline[linecolor=black, linewidth=0.018, arrowsize=0.05291667cm 2.0,arrowlength=0.8,arrowinset=0.2]{->}(0.3,0.0)(0.7,0.0)
\psline[linecolor=black, linewidth=0.018, arrowsize=0.05291667cm 2.0,arrowlength=0.8,arrowinset=0.2]{->}(1.2,0.2)(1.6,0.6)
\psline[linecolor=black, linewidth=0.018, arrowsize=0.05291667cm 2.0,arrowlength=0.8,arrowinset=0.2]{->}(1.2,-0.2)(1.6,-0.6)
\rput[bl](1.46,-0.24){$\vdots$}
\psline[linecolor=black, linewidth=0.018, arrowsize=0.05291667cm 2.0,arrowlength=0.8,arrowinset=0.2]{->}(1.2,0.1)(1.6,0.3)
\rput[bl](0.0,-0.08){$\coint$}
\rput[bl](0.8,-0.1){$\Delta$}
\end{pspicture}
\end{figure}

\noindent where the coproduct has as many outcoming legs as crossings through the curve $\a$. More precisely, we suppose that $\a$ has a basepoint (in $\a\sm\bb$) and that from top to bottom, the legs correspond to the crossings of $\a$ encountered as one follows the orientation of $\a$ starting from this basepoint. Note that since we supposed $\coint$ is two-sided, this tensor does not depends on the basepoint of $\a$ chosen. Now, to each crossing $x\in\aa\cap\bb$ we associate the tensor
\begin{figure}[H]
\centering 
\begin{pspicture}(0,-0.12)(1.5265589,0.12)
\psline[linecolor=black, linewidth=0.018, arrowsize=0.05291667cm 2.0,arrowlength=0.8,arrowinset=0.2]{->}(0.0,-0.02)(0.4,-0.02)
\rput[bl](0.55,-0.12){$S^{\epsilon_x}$}
\psline[linecolor=black, linewidth=0.018, arrowsize=0.05291667cm 2.0,arrowlength=0.8,arrowinset=0.2]{->}(1.2,-0.02)(1.6,-0.02)
\end{pspicture}

\end{figure}

\noindent where $\e_x\in\{0,1\}$ is defined by $(-1)^{\e_x}=m_x$ and $m_x$ is the intersection sign. So far, these are the tensors from \cite{Kup1} and we haven't used yet that $\HH$ is an extended Heegaard diagram. The arcs in $\bolda$ will play a role in the definition of the $\b$-tensors, through the homomorphism $\rho$. To define these recall that, since $R_-(\c)$ is assumed to be connected, the extended Heegaard diagram together with the orientation of $\aaa\cup\bb$ and $\bb$-basepoints determines a presentation of $\pi_1(M,p)$ as in \ref{eq: pres of pi1}. This has a generator $\a^*$ for each $\a\in\aaa$ and each $\b\in\bb$ determines a relator $\ov{\b}$, which is obtained by writing $(\a^*)^{m_x}$ for each crossing $x\in\aaa\cap\b$ with sign $m_x$ and multiplying these elements from left to right as one follows the orientation of $\b$ starting from its basepoint. Thus, for a given $\a$, the appearances of $\a^*$ in $\ov{\b}$ are indexed by the intersection points of $\a\cap\b$ and so the Fox derivative $\p\ov{\b}/\p\a^*$can be written as
\begin{align*}
\frac{\p \ov{\b}}{\p\a^*}=\sum_{x\in\a\cap\b}m_x\ovbx\in\Z[F],
\end{align*}
for some word $\ov{\b}_x\in F$, where $F$ is the free group generated by $\a^*,\a\in\aaa$. More precisely, this word is defined as follows:

\begin{notation}
\label{def: bx}
Suppose a curve $\b\in\bb$ is oriented and has a basepoint in $\b\sm \aaa$. 
For each crossing $x$ through $\b$ let $\a_x$ be the $\a$-curve containing $x$. Write $\ov{\b}=w(\a_x^*)^{m_x}w'$ where $w$ (resp. $w'$) is the product of the $\a^*$'s corresponding to the crossings of $\b$ that precede (resp. succeed) $x$, where the crossings are ordered following the orientation of $\b$ starting from its basepoint. Then 
\begin{align*}
\label{eq: bx}
\ovbx\eq\left\{\begin{array}{ccc}
& w & \text{ if }   m_x=1\\
&w(\a_x^*)^{-1} & \text { if } m_x=-1.
\end{array}\right.
\end{align*}
From now on, we think of $\ov{\b}_x$ as an element of $\pi_1(M,p)$ instead of the free group generated by the $\a^*$'s.
\end{notation}

\medskip

Hence, each crossing $x$ through a given curve $\b\in\bb$ has associated a Hopf automorphism $\rho(\ov{\b}_x)\in\Aut(H)$. If $x_1,\dots, x_k$ are the points of $\aa\cap\b$ encountered as one follows the orientation of $\b$ starting from its basepoint, then we associate to $\b$ the tensor

\begin{figure}[H]
\centering
\begin{pspicture}(0,-1.5591924)(3.79,1.5591924)
\psline[linecolor=black, linewidth=0.018, arrowsize=0.05291667cm 2.0,arrowlength=0.8,arrowinset=0.2]{->}(2.2,-0.6499999)(2.6,-0.24999993)
\psline[linecolor=black, linewidth=0.018, arrowsize=0.05291667cm 2.0,arrowlength=0.8,arrowinset=0.2]{->}(3.1,-0.049999923)(3.5,-0.049999923)
\psline[linecolor=black, linewidth=0.018, arrowsize=0.05291667cm 2.0,arrowlength=0.8,arrowinset=0.2]{->}(2.2,0.5500001)(2.6,0.15000008)
\rput[bl](2.7,-0.14999992){$m$}
\rput[bl](3.6,-0.18999992){$\int$}
\rput[bl](1.3,0.6000001){$\rho(\overline{\beta}_{x_1})$}
\rput[bl](0.8,-0.19999993){$\rho(\overline{\beta}_{x_2})$}
\rput[bl](1.4,-1.0999999){$\rho(\overline{\beta}_{x_k})$}
\psline[linecolor=black, linewidth=0.018, arrowsize=0.05291667cm 2.0,arrowlength=0.8,arrowinset=0.2]{->}(1.9,-0.049999923)(2.5,-0.049999923)
\psline[linecolor=black, linewidth=0.018, arrowsize=0.05291667cm 2.0,arrowlength=0.8,arrowinset=0.2]{->}(1.3,-1.55)(1.7,-1.15)
\psline[linecolor=black, linewidth=0.018, arrowsize=0.05291667cm 2.0,arrowlength=0.8,arrowinset=0.2]{->}(0.0,-0.049999923)(0.6,-0.049999923)
\psline[linecolor=black, linewidth=0.018, arrowsize=0.05291667cm 2.0,arrowlength=0.8,arrowinset=0.2]{->}(1.2,1.5500001)(1.6,1.1500001)
\psarc[linecolor=black, linewidth=0.04, linestyle=dotted, dotsep=0.10583334cm, dimen=outer](1.4,-0.4999999){0.35}{180.0}{230.0}
\end{pspicture}
\end{figure}
\noindent where, from top to bottom, the incoming legs of this tensor correspond to $x_1,\dots, x_k$. 
\medskip

Now, take the tensor product of all the $\a$-tensors according to the ordering of $\aa$. In other words, stack the above $\a$-tensors one over another in such a way that, from top to bottom, they are ordered as in $\aa$. Do the same for the $\b$-tensors, now following the ordering of $\bb$. Since each outcoming leg of an $\a$-tensor corresponds to a unique crossing $x$, as well as to a unique incoming leg of a $\b$-tensor, we can compose all the outcoming legs of the above ordered $\a$-tensor with all the incoming legs of the $\b$-tensors along the tensors corresponding to the crossings. The result is a scalar in the base field of $H$.


\begin{definition}
Let $\HH=\HDD$ be an ordered, oriented, $\bb$-based extended Heegaard diagram of $(M,\c)$ and $\rho:\pi_1(M)\to\Aut(H)$ a group homomorphism. We denote by $$Z_H^{\rho}(\HH)\in\kk$$ the contraction of the tensors defined above. We will also denote by $K_H^{\rho}(\HH)$ the tensor $H^{\ot d}\to H^{\ot d}$ obtained by contracting the tensors above (here $d=|\aa|=|\bb|$), but without the cointegrals and integrals, so that $$Z_H^{\rho}(\HH)=\int^{\ot d}(K_H^{\rho}(\HH)(\coint^{\ot d})).$$
Note that $\HH$ needs to be $\aa$-based as well for $K_H^{\rho}(\HH)$ to be defined (as we need to know at which crossing of a given $\a$ we start applying the coproduct).
\end{definition}
\medskip

The scalar $Z_H^{\rho}(\HH)$ is not always a topological invariant of $(M,\c,\rho)$, but it is so provided $\Im(\rho)\sb \Ker(\lH)$ and the cointegral of $H$ has degree zero (as we will see later). If $\Im(\rho)$ is not contained in $\Ker(\lH)$, then $Z_H^{\rho}(\HH)$ has a $\pm\Im(\lH\circ\rho)$-valued indeterminacy coming from the choice of basepoints of $\bb$ as well as the orientations of the alpha curves. This is analogous to what happens to the torsion as we saw in Subsection \ref{subs: normalizing tau via Spinc}. For instance, let $q,q'$ be two basepoints on a curve $\b\in\bb$ and let $c\sb\b$ the oriented arc from $q$ to $q'$. If $\ov{\b}_x,\ov{\b}_x'$ denote the above Fox calculus elements starting from $q$ and $q'$ respectively, then $$\ov{\b}_x=\ov{c}\cdot\ov{\b}'_x$$ for each crossing $x$ through $\b$. Since $\int\circ\rho(\ov{c})=\lH(\rho(\ov{c}))\int$, it follows that the tensor associated to $\b$ starting from $q$ is $\lH(\rho(\ov{c}))$ times the tensor associated to $\b$ but starting from $q'$. Hence $Z_H^{\rho}(\HH)=\lH(\rho(\ov{c}))Z_H^{\rho}(\HH')$
where $\HH'$ is the ordered, oriented, based extended Heegaard diagram obtained by moving the basepoint $q$ to $q'$. A similar indeterminacy appears when reversing the orientation of an $\a$-curve. When the cointegral has degree one, then $Z_H^{\rho}(\HH)$ has a sign indeterminacy coming from the choice of ordering and orientations of $\aa\cup\bb$.

\begin{example}
\label{example: left trefoil}
Let $K\sb S^3$ be the left trefoil knot and let $(M,\c)$ be the associated sutured 3-manifold, that is, $M=S^3\sm N(K)$ and $\c$ consists of a pair of oppositely oriented meridians. Consider the (oriented, based) extended Heegaard diagram of $(M,\c)$ of Figure \ref{fig: dual curves in pi1} (left). For simplicity of notation, we denote $\a^*,a^*\in\pi_1(M)$ just by $\a,a$. If $x_1,x_2,x_3$ are the points of $\b\cap\a$, encountered as one follows the orientation of $\b$ starting from the given basepoint, then $\ov{\b}_{x_1}=a, \ov{\b}_{x_2}=a\a a^{-1}\a^{-1}$ and $\ov{\b}_{x_3}=a\a a^{-1} \a^{-1}a^{-1}$. If we suppose $\rho$ abelian, then the above contraction of tensors is given as follows:
\begin{figure}[H]
\centering
\begin{pspicture}(0,-0.975)(6.08,0.975)
\psline[linecolor=black, linewidth=0.018, arrowsize=0.05291667cm 2.0,arrowlength=0.8,arrowinset=0.2]{->}(0.4,-0.015)(0.8,-0.015)
\rput[bl](0.0,-0.115){$\coint$}
\rput[bl](1.0,-0.115){$\Delta$}
\rput[bl](2.85,-0.115){$S$}
\rput[bl](2.65,0.625){$\rho(a)$}
\psline[linecolor=black, linewidth=0.018, arrowsize=0.05291667cm 2.0,arrowlength=0.8,arrowinset=0.2]{->}(1.4,-0.015)(2.7,-0.015)
\psline[linecolor=black, linewidth=0.018, arrowsize=0.05291667cm 2.0,arrowlength=0.8,arrowinset=0.2]{->}(3.2,-0.015)(4.6,-0.015)
\rput[bl](2.45,-0.975){$\rho(a^{-1})$}
\psline[linecolor=black, linewidth=0.018, arrowsize=0.05291667cm 2.0,arrowlength=0.8,arrowinset=0.2]{->}(5.3,-0.015)(5.7,-0.015)
\rput[bl](4.8,-0.115){$m$}
\rput[bl](5.9,-0.145){$\int$}
\psbezier[linecolor=black, linewidth=0.018, arrowsize=0.05291667cm 2.0,arrowlength=0.8,arrowinset=0.2]{->}(1.4,0.185)(1.5833334,0.585)(2.1333334,0.785)(2.5,0.7850000000000046)
\psbezier[linecolor=black, linewidth=0.018, arrowsize=0.05291667cm 2.0,arrowlength=0.8,arrowinset=0.2]{->}(1.4,-0.215)(1.6,-0.615)(2.0,-0.815)(2.4,-0.815)
\psbezier[linecolor=black, linewidth=0.018, arrowsize=0.05291667cm 2.0,arrowlength=0.8,arrowinset=0.2]{<-}(4.6,0.185)(4.4,0.585)(3.8,0.785)(3.4,0.785)
\psbezier[linecolor=black, linewidth=0.018, arrowsize=0.05291667cm 2.0,arrowlength=0.8,arrowinset=0.2]{<-}(4.6,-0.215)(4.4,-0.615)(4.0,-0.815)(3.6,-0.815)
\end{pspicture}

\end{figure}

\noindent If $H=\kk[X]/X^2$ is an exterior Hopf algebra, then $\rho(a)\in\Aut(H)$ has the form $\rho(a)(X)=tX$ for some $t\in\kk^{\t}$. Using that $\coint=X$ and $\int(X^j)=\d_{j,1}$ the above contraction of Hopf algebra tensors reduces to
\medskip
 $$\int(\rho(a)(X)+S(X)+\rho(a^{-1})(X))=t-1+t^{-1}$$
 
\medskip
 
\noindent which is the Alexander polynomial of the left trefoil. 
\end{example}


\begin{remark}
\label{remark: twisted Kup at trivial rho}
If $\HH=\HD$ is a Heegaard diagram of a closed 3-manifold $Y$ in the usual sense, so $\S$ is a closed surface, then deleting a small disk to $\S$ defines a sutured Heegaard diagram $\HH_0$ of the associated sutured 3-manifold $(M_0=Y\sm B^3,\c_0=S^1\sb \p B)$, where $B\sb Y$ is an embedded 3-ball. In this situation, the (involutory) formulas of Kuperberg \cite{Kup1} are the $\rho\equiv 1$ case of the formulas here, so we get
\begin{align*}
Z_H^{\rho\equiv 1}(\HH_0)=I_H^{Kup}(Y).
\end{align*}
Similarly, if $\HH_K$ is a Heegaard diagram of the sutured complement of a knot $K$ in a closed 3-manifold $Y$, then computing $Z_H^{\rho}(\HH_K)$ at $\rho\equiv 1$ has the effect of forgetting the sutures. But doing this results in a Heegaard diagram of the underlying 3-manifold, so
\begin{align*}
Z_H^{\rho\equiv 1}(\HH_K)=I_H^{Kup}(Y)
\end{align*}
as well.
\end{remark}

\subsection{Definition of $I_H^{\rho}(M,\c,\ss,\o)$}\label{subs: general definition w Spinc and hom or} Fix a $\Spinc$ structure $\ss$ on $(M,\c)$ and an orientation $\o$ of the vector space $H_*(M,R_-(\c);\R)$. We now explain how to remove the indeterminacies of $Z_H^{\rho}(\HH)$ (when $R_-(\c)$ is connected) and define the invariant $I_H^{\rho}(M,\c,\ss,\o)$ for arbitrary $(M,\c,\rho,\ss,\o)$. That this is an invariant is a consequence of Theorem \ref{Theorem: I for sd product is Fox calculus}, but we emphasize that the arguments are a generalization of those of Subsection \ref{subs: normalizing tau via Spinc} for the torsion.
\medskip

Let $\HH$ be an ordered, oriented, extended Heegaard diagram of $(M,\c)$ (with connected $R_-(\c)$). Let $\x\in\Tab$ be a multipoint, and suppose $\HH$ is $\bb$-based via $\x$ using the convention of Definition \ref{def: basepoints from multipoint}. We note by $\ZHrho(\HH,\x)$ the tensor $\ZHrho(\HH)$ where $\HH$ has the basepoints on $\bb$ induced from $\x$. The multipoint $\x$ defines a $\Spinc$ structure $s(\x)$ as in Subsection \ref{subs: Spinc}, and comparing this with our chosen $\ss\in\Spinc(M,\c)$ we get an homology class $$f_{\ss,\x}\eq PD[s(\x)-\ss]\in H_1(M).$$
The composition $\lH\circ\rho:\pi_1(M,p)\to \kk^{\t}$ descends to $H_1(M)$ and so can be evaluated over $f_{\ss,\x}$. Then the scalar $$\lH\circ\rho(f_{\ss,\x})\cdot\ZHrho(\HH,\x)\in\kk$$ turns out to be independent of the multipoint chosen. The argument is the same as in Subsection \ref{subs: normalizing tau via Spinc}, only that $\det$ is replaced by $\lH$. Moreover, this scalar depends on the orientations of the closed curves only up to a sign. To correct this sign indeterminacy (along with the sign indeterminacy coming from the ordering of the closed curves) we use the homology orientation in the same way as we did for the torsion: the ordering and orientation of $\HH$ induces a homology orientation $o(\HH)$ and so there is a sign $\d$ defined by $$o(\HH)=\d\o.$$ 
Then
$$\d^{|\coint|}\cdot\lH\circ\rho(f_{\ss,\x})\cdot\ZHrho(\HH,\x)$$
is independent of the ordering, orientations and the multipoint $\x$ chosen. Then one has to show that this formula is invariant under extended Heegaard moves, so it defines a topological invariant of the tuple $(M,\c,\rho,\ss,\o)$. 
This is a consequence of \cite{LN:Kup} via Theorem \ref{Theorem: I for sd product is Fox calculus}.

\begin{definition}
Let $(M,\c)$ be a balanced sutured manifold with connected $R_-(\c)$. We denote the above invariant by $I_H^{\rho}(M,\c,\ss,\o)$, that is,
\begin{align*}
I_H^{\rho}(M,\c,\ss,\o)\eq \d^{|\coint|}\cdot\lH\circ\rho(f_{\ss,\x})\cdot\ZHrho(\HH,\x)\in\kk
\end{align*}
where $\HH$ is any ordered, oriented, extended Heegaard diagram of $(M,\c)$ with basepoints coming from a multipoint $\x$ of $\HH$ and $f_{\ss,\x}\in H_1(M)$ and $\d\in \{\pm 1\}$ are defined as above.
\end{definition}

Whenever $(M,\c)$ has disconnected $R_-(\c)$ we proceed as in Subsection \ref{subs: the disconnected case}. Thus, we take any balanced sutured manifold $(M',\c')$ obtained by adding one-handles to $M$ along $\c$ in such a way that $R_-(\c')$ is connected and set $$I_H^{\rho}(M,\c,\ss,\o)\eq I_H^{\rho'}(M',\c',i(\ss),\o')$$ where $\rho':\pi_1(M')\to\Aut(H)$ is any homomorphism such that $\rho'\circ j_*=\rho$, where $j_*:\pi_1(M)\to\pi_1(M')$ is induced by inclusion and $i(\ss),\o'$ are as in Subsection \ref{subs: the disconnected case}. Again, this is well-defined as a consequence of our previous work via Theorem \ref{Theorem: I for sd product is Fox calculus}, see \cite[Proposition 4.7]{LN:Kup}.


\begin{remark}
\label{remark: Spinc dependence of I}
The invariant $I_H^{\rho}$ depends on the $\Spinc$ structure on $(M,\c)$ as follows: for any $h\in H^2(M,\p M)$
\begin{align*}
I^{\rho}_{H}(M,\c,\ss+h,\o)=\lH\circ\rho(PD[h])^{-1}\cdot I^{\rho}_{H}(M,\c,\ss,\o)
\end{align*}
where $PD:H^2(M,\p M)\to H_1(M)$ is Poincar\'e duality. This follows from the definitions using that $PD[s(\x)-(\ss+h)]=PD[s(\x)-\ss]PD[h]^{-1}$.
\end{remark}

\def\FM{H_1(M)}\def\rhoc{\rho\ot h}\def\HM{H_M}\def\kkM{\kk[H_1(M)]}

\subsection{Twisted Kuperberg polynomials}\label{subs: twisted Kup polynomials} Suppose now that $H$ is $\N$-graded (together with the previous hypothesis as well). We denote the $\N$-degree of $H$ by $|\cdot|_0$. We explain now how to extend our invariant $I_H^{\rho}$ to $I_H^{\rho\ot h}\in\kk[H_1(M)]$ in a way that mimics the passage from the torsion $\tau^{\rho}$ to $\tau^{\rho\ot h}$.

\medskip
Let $(M,\c)$ be a balanced sutured manifold and set $\HM\eq H\ot_{\kk}\kk[H_1(M)]$. Recall that for any $\a\in \Aut(H)$ and $f\in H_1(M)$ there is a $\kk[H_1(M)]$-linear Hopf automorphism $\a\ot f$ of $\HM$ defined by
\begin{align*}
\a\ot f(h\ot f')\eq \a(h)\ot (f^{|h|_0}\cdot f')
\end{align*}
where $h\in H$ is homogeneous and $f'\in \kk[H_1(M)]$, see (\ref{eq: alpha tensor f}). Then, a homomorphism $\rho:\pi_1(M,p)\to \Aut(H)$ can be combined with the Hurewicz map $h:\pi_1(M)\to H_1(M;\Z)$ to define a homomorphism
\begin{align*}
\rho\ot h:\pi_1(M,p)&\to \Aut(\HM)\\
\d & \mapsto  \rho(\d)\ot h(\d) .
\end{align*}
Note that though $\kk[H_1(M)]$ is not even necessarily a domain, $\HM$ has a two-sided cointegral and integral induced from that of $H$ (given by $\coint_{\HM}\eq\coint\ot 1$ and $\int_{\HM}\eq\int\ot\id_{\kk[H_1(M)]}$). 
Thus, Theorem \ref{Theorem: I for sd product is Fox calculus} holds for the $\kkM$-linear Hopf superalgebra $\HM$, and we get an invariant $$I^{\rhoc}_{\HM}(M,\c,\ss,\o)\eq \d^{|\coint|}\cdot r_{\HM}\circ\rhoc (f_{\ss,\x})\cdot Z_{\HM}^{\rhoc}(\HH,\x)\in \kkM,$$
where $f_{\ss,\x}=PD[s(\x)-\ss]\in H_1(M)$ as usual. For simplicity of notation, we denote this just by $I_H^{\rhoc}(M,\c,\ss,\o)$ and denote $Z_{\HM}^{\rhoc}$ just by $Z_H^{\rhoc}$ as well. Here $Z_H^{\rho\ot h}(\HH)$ is obtained by contracting the same tensors associated to the $\a$'s and the crossings as before (since $\coint_{\HM}=\coint\ot 1$), only the $\b$-tensors contribute $H_1(M)$-coefficients coming from $\rho\ot h$. Since we are only considering automorphisms of $\HM$ of the form $\a\ot f$, by Lemma \ref{lemma: lHM homomorphism} we can write $I_H^{\rho\ot h}$ more explicitly as
$$I_H^{\rhoc}(M,\c,\ss,\o)=\d^{|\coint|}\cdot \lH\circ\rho(f_{\ss,\x})\cdot f_{\ss,\x}^{|\coint|_0}\cdot Z_H^{\rho\ot h}(\HH,\x).$$
Suppose now that $(M,\c)$ has disconnected $R=R_-(\c)$. Then our definition of $I_H^{\rhoc}$ is $$I_H^{\rho\ot h}(M,\c,\ss,\o)\eq I_H^{\rho'\ot h'}(M',\c',i(\ss),\o')$$ where $(M',\c')$ is a sutured manifold obtained by adding one-handles to $\c$ so that $R'=R_-(\c')$ is connected, $\rho':\pi_1(M')\to\Aut(H)$ is such that 
$\rho'\circ j_*=\rho$ where $j_*:\pi_1(M)\to\pi_1(M')$ is induced by inclusion, $i(\ss),\o'$ are as in Subsection \ref{subs: the disconnected case} and $h':\pi_1(M')\to H_1(M')$ is the Hurewicz map of $M'$. However, the right hand side belongs to $\kk[H_1(M')]$ and $H_1(M')=H_1(M)\oplus \Z^m$ where $m$ is the number of one-handles attached to $\c$. 

\begin{lemma}
\label{lemma: disconnected case twisted poly belongs to H1M}
Let $(M,\c)$ be a connected balanced sutured manifold with disconnected $R_-(\c)$ and let $(M',\c'),\rho',i(\ss),\o',h'$ be as above. Then $I_H^{\rho'\ot h'}(M',\c',i(\ss),\o')\in\kk[H_1(M)].$
\end{lemma}
\begin{proof}
Let $R_0,\dots,R_{m}$ be the components of $R$. It suffices to prove the statement for a single $(M',\c')$ which we suppose is obtained by attaching a single one-handle between $R _0\t I$ and $R_{i}\t I$ for each $i=1,\dots,m$, where $I=[-1,1]$. We will show the lemma by finding an appropriate oriented, based Heegaard diagram of $(M',\c')$ for which $h'(\ov{\b}_x)\in H_1(M)$ for each $\b$ and each crossing $x$. To achieve this, let $\HH=\HDD$ be an extended Heegaard diagram of $(M,\c)$ where as usual $\aaa=\aa\cup\bolda$ and $\aa=\{\a_1,\dots,\a_d\}$. For each $i=1,\dots, d$ let $h_i$ be the one-handle attached to $R\t I$ associated to $\a_i$, so $h_i$ has belt circle $\a_i$. Since $M$ is connected, after sufficiently many handleslidings among the $h_i$'s, we can suppose that $h_i$ has one feet in $R_0\t \{1\}$ and the other in $R_{i}\t\{1\}$ for each $i=1,\dots,m$ and that $h_i$ is attached to $R_0\t\{1\}$ for each $i>m$. Thus, $\S\sm (\a_1\cup\dots\cup\a_m)$ has $m+1$ components $\S_0,\dots,\S_m$, where for each $i=1,\dots,m$, $\a_i$ has one side on $\S_0$ and the other on $\S_i$. Note that the $\a^*_i$ with $i>m$ together with the $\a^*,\a\in\bolda$ generate $H_1(M)$ while $\a^*_1,\dots,\a^*_m$ generate the $\Z^m$ factor in $H_1(M')=H_1(M)\oplus\Z^m$. We will give $\a_i$ the orientation induced from $\S_i$ for each $i=1,\dots,m$ and an arbitrary orientation for $i>m$. Let $\b\in\bb$ with an arbitrary orientation. By our choice of orientation of the $\a\in\aa$, if a crossing $x\in\a_i\cap\b$ is positive, then $\b$ is oriented from $\S_0$ to $\S_i$ near $x$ and oppositely if the crossing is negative. Moreover, since $\S\sm\a_i$ is disconnected and since $\S_i$ contains no $\a\in\aa$, if $\b$ intersects $\a_i$ positively at $x$, then the next intersection of $\b$ with $\aa$ (following the orientation of $\b$) is a negative intersection at $\a_i$. Now, suppose the basepoint of $\b$ lies in $\S_0$. Then, when following the orientation of $\b$, the first intersection of $\b$ with $\a_1\cup\dots\cup\a_m$ is positive, and each positive intersection with the same set is followed by a negative one (possibly after several intersections with $\bolda$). Hence, the appearances of $\a^*_1,\dots,\a^*_m$ in each $\ov{\b}_x$ cancel out in homology, so $h'(\ov{\b}_x)\in H_1(M)$ for each crossing $x$ through $\b$. In general, the basepoints of $\bb$ come from a multipoint $\x=\{x_1,\dots,x_d\}$ of $\HH$, say with $x_i\in\a_i\cap\b_i$ for each $i$, but with our choice of orientations of $\a_1,\dots,\a_m$ the induced basepoints on $\bb$ indeed lie in $\S_0$. Thus $h'(\ov{\b}_x)\in H_1(M)$ for each $\b\in\bb$ and each crossing $x$. Therefore $Z_H^{\rho'\ot h'}(\HH',\x)\in\kk[H_1(M)]$ where $\HH'$ is the Heegaard diagram of $(M',\c')$ obtained from $\HH$ by gluing the corresponding 2-dimensional one-handles to $\S$. On the other hand, the vertical extension map $i:\Spinc(M,\c)\to\Spinc(M',\c')$ satisfies $s'(\x)=i(s(\x))$ where $s'$ is the map from multipoints to $\Spinc$ of $M'$ so $$PD[s'(\x)-i(\ss)]=j_*(PD[s(\x)-\ss])$$ lies in the image of the map $j_*:H_1(M)\to H_1(M')$ induced by inclusion. It follows that $I_H^{\rho'\ot h'}(M',\c',i(\ss),\o')\in\kk[H_1(M)]$ as desired.

\end{proof}

\begin{definition}
Let $(M,\c)$ be an arbitrary balanced sutured 3-manifold and let $\rho:\pi_1(M,p)\to\Aut(H), \ss\in\Spinc(M,\c)$ and $\o$ be as usual and $h:\pi_1(M)\to H_1(M)$ be the Hurewicz map. We call $I^{\rhoc}_{H}(M,\c,\ss,\o)\in \kkM$ the {\em twisted Kuperberg polynomial} of $(M,\c)$ with respect to $\rho$.
\end{definition}


\begin{example}
Consider the left trefoil complement as in Example \ref{example: left trefoil} and let $\rho\equiv 1$. Let $t$ be a generator of $H_1(M)\cong\Z$. Then, for the diagram of Figure \ref{fig: dual curves in pi1}, the twisted Kuperberg polynomial is given by 
\begin{align*}
Z_H^{1\ot h}(\HH)=\int(t^{|\coint_{(1)}|_0}\coint_{(1)}\cdot S(\coint_{(2)})\cdot t^{-|\coint_{(3)}|_0}\coint_{(3)})
\end{align*}
where this time we have used Sweedler's notation for the coproduct, that is, we write $\De(x)=x_{(1)}\ot x_{(2)}$, $(\De\ot \id_H)\De(x)=x_{(1)}\ot x_{(2)}\ot x_{(3)}$, etc. for each $x\in H$. As before, this recovers the Alexander polynomial of the trefoil knot if $H$ is set to be an exterior algebra on one generator of degree one.
\end{example}

Recall that the augmentation map $\aug\colon \kkM\to\kk$ is the $\kk$-linear map defined by $\aug(f)=1$ for all $f\in H_1(M)$. Then we have the following.

\begin{proposition}
\label{prop: aug of twisted Kup poly}
The twisted Kuperberg polynomial satisfies
\begin{align*}
\aug(I^{\rhoc}_{H}(M,\c,\ss,\o))=I_H^{\rho}(M,\c,\ss,\o).
\end{align*}
\end{proposition}
\begin{proof}
Set $\aug_H\eq \id_H\ot\aug\colon\HM\to H$. It is easy to see that $\aug\circ\int_{\HM}=\int\circ \aug_H$ and $\aug_H\circ[(\rhoc)(\d)]\circ j_H=\rho(\d)$ for any $\d\in\pi_1(M)$, where $j_H:H\to\HM, x\mapsto x\ot 1$ is the inclusion. Since the cointegral of $\HM$ is $j_H(\coint)$, where $\coint$ is the cointegral of $H$, it follows that $\aug(Z^{\rhoc}_{\HM}(\HH))=\ZHrho(\HH)$ where $\HH$ is a Heegaard diagram of $(M,\c)$. From this the result follows.
\end{proof}

\def\lHM{r_{\HM}}

Now let $L$ be an ordered oriented $m$-component link in an homology sphere $Y$ and let $(M_L,\c_L)$ be the sutured manifold complement. The ordering and the orientation of $L$ allows to identify $\kk[H_1(M_L)]=\kk[t_1^{\pm 1},\dots, t_m^{\pm 1}]$ canonically.


\begin{corollary}
\label{corollary: multivariable link polynomial}
Let $L$ be an ordered oriented $m$-component link in an homology sphere $Y$. Then $I_H^{\rho\ot h}(M_L,\c_L,\ss,\o)$ is a multivariable polynomial invariant of $L$ belonging to $\kk[t_1^{\pm 1},\dots, t_m^{\pm 1}]$. This satisfies that $$\aug(I_H^{\rho\ot h}(M_L,\c_L,\ss,\o))=I_H^{\rho}(M_L,\c_L,\ss,\o).$$ 
In particular, for $\rho\equiv 1$ we have $$\aug(I_H^{h}(M_L,\c_L,\ss,\o))=I_H^{Kup}(Y).$$
\end{corollary}
\begin{proof}
That $I_H^{\rhoc}(M_L,\c_L,\ss,\o)$ belongs to $\kk[t_1^{\pm 1},\dots, t_m^{\pm 1}]$ is a consequence of Lemma \ref{lemma: disconnected case twisted poly belongs to H1M}. The rest follows from the above proposition together with Remark \ref{remark: twisted Kup at trivial rho}.
\end{proof}

\subsection{Reidemeister torsion as a Kuperberg invariant}\label{subs: computation of ZHn} We now show Theorem \ref{Theorem: I at exterior is torsion}. This follows essentially from Proposition \ref{theorem: Z via Fox calculus} below. So let $(M,\c)$ be a balanced sutured 3-manifold with connected $R_-(\c)$, $p\in s(\c)$ and let $\rho:\pi_1(M,p)\to GL(V)$ be a homomorphism, where $V$ is a finite dimensional vector space over a field $\kk$. Let $\HH=\HDD$ be an extended Heegaard diagram of $(M,\c)$ which is ordered, oriented and based. This determines a presentation of $\pi_1(M,p)$ as in (\ref{eq: pres of pi1}). As usual, let $\aaa=\aa\cup\bolda$, $d=|\aa|=|\bb|$ and $\aa=\{\a_1,\dots,\a_d\}$.
\medskip

\begin{proposition}
\label{theorem: Z via Fox calculus}
If $\La(V)$ is the exterior Hopf algebra over $V$, then
\begin{align*}
Z^{\rho}_{\LaV}(\HH)=\det\left(\rho\left(\frac{\p\ov{\b}_i}{\p\a^*_j}\right)\right)_{i,j=1,\dots,d}\in\kk.
\end{align*}
The same formula is valid for $\rho\ot h$ as well.
\end{proposition}

This follows from the following simple observations on exterior algebras: \begin{enumerate}
\item $\La(V)$ is $\N$-graded and super-commutative,
\item $\La(V)$ is a universal enveloping algebra, that is, algebra maps $\La(V)\to A$ into some superalgebra $A$ over $\kk$ are in one-to-one correspondence (via restriction) with Lie superalgebra maps $V\to A$ ($V$ has the trivial Lie algebra structure),
\item \label{eq: det} If $T:V\to V$ is a linear map, then the induced algebra morphism $\La(T):\La(V)\to\La(V)$ satisfies $$\det(T)=\int(\La(T)(\coint))$$
where $\int$ and $\coint$ denote the integral and cointegral of $\La(V)$ normalized by $\int(\coint)=1$.
\end{enumerate}

\begin{proof}[Proof of Proposition \ref{theorem: Z via Fox calculus}]
Recall that $Z_{\La(V)}^{\rho}(\HH)$ is defined as $$Z_{\La(V)}^{\rho}(\HH)=\int^{\ot d}(K_{\La(V)}^{\rho}(\HH)(\coint^{\ot d}))$$
where $K_{\La(V)}^{\rho}(\HH):\La(V)^{\ot d}\to \La(V)^{\ot d}$ is the twisted Kuperberg tensor. Note that $\La(V)^{\ot d}$ is naturally isomorphic to $\La(V^{\oplus d})$ and that $\coint^{\ot d},\int^{\ot d}$ are respectively a cointegral and integral for $\La(V^{\oplus d})$. It is easy to see that $K_{\La(V)}^{\rho}$ preserves the $\N$-grading of $\La(V)$, since it is a composition of degree-preserving maps. Now, by definition, $K_{\La(V)}^{\rho}$ is a composition of coproducts, antipodes, automorphisms of $\La(V)$ and the multiplication map. Since $\La(V)$ is commutative, all these are algebra maps, hence $K_{\La(V)}^{\rho}$ is also an algebra map. By the universal enveloping algebra property, $K_{\La(V)}^{\rho}$ is determined over $V^{\oplus d}$, and since it preserves the $\N$-grading, the restriction to $V^{\oplus d}$ is a linear map $T:V^{\oplus d}\to V^{\oplus d}$ so $K^{\rho}_{\La(V)}(\HH)=\La(T)$ and $$Z^{\rho}_{\La(V)}(\HH)=\det(T)$$ by property (\ref{eq: det}) above. Now, by definition of $K^{\rho}_{\La(V)}$, one easily sees that the map $T$ is represented by the matrix $(\rho(\p\ov{\b_i}/\p\a^*_j))$, which implies the proposition. The assertion for $\rho\ot h$ follows from its definition and the fact that $V$ is concentrated in degree one.

\end{proof}

\begin{proof}[Proof of Theorem \ref{Theorem: I at exterior is torsion}]
Suppose first that $R_-(\c)$ is connected.
Note that the inverse-transpose  satisfies 
$\rho^{-T}(\s(x))=\rho(x)^T$ for any $x\in\Z[\pi]$, where recall that $\s:\Z[\pi]\to\Z[\pi]$ is the map defined by $\s(g)=g^{-1}$ for $g\in\pi$. Recall also that $f_{\ss,\x}\eq PD[s(\x)-\ss]$, that $|\coint|=n \pmod 2$ if $\dim(V)=n$ and $\coint$ is the cointegral of $\La(V)$, and that $r_{\La(V)}=\det$. Thus, we get
\begin{align*}
I^{\rho}_{\LaV}(M,\c,\ss,\o)&\eq\d^{|\coint|}\cdot r_{\La(V)}\circ\rho(f_{\ss,\x})\cdot Z^{\rho}_{\LaV}(\HH,\x)\\
&=\d^n\cdot\det\circ\rho(f_{\ss,\x})\cdot\det(\rho(\p\ov{\b}_i/\p\a^*_j))\\
&=\d^n\cdot\det\circ\rho(f_{\ss,\x})\cdot\det(\rho(\p \ov{\b}_j/\p\a^*_i)^T)\\
&=\d^n\cdot\det\circ\rho(f_{\ss,\x})\cdot\det(\rho^{-T}(\s(\p \ov{\b}_j/\p\a^*_i)))\\
&=\d^n\cdot\det\circ\rho^{-T}(PD[\ss-s(\x)])\cdot\det(\rho^{-T}(\s(\p \ov{\b}_j/\p\a^*_i)))\\
&=\ovtau^{\rho^{-T}}(M,\c,\ss,\o)
\end{align*}
where we use Proposition \ref{theorem: Z via Fox calculus} in the second equality and the last equality is by definition of $\ovtau$ (which is itself based on Proposition \ref{prop: torsion via FOX calculus}). The case when $R_-(\c)$ is disconnected follows from the connected case by the definition of $\ovtau$, see Subsection \ref{subs: the disconnected case}. The assertion for $I^{\rhoc}_{\La(V)}$ follows from the above since clearly $\tau_0^{(\rhoc)^{-T}}=\tau_0^{\rho^{-T}\ot h^{-1}}=\s(\tau_0^{\rho^{-T}\ot h})$ for any $(M,\c,\ss,\o)$. 
\end{proof}

Note that if we had used the convention that $(g\cdot c)\ot v=c\ot (\rho(g)^t(v))$ for the tensor product $C_*(\wt{M})\ot_{\Z[\pi]}V$ (as in \cite{Porti:survey}), then $I^{\rho}_{\LaV}$ would be exactly $\ovtau^{\rho}$. Note also that Corollary \ref{Corollary: I at ext of link is twisted Alex poly} follows directly from the above theorem together with Corollary \ref{corollary: sutured torsion recovers TWISTED ALEX}.

\subsection{Going beyond Reidemeister torsion?}
\label{sect: going beyond R torsion}

We now address the question of whether invariants different than Reidemeister torsion can be found within the present framework. Recall that we suppose $H$ is an involutory Hopf superalgebra over a field $\kk$ whose cointegral and integral are two-sided. We will further assume that $\char(\kk)=0$. It turns out that the involutory condition imposes serious restrictions. 

\medskip

Suppose first that we seek for involutory Hopf algebras (i.e. without mod 2 grading). Then we are constrained by the following theorem of Larson-Radford \cite{LR:cosemisimplechar0, LR:semisimple}: a finite dimensional (ungraded) Hopf algebra over a field $\kk$ of characteristic zero is involutory if and only if it is semisimple and cosemisimple. As far as the author knows, all finite dimensional semisimple Hopf algebras over a field of characteristic zero are built from group algebras in some way. Therefore, one may expect that such Hopf algebras only lead to invariants related to a count of homomorphisms $\pi_1(M)\to G$. 

\medskip
Fortunately, Larson-Radford's theorem is no longer true for Hopf {\em superalgebras}, indeed, the opposite holds \cite[Corollary 3.1.2]{AEG:Triangular}: if $H=H_0\oplus H_1$ is a finite dimensional involutory Hopf superalgebra over a field of characteristic zero, then $H_1\neq 0$ if and only if $H$ is non-semisimple. So, suppose we seek for involutory Hopf superalgebras (with $H_1\neq 0$ and $\char(\kk)=0$). Suppose we add the stronger condition of cocommutativity (which immediately implies involutority). Then, we are constrained by the classical theorem of Cartier-Kostant-Milnor-Moore (see e.g. \cite[Theorem 2.3.4]{AEG:Triangular}), which 
says that $H$ is isomorphic, as a Hopf superalgebra, to a semidirect product:
\begin{align*}
H\cong\kk[G(H)]\ltimes U(P(H)).
\end{align*}
Here $G(H)$ is the group of {\em group-likes} of $H$, that is, the elements that satisfy $\De(g)=g\ot g$ and $\e(g)=1$, $P(H)$ is the Lie algebra of primitive elements, i.e. those $h\in H$ satisfying $\De(h)=1\ot h+h\ot 1$ and $U(P(H))$ is the universal enveloping algebra of $P(H)$. 
Now, if we wish $U(P(H))$ to be finite dimensional, then $P(H)$ is forced to be an abelian Lie algebra concentrated in degree one, so $U(P(H))$ is just an exterior algebra. Therefore, a finite dimensional cocommutative Hopf superalgebra (over a characteristic zero field) is necessarily of the form $\kk[G]\ltimes \La(V)$ for some finite subgroup $G\sb GL(V)$. Thus, we get nothing beyond torsion with such Hopf superalgebras.
\medskip

In view of the above, if one wishes to find an example of a finite dimensional involutory Hopf superalgebra considerably different from group algebras, exterior algebras and their semidirect products, then one has to look for non-commutative and non-cocommutative Hopf superalgebras. The quantum group $\Uqgl11$ at a root of unity (see e.g. \cite{Sartori:Alexander}) is such an example, for another one see \cite[Example 5.6]{KV:generalized}. However, the cointegrals and integrals of such are not two-sided, so they don't fit in the present framework. Still, these Hopf superalgebras seem to be obtained from exterior algebras somehow, for instance, $\Uqgl11$ (at a root of unity of order $n$) is the Drinfeld double of a semidirect product $\kk[\Z/n\Z]\ltimes\La(\kk)$. Therefore, it seems difficult to expect something beyond torsion with such examples.
\medskip

It has to be noted that more interesting involutory examples can be found in positive characteristic. For instance, for every restricted Lie algebra $\gg$ there is an associated restricted universal enveloping algebra $U^{res}(\gg)$ which is cocommutative (hence involutory) and is of finite dimension if $\gg$ is. Whether such Hopf algebras lead to interesting invariants will be the subject of future work.
\medskip

\section{Fox calculus via semidirect products}
\label{section: Fox calculus and semidirect products}

In this section we explain where the Fox calculus formula for $I_H^{\rho}(M,\c,\ss,\o)$ comes from. We begin in Subsection \ref{subs: invs from relative int} by briefly recalling the setting of \cite{LN:Kup}. In Subsection \ref{subs: Proof of Thm 1} we show that a semidirect product $\kk[\Aut(H)]\ltimes H$ fits into this setting and we prove Theorem \ref{Theorem: I for sd product is Fox calculus}. In Subsection \ref{subs: Kup for Hopf G-algebra} we give an explanation of our invariant in terms of the more common notion of a Hopf group-algebra, therefore relating it to work of Virelizier \cite{Virelizier:flat}. Furthermore, we discuss why we don't use more general Hopf group-algebras. Finally, in Subsection \ref{subs: rel int and Hopf G-alg}, we discuss the relation between both approaches. 

\subsection{Invariants from relative integrals}
\label{subs: invs from relative int}

\def\iotal{\iota_l}\def\iotar{\iota_r}\def\wtmu{\wt{\int}}\def\wtI{\wt{I}}\def\wtZ{\wt{Z}}


We begin by briefly recalling the setting of \cite{LN:Kup}. There, we also built a Kuperberg-style invariant of balanced sutured 3-manifolds, but with a more involved and general algebraic input. This consisted of a (possibly infinite-dimensional) involutory Hopf superalgebra $\wt{H}$ endowed with relative versions of the Hopf algebra cointegral and integral, satisfying certain conditions. More precisely, we defined a {\em right relative cointegral} as a tuple $(A,\pi_A,\iotar)$ where $A$ is a Hopf subalgebra of $\wt{H}$, $\pi_A:\wt{H}\to A$ is a cocentral Hopf morphism with $\pi_A|_A=\id_A$ (where cocentral means $(\pi_A\ot \id_{\wt{H}})\De_{\wt{H}}=(\pi_A\ot \id_{\wt{H}})\De^{op}_{\wt{H}}$) and $\iotar:A\to \wt{H}$ is a left $A$-comodule map (where $\wt{H}$ is a left $A$-comodule via $\pi_A$) that satisfies the following relative version of the right cointegral equation: 
\begin{figure}[H]
\centering
\begin{pspicture}(0,-0.765)(5.7,0.765)
\psline[linecolor=black, linewidth=0.018, arrowsize=0.05291667cm 2.0,arrowlength=0.8,arrowinset=0.2]{->}(2.1,0.335)(2.4,0.335)
\rput[bl](1.34,0.125){$m_{\widetilde{H}}$}
\psline[linecolor=black, linewidth=0.018, arrowsize=0.05291667cm 2.0,arrowlength=0.8,arrowinset=0.2]{->}(0.0,0.335)(0.3,0.335)
\psline[linecolor=black, linewidth=0.018, arrowsize=0.05291667cm 2.0,arrowlength=0.8,arrowinset=0.2]{<-}(1.6,0.035)(1.6,-0.265)
\rput[bl](2.8,0.095){=}
\rput[bl](0.46,0.265){$\iota_r$}
\psline[linecolor=black, linewidth=0.018, arrowsize=0.05291667cm 2.0,arrowlength=0.8,arrowinset=0.2]{->}(0.9,0.335)(1.2,0.335)
\rput[bl](3.84,0.525){$m_A$}
\psline[linecolor=black, linewidth=0.018, arrowsize=0.05291667cm 2.0,arrowlength=0.8,arrowinset=0.2]{->}(3.4,0.635)(3.7,0.635)
\psline[linecolor=black, linewidth=0.018, arrowsize=0.05291667cm 2.0,arrowlength=0.8,arrowinset=0.2]{->}(4.5,0.635)(4.8,0.635)
\rput[bl](4.96,0.565){$\iota_r$}
\psline[linecolor=black, linewidth=0.018, arrowsize=0.05291667cm 2.0,arrowlength=0.8,arrowinset=0.2]{->}(5.4,0.635)(5.7,0.635)
\psline[linecolor=black, linewidth=0.018, arrowsize=0.05291667cm 2.0,arrowlength=0.8,arrowinset=0.2]{<-}(4.0,-0.465)(4.0,-0.765)
\psline[linecolor=black, linewidth=0.018, arrowsize=0.05291667cm 2.0,arrowlength=0.8,arrowinset=0.2]{<-}(4.0,0.435)(4.0,0.135)
\rput[bl](3.81,-0.265){$\pi_A$}
\rput[bl](5.65,0.085){.}
\end{pspicture}

\end{figure}
\noindent Moreover, $\wt{H}$ is endowed with a right relative integral, which is the notion dual to the above one, relative to a central Hopf subalgebra $B\sb \wt{H}$. Here we will only consider the case $B=\kk$, so a relative integral is just an ordinary integral $\wtmu:\wt{H}\to\kk$. We further required the existence of two distinguished group-likes $a^*\in \Hom_{alg}(A,\kk)$ and $b\in G(B)$, so $b=1$ if we assume $B=\kk$. The group-like $a^*$ must satisfy 
\begin{align*}
\De_{a^*}\circ\iotar&=(-1)^{|\iota|}S_{\wt{H}}\circ\iota_r\circ S_A & \text{ and } & &
\wtmu\circ m_{\wt{H}}^{op}&=\wtmu\circ m_{\wt{H}}\circ(\id_{\wt{H}}\ot\De_{a^*}),
\end{align*}
where $\De_{a^*}:\wt{H}\to\wt{H}$ is given by $\De_{a^*}\eq(a^*\pi_A\ot \id_{\wt{H}})\circ\De_{\wt{H}}$. Since $b=1$, the dual condition becomes 
\begin{align*}
\wtmu\circ S_{\wt{H}}&=(-1)^{|\wtmu|}\wtmu & \text{ and } & & \De_{\wt{H}}^{op}\circ\iota_r&=\De_{\wt{H}}\circ\iota.
\end{align*}
In addition, we assume $\wtmu(\iota_r(1_A))=1$.
\medskip

Using such structures we defined a topological invariant $\wtI_{\wt{H}}^{\rho}(M,\c,\ss,\o)$ where $(M,\c),\ss,\o$ are as in the present paper and $\rho$ is a homomorphism from $H_1(M)$ into the group of group-likes $G(A\ot B^*)$ ($=G(A)$ since here we assume $B=\kk$). The construction also holds for arbitrary homomorphisms $\rho:\pi_1(M)\to G(A)$, the necessary extra details will be carried out in Appendix \ref{subs: non-abelian case} (we will do this only for the semidirect product case). When $R_-(\c)$ is connected, this invariant was built as follows: pick an ordered, oriented, based extended Heegaard diagram $\HH=(\S,\aaa,\bb)$ of $(M,\c)$. Now we require that $\HH$ is both $\aa$-based and $\bb$-based. Then to each $\a\in\aaa$ we associate the tensor


\begin{figure}[H]
\centering

\begin{pspicture}(0,-0.5786306)(9.77,0.5786306)
\psline[linecolor=black, linewidth=0.018, arrowsize=0.05291667cm 2.0,arrowlength=0.8,arrowinset=0.2]{->}(2.02,0.0213694)(2.42,0.0213694)
\rput[bl](2.59,-0.1586306){$\Delta_{\wt{H}}$}
\psline[linecolor=black, linewidth=0.018, arrowsize=0.05291667cm 2.0,arrowlength=0.8,arrowinset=0.2]{->}(3.22,0.2213694)(3.62,0.6213694)
\psline[linecolor=black, linewidth=0.018, arrowsize=0.05291667cm 2.0,arrowlength=0.8,arrowinset=0.2]{->}(3.22,-0.1786306)(3.62,-0.5786306)
\rput[bl](3.48,-0.2186306){$\vdots$}
\rput[bl](0.0,-0.1586306){$\rho(\alpha^*)$}
\psline[linecolor=black, linewidth=0.018, arrowsize=0.05291667cm 2.0,arrowlength=0.8,arrowinset=0.2]{->}(3.22,0.1213694)(3.62,0.3213694)
\psline[linecolor=black, linewidth=0.018, arrowsize=0.05291667cm 2.0,arrowlength=0.8,arrowinset=0.2]{->}(1.02,0.0213694)(1.42,0.0213694)
\rput[bl](1.66,-0.1586306){$\iotar$}
\psline[linecolor=black, linewidth=0.018, arrowsize=0.05291667cm 2.0,arrowlength=0.8,arrowinset=0.2]{->}(7.82,0.0213694)(8.22,0.0213694)
\rput[bl](8.4,-0.1586306){$\Delta_{\wt{H}}$}
\psline[linecolor=black, linewidth=0.018, arrowsize=0.05291667cm 2.0,arrowlength=0.8,arrowinset=0.2]{->}(9.02,0.2213694)(9.42,0.6213694)
\psline[linecolor=black, linewidth=0.018, arrowsize=0.05291667cm 2.0,arrowlength=0.8,arrowinset=0.2]{->}(9.02,-0.1786306)(9.42,-0.5786306)
\rput[bl](9.28,-0.2186306){$\vdots$}
\rput[bl](5.8,-0.1586306){$\rho(\alpha^*)$}
\psline[linecolor=black, linewidth=0.018, arrowsize=0.05291667cm 2.0,arrowlength=0.8,arrowinset=0.2]{->}(9.02,0.1213694)(9.42,0.3213694)
\psline[linecolor=black, linewidth=0.018, arrowsize=0.05291667cm 2.0,arrowlength=0.8,arrowinset=0.2]{->}(6.82,0.0213694)(7.22,0.0213694)
\rput[bl](7.36,-0.1586306){$i_A$}
\rput[bl](4.62,-0.0786306){\text{or}}
\end{pspicture}
\end{figure}
\noindent depending on whether $\a$ is a closed curve or an arc. Here $i_A:A\to\wt{H}$ is the inclusion and the outcoming legs of the iterated coproduct correspond to the crossings through $\a$, starting from its basepoint and following its orientation. Since we consider $B=\kk$, to each $\b\in\bb$ we associate the tensor

\begin{figure}[H]
\centering

\begin{pspicture}(0,-0.6091925)(1.759192,0.6091925)
\psline[linecolor=black, linewidth=0.018, arrowsize=0.05291667cm 2.0,arrowlength=0.8,arrowinset=0.2]{->}(0.009192505,-0.6)(0.4091925,-0.2)
\psline[linecolor=black, linewidth=0.018, arrowsize=0.05291667cm 2.0,arrowlength=0.8,arrowinset=0.2]{->}(0.009192505,0.6)(0.4091925,0.2)
\rput[bl](0.06919251,-0.24){$\vdots$}
\rput[bl](0.5191925,-0.13){$m_{\wt{H}}$}
\psline[linecolor=black, linewidth=0.018, arrowsize=0.05291667cm 2.0,arrowlength=0.8,arrowinset=0.2]{->}(1.2091925,0.0)(1.6091925,0.0)
\psline[linecolor=black, linewidth=0.018, arrowsize=0.05291667cm 2.0,arrowlength=0.8,arrowinset=0.2]{->}(0.009192505,0.3)(0.4091925,0.1)
\rput[bl](1.7991925,-0.13){$\wtmu$}

\end{pspicture}
\end{figure}

\noindent where the incoming legs correspond to the crossings through $\b$ as usual. Finally, to each crossing $x\in\aaa\cap\bb$ we associate $S_{\wt{H}}^{\e_x}$ (where, as before, $\e_x\in\{0,1\}$ is defined by $(-1)^{\e_x}=m_x$ and $m_x$ is the intersection sign at $x$). The contraction of all these tensors is denoted by $\wtZ_{\wt{H}}^{\rho}(\HH)\in\kk$ or $\wtZ_{\wt{H}}^{\rho}(\HH,\x)$ if the basepoints of $\aa$ and $\bb$ come from a multipoint $\x$. Here the rule of Definition \ref{def: basepoints from multipoint} is extended in an obvious way to $\aa$: if the crossing $x_i$ is positive (resp. negative), we put a basepoint on $\a_i$ just before (resp. after) $x_i$ when following the orientation of $\a_i$. The invariant $\wtI_{\wt{H}}^{\rho}(M,\c,\ss,\o)$ of \cite{LN:Kup} is then defined by 
\begin{align*}
\wtI_{\wt{H}}^{\rho}(M,\c,\ss,\o)=\d^{|\iota|}\cdot\lb a^*,\rho(PD[s(\x)-\ss])\rb\cdot \wtZ_{\wt{H}}^{\rho}(\HH,\x)
\end{align*}
where $\d$ is the sign characterized by $o(\HH)=\d\o$ as before and $a^*\in\Hom_{alg}(A,\kk)$ is the above distinguished group-like. This definition extends to the case of disconnected $R_-(\c)$ by using $(M',\c'),\rho',i(\ss),\o'$ as in Subsection \ref{subs: the disconnected case}.

\medskip

\subsection{Proof of Theorem \ref{Theorem: I for sd product is Fox calculus}}\label{subs: Proof of Thm 1} We begin by showing that semidirect products fit into the preceding framework, hence they define invariants of balanced sutured 3-manifolds.


\begin{lemma}
Let $H$ be a finite dimensional Hopf superalgebra with a two-sided cointegral $\coint$ and a two-sided integral $\int$, normalized by $\int(\coint)=1$, and let $A=\kk[\Aut(H)]$. Then $\wt{H}=A\ltimes H$ has a right relative cointegral over $A$ given by $\pi_A=\id_A\ot\e_H$ and $\iotar(\a)\eq \coint\cdot\a=\lH(\a)^{-1}\a\ot\coint$ for each $\a\in \Aut(H)$. The distinguished group-like $a^*\in G(A)$ is given by $a^*(\a)\eq \lH(\a)$ for each $\a\in\Aut(H)$. The integral is given by $\wt{\int}(\a\ot h)=\d_{\a,\id_H}\int(h)$ for any $\a\in\Aut(H)$ and $h\in H$.
\end{lemma}

\begin{proof}
That $\iota_r$ is a relative cointegral and $\wtmu$ is an integral is easy to check. That these satisfy the conditions stated in Subsection \ref{subs: invs from relative int} follows from the fact that the cointegral and integral of $H$ are assumed to be two-sided. For instance, we have $$\De_{a^*}\circ\iota_r(\a)=a^*(\pi_A(\coint_{(1)}\cdot\a))\coint_{(2)}\cdot\a=\lH(\a)\e(\coint_{(1)})\coint_{(2)}\cdot\a=\lH(\a)\coint\cdot\a$$
while $$(-1)^{|\iota_r|}S_{\wt{H}}\iota S_A(\a)=(-1)^{|\coint|}S_{\wt{H}}(\coint\cdot\a^{-1})=(-1)^{|\coint|}\a\cdot S(\coint)=\a\cdot \coint=\lH(\a)\coint\cdot\a$$
where we used that $S(\coint)=(-1)^{|\coint|}\coint$, which is equivalent to say that $\coint$ is two-sided.
\end{proof}

\begin{proof}[Proof of Theorem \ref{Theorem: I for sd product is Fox calculus}]
Let $\wt{H}=\kk[\Aut(H)]\ltimes H$. We begin by showing that
\begin{align*}
\wtZ_{\wt{H}}^{\rho}(\HH)=Z_H^{\rho}(\HH)
\end{align*}
for any ordered, oriented, based extended Heegaard diagram of a balanced sutured manifold $(M,\c)$ (with connected $R_-(\c)$), where the right hand side is the Fox calculus tensor of Subsection \ref{subs: tensors and HDs}. This follows by writing the contraction of the tensors of $\wt{H}$ above in terms of the tensors of $H$. Indeed, for the semidirect product Hopf algebra $\wt{H}$, the $\wt{H}$-tensor associated to a closed $\a$ is 

\begin{figure}[H]
\centering
\begin{pspicture}(0,-1.08)(8.77,1.08)
\psline[linecolor=black, linewidth=0.018, arrowsize=0.05291667cm 2.0,arrowlength=0.8,arrowinset=0.2]{->}(1.42,0.3)(1.82,0.3)
\rput[bl](1.99,0.16){$\Delta_{\widetilde{H}}$}
\psline[linecolor=black, linewidth=0.018, arrowsize=0.05291667cm 2.0,arrowlength=0.8,arrowinset=0.2]{->}(2.62,0.5)(3.02,0.9)
\psline[linecolor=black, linewidth=0.018, arrowsize=0.05291667cm 2.0,arrowlength=0.8,arrowinset=0.2]{->}(2.62,0.1)(3.02,-0.3)
\rput[bl](2.88,0.06){$\vdots$}
\rput[bl](0.0,0.22){$\alpha$}
\psline[linecolor=black, linewidth=0.018, arrowsize=0.05291667cm 2.0,arrowlength=0.8,arrowinset=0.2]{->}(2.62,0.4)(3.02,0.6)
\psline[linecolor=black, linewidth=0.018, arrowsize=0.05291667cm 2.0,arrowlength=0.8,arrowinset=0.2]{->}(0.42,0.3)(0.82,0.3)
\rput[bl](1.06,0.24){$\iotar$}
\rput[bl](3.47,0.3){=}
\psline[linecolor=black, linewidth=0.018, arrowsize=0.05291667cm 2.0,arrowlength=0.8,arrowinset=0.2]{->}(4.52,0.3)(4.92,0.3)
\rput[bl](5.09,0.16){$\Delta_H$}
\psline[linecolor=black, linewidth=0.018, arrowsize=0.05291667cm 2.0,arrowlength=0.8,arrowinset=0.2]{->}(5.72,0.5)(6.12,0.9)
\psline[linecolor=black, linewidth=0.018, arrowsize=0.05291667cm 2.0,arrowlength=0.8,arrowinset=0.2]{->}(5.72,0.1)(6.12,-0.3)
\rput[bl](4.16,0.24){$\coint$}
\rput[bl](7.1,0.12){$\alpha$}
\psline[linecolor=black, linewidth=0.018, arrowsize=0.05291667cm 2.0,arrowlength=0.8,arrowinset=0.2]{->}(7.32,0.4)(7.52,0.7)
\rput[bl](8.72,-0.3){.}
\rput[bl](7.42,0.75){$m_{\widetilde{H}}$}
\rput[bl](7.1,-1.08){$\alpha$}
\psline[linecolor=black, linewidth=0.018, arrowsize=0.05291667cm 2.0,arrowlength=0.8,arrowinset=0.2]{->}(7.32,-0.8)(7.52,-0.5)
\rput[bl](7.42,-0.45){$m_{\widetilde{H}}$}
\psline[linecolor=black, linewidth=0.018, arrowsize=0.05291667cm 2.0,arrowlength=0.8,arrowinset=0.2]{->}(8.12,0.9)(8.52,0.9)
\psline[linecolor=black, linewidth=0.018, arrowsize=0.05291667cm 2.0,arrowlength=0.8,arrowinset=0.2]{->}(8.12,-0.3)(8.52,-0.3)
\psline[linecolor=black, linewidth=0.018, arrowsize=0.05291667cm 2.0,arrowlength=0.8,arrowinset=0.2]{->}(6.82,0.9)(7.22,0.9)
\psline[linecolor=black, linewidth=0.018, arrowsize=0.05291667cm 2.0,arrowlength=0.8,arrowinset=0.2]{->}(6.82,-0.3)(7.22,-0.3)
\rput[bl](6.32,0.8){$i_H$}
\rput[bl](6.32,-0.4){$i_H$}
\rput[bl](7.78,0.16){$\vdots$}
\rput[bl](6.48,0.16){$\vdots$}
\end{pspicture}

\end{figure}
\noindent by definition of $\iota_r$, where we write $\a\in\Aut(H)$ instead of $\rho(\a^*)$ for simplicity and $i_H:H\to\wt{H}$ is the inclusion. The tensor corresponding to an arc only puts the automorphism associated to the dual of the arc on the crossings through that arc. It follows that contracting all the $\wt{H}$-tensors from $\aaa\cup\bb$ is the same as contracting the $H$-tensors associated to the closed $\a$'s (of Subsection \ref{subs: tensors and HDs}) with the following $\wt{H}$-tensors for each $\b$:
\begin{figure}[H]
\centering
\begin{pspicture}(0,-1.5091925)(3.61,1.5091925)
\psline[linecolor=black, linewidth=0.018, arrowsize=0.05291667cm 2.0,arrowlength=0.8,arrowinset=0.2]{->}(1.1,0.0)(2.0,0.0)
\psline[linecolor=black, linewidth=0.018, arrowsize=0.05291667cm 2.0,arrowlength=0.8,arrowinset=0.2]{->}(1.7,0.70000005)(2.1,0.30000007)
\rput[bl](2.11,-0.16999994){$m_{\widetilde{H}}$}
\psline[linecolor=black, linewidth=0.018, arrowsize=0.05291667cm 2.0,arrowlength=0.8,arrowinset=0.2]{->}(2.8,0.0)(3.2,0.0)
\rput[bl](3.39,-0.12999994){$\widetilde{\boldsymbol{\mu}}$}
\rput[bl](1.3,0.8000001){$i_H$}
\rput[bl](0.98,0.52000004){$\alpha_{i_1}$}
\psline[linecolor=black, linewidth=0.018, arrowsize=0.05291667cm 2.0,arrowlength=0.8,arrowinset=0.2]{->}(1.5,0.50000006)(2.0,0.20000006)
\rput[bl](0.6,-0.09999994){$i_H$}
\rput[bl](0.68,-0.47999993){$\alpha_{i_2}$}
\psline[linecolor=black, linewidth=0.018, arrowsize=0.05291667cm 2.0,arrowlength=0.8,arrowinset=0.2]{->}(1.2,-0.29999995)(2.0,-0.09999994)
\psline[linecolor=black, linewidth=0.018, arrowsize=0.05291667cm 2.0,arrowlength=0.8,arrowinset=0.2]{->}(0.0,0.0)(0.5,0.0)
\psline[linecolor=black, linewidth=0.018, arrowsize=0.05291667cm 2.0,arrowlength=0.8,arrowinset=0.2]{->}(1.6,-0.79999995)(2.1,-0.29999995)
\rput[bl](1.3,-1.0999999){$i_H$}
\rput[bl](1.58,-1.4799999){$\alpha_{i_k}$}
\psline[linecolor=black, linewidth=0.018, arrowsize=0.05291667cm 2.0,arrowlength=0.8,arrowinset=0.2]{->}(1.9,-1.1999999)(2.2,-0.39999995)
\psline[linecolor=black, linewidth=0.018, arrowsize=0.05291667cm 2.0,arrowlength=0.8,arrowinset=0.2]{->}(0.9,-1.5)(1.2,-1.1999999)
\psarc[linecolor=black, linewidth=0.04, linestyle=dotted, dotsep=0.10583334cm, dimen=outer](1.05,-0.6499999){0.4}{180.0}{230.0}
\psline[linecolor=black, linewidth=0.018, arrowsize=0.05291667cm 2.0,arrowlength=0.8,arrowinset=0.2]{->}(0.9,1.5)(1.2,1.2)
\end{pspicture}
\end{figure}
\noindent where we write $\ov{\b}=\a_{i_1}\dots\a_{i_k}$ and, for the moment, we assume that all crossings through $\b$ are positive. Now, using the definition of the multiplication of $\wt{H}$, and that $\wt{\int}\circ i_H=\int$, the last tensor is the same as
\begin{figure}[H]
\centering
\begin{pspicture}(0,-1.3593855)(3.31,1.3593855)
\rput[bl](1.81,-0.020193024){$m_H$}
\psline[linecolor=black, linewidth=0.018, arrowsize=0.05291667cm 2.0,arrowlength=0.8,arrowinset=0.2]{->}(2.5,0.049806975)(2.9,0.049806975)
\rput[bl](3.09,-0.08019302){$\boldsymbol{\mu}$}
\psline[linecolor=black, linewidth=0.018, arrowsize=0.05291667cm 2.0,arrowlength=0.8,arrowinset=0.2]{->}(0.6,1.349807)(1.8,0.24980697)
\rput[bl](0.68,-0.030193023){$\alpha_{i_1}$}
\psline[linecolor=black, linewidth=0.018, arrowsize=0.05291667cm 2.0,arrowlength=0.8,arrowinset=0.2]{->}(1.3,0.049806975)(1.7,0.049806975)
\psline[linecolor=black, linewidth=0.018, arrowsize=0.05291667cm 2.0,arrowlength=0.8,arrowinset=0.2]{->}(0.0,0.049806975)(0.5,0.049806975)
\rput[bl](0.28,-0.930193){$\alpha_{i_1}\alpha_{i_2}\dots\alpha_{i_{k-1}}$}
\psline[linecolor=black, linewidth=0.018, arrowsize=0.05291667cm 2.0,arrowlength=0.8,arrowinset=0.2]{->}(1.4,-0.550193)(1.8,-0.15019302)
\psline[linecolor=black, linewidth=0.018, arrowsize=0.05291667cm 2.0,arrowlength=0.8,arrowinset=0.2]{->}(0.6,-1.350193)(0.9,-1.0501931)
\psarc[linecolor=black, linewidth=0.04, linestyle=dotted, dotsep=0.10583334cm, dimen=outer](1.25,-0.20019302){0.4}{180.0}{230.0}
\end{pspicture}
\end{figure}
\noindent which is the Fox calculus tensor of Subsection \ref{subs: tensors and HDs}. This shows that $\wtZ_{\wt{H}}^{\rho}(\HH)=Z_H^{\rho}(\HH)$ assuming all crossings are positive. If there is any negative crossings through $\b$, it is easy to see that the tensor $S_{\wt{H}}$ contributes the extra $\a_{i_j}^{-1}$ at the tail of $\a_{i_1}\a_{i_2}\dots\a_{i_{j-1}}$ for each negative $x\in\a_{i_j}\cap\b$ (see Notation \ref{def: bx}). This follows from
\begin{align*}
S_{\wt{H}}(h\cdot\a)=\a^{-1}\cdot S_H(h)=\a^{-1}(S_H(h))\cdot \a^{-1}.
\end{align*}
This shows that $\wtZ_{\wt{H}}^{\rho}(\HH)=Z_H^{\rho}(\HH)$ in all cases. That $\wtI_{\wt{H}}^{\rho}=I_H^{\rho}$ follows by the definitions of the refined invariants, since the distinguished group-like of $\wt{H}$ with the given relative cointegral structure is $a^*=\lH:\Aut(H)\to\kk$.

\end{proof}


\subsection{Hopf $G$-algebras and Fox calculus}\label{subs: Kup for Hopf G-algebra}

The Fox calculus formula of Subsection \ref{subs: tensors and HDs} is also quite transparent when thought in terms of Hopf $G$-algebras. Indeed, we show that Virelizier's extension of Kuperberg's invariant through Hopf $G$-algebras \cite{Virelizier:flat}, restricted to the Hopf $G$-algebra of Example \ref{example: Hopf group-algebra from semidirect product}, leads directly to our Fox calculus formula.

\def\PHH{P_{\HH}}
\medskip

We begin by briefly recalling the (dual) construction of \cite{Virelizier:flat}, using an arbitrary involutory finite type Hopf $G$-algebra $\uHH$. Let $\ucoint=\{\ca\}_{\a\in G}$ be a cointegral of $\uH$ and $\int:H_1\to\kk$ be an integral of $H_1$ such that $\int(\coint_1)=1$. Let $Y$ be a closed oriented 3-manifold endowed with a representation $\rho:\pi_1(Y)\to G$. Let $\HH=(\S,\aa,\bb)$ be an ordered, oriented, $\bb$-based Heegaard diagram of $Y$. For simplicity, we note $\Ha$ instead of $H_{\rho(\a^*)}$ and $\ca$ instead of $\coint_{\rho(\a^*)}$ for each $\a\in\aa$. Then, to each $\a\in\aa$ we associate the tensor 
 \begin{figure}[H]
 \centering
 \begin{pspicture}(0,-0.5572612)(2.0572608,0.5572612)
\psline[linecolor=black, linewidth=0.018, arrowsize=0.05291667cm 2.0,arrowlength=0.8,arrowinset=0.2]{->}(0.5,0.0)(0.9,0.0)
\psline[linecolor=black, linewidth=0.018, arrowsize=0.05291667cm 2.0,arrowlength=0.8,arrowinset=0.2]{->}(1.7,0.2)(2.1,0.6)
\psline[linecolor=black, linewidth=0.018, arrowsize=0.05291667cm 2.0,arrowlength=0.8,arrowinset=0.2]{->}(1.7,-0.2)(2.1,-0.6)
\rput[bl](1.96,-0.24){$\vdots$}
\psline[linecolor=black, linewidth=0.018, arrowsize=0.05291667cm 2.0,arrowlength=0.8,arrowinset=0.2]{->}(1.7,0.1)(2.1,0.3)
\rput[bl](0.0,-0.1){$\coint_{\alpha}$}
\rput[bl](1.1,-0.1){$\Delta_{\alpha}$}
\end{pspicture}
 
 \end{figure}
 \noindent where the are as many outcoming legs as crossings through $\a$ and to each crossing $x$ we associate the tensor
 
 \begin{figure}[H]
 \centering
 \begin{pspicture}(0,-0.16)(1.75,0.16)
\psline[linecolor=black, linewidth=0.018, arrowsize=0.05291667cm 2.0,arrowlength=0.8,arrowinset=0.2]{->}(0.0,0.04)(0.4,0.04)
\rput[bl](0.55,-0.16){$S^{\epsilon_x}_{\alpha}$}
\psline[linecolor=black, linewidth=0.018, arrowsize=0.05291667cm 2.0,arrowlength=0.8,arrowinset=0.2]{->}(1.2,0.04)(1.6,0.04)
\rput[bl](1.8,0.04){.}
\end{pspicture}
 
 \end{figure}
 
\noindent  Now, suppose a curve $\b\in\bb$ has associated the word
$$\ov{\b}=\a_{i_1}^{m_1}\dots \a_{i_k}^{m_k}$$
when starting from its basepoint and following its orientation (as usual, we denote $\a^*\in\pi_1(M,p)$ just by $\a$). Since $\ov{\b}=1$ in $\pi_1(M,p)$, the iterated multiplication $m_{\a_{i_1}^{m_1}, \dots, \a_{i_k}^{m_k}}$ of $\uH$ has target $H_1$, which can then be composed with the integral. Thus, we associate to each $\b$ the tensor 

\begin{figure}[H]
\centering
\begin{pspicture}(0,-0.6091925)(3.5591924,0.6091925)
\psline[linecolor=black, linewidth=0.018, arrowsize=0.05291667cm 2.0,arrowlength=0.8,arrowinset=0.2]{->}(0.009192505,-0.6)(0.4091925,-0.2)
\psline[linecolor=black, linewidth=0.018, arrowsize=0.05291667cm 2.0,arrowlength=0.8,arrowinset=0.2]{->}(0.009192505,0.6)(0.4091925,0.2)
\rput[bl](0.06919251,-0.24){$\vdots$}
\psline[linecolor=black, linewidth=0.018, arrowsize=0.05291667cm 2.0,arrowlength=0.8,arrowinset=0.2]{->}(2.4091926,0.0)(2.8091924,0.0)
\psline[linecolor=black, linewidth=0.018, arrowsize=0.05291667cm 2.0,arrowlength=0.8,arrowinset=0.2]{->}(0.009192505,0.3)(0.4091925,0.1)
\rput[bl](0.5091925,-0.2){$m_{\alpha_{i_1}^{m_1},\dots,\alpha_{i_k}^{m_k}}$}
\rput[bl](3.0091925,-0.1){$\int$}
\end{pspicture}

\end{figure}

\noindent where there are as many incoming legs as crossings through $\b$. These tensors can be contracted as usual, leading to a scalar $Z_{\uH}^{\rho}(\HH)\in\kk$. The main theorem of \cite{Virelizier:flat} is that this is an invariant of $(Y,\rho)$ when $\uH$ is involutory with $\dim(H_1)\neq 0$ in $\kk$, we denote it by $I_{\uH}^{\rho}(Y)$. 
\medskip

\begin{proposition}
\label{prop: Virelizier at sd product is ours}
Let $\uH$ be the Hopf $\Aut(H)$-algebra associated to the semidirect product $\kk[\Aut(H)]\ltimes H$ where $H$ is a finite dimensional involutory Hopf algebra with $\dim(H)\neq 0$ in $\kk$. Let $Y$ be a closed 3-manifold and let $(M_0,\c_0)$ be the associated sutured manifold, that is, $M_0=Y\sm B^3$ where $B$ is an embedded 3-ball in $Y$ and $\c_0$ is a single suture in $\p M_0$. Then $$I_{\uH}^{\rho}(Y)=I_H^{\rho}(M_0,\c_0)$$ 
for any $\rho:\pi_1(Y)\to\Aut(H)$.
\end{proposition}

\begin{proof}
Let $\HH=\HD$ be a (ordered, oriented, $\bb$-based) Heegaard diagram of $Y$ and let $\HH_0$ be the resulting diagram for $(M_0,\c_0)$ obtained by deleting a small disk to $\S$. For each $\a\in\Aut(H)$, identify $\Ha\cong H$ via the coalgebra isomorphism $h\cdot\a\mapsto h$ (recall $\Ha=\{h\cdot \a \ | \ h\in H\}\sb\kk[\Aut(H)]\ltimes H$). The result follows by writing the above $\uH$-tensors associated to $\HH$ in terms of the structure tensors of $H$ and its automorphisms. Indeed, since our identification is a coalgebra isomorphism, the above $\a$-tensors immediately correspond to those of Subsection \ref{subs: tensors and HDs}. Now, the above tensor associated to a crossing corresponds to
\begin{figure}[H]
\centering
\begin{pspicture}(0,-0.14)(2.8265588,0.14)
\psline[linecolor=black, linewidth=0.018, arrowsize=0.05291667cm 2.0,arrowlength=0.8,arrowinset=0.2]{->}(0.0,-0.04)(0.4,-0.04)
\rput[bl](0.55,-0.14){$S^{\epsilon_x}$}
\psline[linecolor=black, linewidth=0.018, arrowsize=0.05291667cm 2.0,arrowlength=0.8,arrowinset=0.2]{->}(1.2,-0.04)(1.6,-0.04)
\psline[linecolor=black, linewidth=0.018, arrowsize=0.05291667cm 2.0,arrowlength=0.8,arrowinset=0.2]{->}(2.6,-0.04)(3,-0.04)
\rput[bl](1.75,-0.14){$\alpha^{-\e_x}$}
\rput[bl](3.2,-0.04){.}
\end{pspicture}
\end{figure}

\noindent This is not the same as the corresponding tensor of Subsection \ref{subs: tensors and HDs}, but it will be after composition with the $\b$-tensors. Indeed, under the above identification the structure map $m_{\a_1,\a_2}$ becomes

\begin{figure}[H]
\centering
\begin{pspicture}(0,-0.5125166)(2.6538305,0.5125166)
\psline[linecolor=black, linewidth=0.018, arrowsize=0.05291667cm 2.0,arrowlength=0.8,arrowinset=0.2]{->}(0.9038306,-0.19990616)(1.3038306,-0.09990616)
\psline[linecolor=black, linewidth=0.018, arrowsize=0.05291667cm 2.0,arrowlength=0.8,arrowinset=0.2]{->}(0.0038305663,0.5000938)(1.3038306,0.10009384)
\psline[linecolor=black, linewidth=0.018, arrowsize=0.05291667cm 2.0,arrowlength=0.8,arrowinset=0.2]{->}(2.0038307,0.0)(2.4038305,0.0)
\rput[bl](1.5038306,-0.049906157){$m$}
\psline[linecolor=black, linewidth=0.018, arrowsize=0.05291667cm 2.0,arrowlength=0.8,arrowinset=0.2]{->}(0.0038305663,-0.49990615)(0.40383056,-0.39990616)
\rput[bl](0.48383057,-0.40990615){$\alpha_1$}
\rput[bl](2.6038306,-0.1){,}
\end{pspicture}
\end{figure}

\noindent hence, the iterated multiplication $m_{\a_{i_1}^{m_1},\dots, \a_{i_k}^{m_k}}$ is given by 
\begin{figure}[H]
\centering
\begin{pspicture}(0,-1.1113902)(3.1066833,1.1113902)
\psline[linecolor=black, linewidth=0.018, arrowsize=0.05291667cm 2.0,arrowlength=0.8,arrowinset=0.2]{->}(1.3066833,0.09976029)(1.8066833,0.09976029)
\psline[linecolor=black, linewidth=0.018, arrowsize=0.05291667cm 2.0,arrowlength=0.8,arrowinset=0.2]{->}(0.0066833496,1.0997603)(1.8066833,0.19976029)
\psline[linecolor=black, linewidth=0.018, arrowsize=0.05291667cm 2.0,arrowlength=0.8,arrowinset=0.2]{->}(2.5066833,0.09976029)(2.9066834,0.09976029)
\rput[bl](2.0066833,0.049760286){$m$}
\psline[linecolor=black, linewidth=0.018, arrowsize=0.05291667cm 2.0,arrowlength=0.8,arrowinset=0.2]{->}(0.0066833496,0.09976029)(0.50668335,0.09976029)
\rput[bl](0.6066834,-0.050239716){$\alpha_{i_1}^{m_1}$}
\psline[linecolor=black, linewidth=0.018, arrowsize=0.05291667cm 2.0,arrowlength=0.8,arrowinset=0.2]{->}(1.3066833,-0.3002397)(1.8066833,0.0)
\rput[bl](0.30668336,-0.75023973){$\alpha_{i_1}^{m_1}\dots\alpha_{i_{k-1}}^{m_{k-1}}$}
\psline[linecolor=black, linewidth=0.018, arrowsize=0.05291667cm 2.0,arrowlength=0.8,arrowinset=0.2]{->}(0.0066833496,-1.1002398)(0.50668335,-0.80023974)
\rput[bl](3.0566833,0.09976029){.}
\rput[bl](0.06668335,-0.5402397){$\vdots$}
\end{pspicture}
\end{figure}

\noindent Thus, whenever a crossing $x$ is negative, the $\a^{-1}$ coming from the antipode tensor provides the $\a^{-1}$ in the tail of $\ov{\b}_x$ as in the $\b$-tensors of Subsection \ref{subs: tensors and HDs}. This immediately implies that $$Z^{\rho}_{\uH}(\HH)=Z_H^{\rho}(\HH_0)$$
where the left hand side is the above tensor and the right hand side is ours, hence also $I_{\uH}^{\rho}(Y)=I_H^{\rho}(M_0,\c_0)$.

\end{proof}

\begin{remark}
\label{remark: Spinc and o vs non-unimodular}
The conditions on $H$ in the above proposition imply semisimplicity of $H$ and $H^*$ \cite[Theorem 10.4.3]{Radford:BOOK}. However, in Theorem \ref{Theorem: I for sd product is Fox calculus}, $H$ is not necessarily semisimple, for instance if $H$ is a Hopf superalgebra with $H_1\neq 0$ \cite[Corollary 3.1.2]{AEG:Triangular} or if $\char(\kk)=p>0$ and $p\mid\dim(H)$ (also by \cite[Theorem 10.4.3]{Radford:BOOK}). In this case, the Hopf $\Aut(H)$-algebra $\uH$ coming from $\kk[\Aut(H)]\ltimes H$ is non-semisimple, in particular, it may be non-unimodular so the construction of \cite{Virelizier:flat} does not applies. It turns out that there are two sorts of non-unimodularity of $\uH$: on the one hand, the graded-cointegral of such $\uH$ is not two-sided if $\lH\nequiv 1$ (see Example \ref{example: Hopf group-algebra from semidirect product}). On the other hand, even if we assume that the cointegral of the neutral component (which is $H$) is two-sided, $H$ is not unimodular if the cointegral has mod 2 degree one (see Remark \ref{remark: unimodularity}). In our setting, the $\Spinc$ structure takes care of the former failure of unimodularity while the orientation $\o$ takes care of the latter. In particular, $\Spinc$ only appears in the non-unimodular $G$-graded setting, this is why this additional structure is not present in other works on Kuperberg invariants.

\end{remark}


The reason we do not define sutured manifold invariants from more general (involutory) Hopf $G$-algebras is the following. Suppose we want to extend the construction of \cite{Virelizier:flat} sketched above to an extended Heegaard diagram of a sutured 3-manifold. Then to each arc $\a\in\bolda$ we need to associate a tensor
 \begin{figure}[H]
 \centering
 \begin{pspicture}(0,-0.5572612)(2.0572608,0.5572612)
\psline[linecolor=black, linewidth=0.018, arrowsize=0.05291667cm 2.0,arrowlength=0.8,arrowinset=0.2]{->}(0.5,0.0)(0.9,0.0)
\psline[linecolor=black, linewidth=0.018, arrowsize=0.05291667cm 2.0,arrowlength=0.8,arrowinset=0.2]{->}(1.7,0.2)(2.1,0.6)
\psline[linecolor=black, linewidth=0.018, arrowsize=0.05291667cm 2.0,arrowlength=0.8,arrowinset=0.2]{->}(1.7,-0.2)(2.1,-0.6)
\rput[bl](1.96,-0.24){$\vdots$}
\psline[linecolor=black, linewidth=0.018, arrowsize=0.05291667cm 2.0,arrowlength=0.8,arrowinset=0.2]{->}(1.7,0.1)(2.1,0.3)
\rput[bl](0.0,-0.1){$i_{\alpha}$}
\rput[bl](1.1,-0.1){$\Delta_{\alpha}$}
\end{pspicture}
 
 \end{figure}
 
\def\ia{i_{\a}} \def\iam{i_{\a^{-1}}}

\noindent for some special element $i_{\a}\in \Ha$ (recall that we denote $\Ha$ instead of $H_{\rho(\a^*)}$ for simplicity). By the cointegral property, the scalar obtained by contracting all the tensors associated to curves and arcs is invariant under arc-curve slidings. However, if we want this scalar to be invariant under arc-arc slidings, so as to get an invariant $I_{\uH}^{\rho}(M,\c)$ of $(M,\c)$ together with $\rho:\pi_1(M)\to G$, then it is natural to require $i_{\a}$ to be a group-like element of $\Ha$ and that $m_{\a_1,\a_2}(i_{\a_1}\ot i_{\a_2})=i_{\a_1\a_2}$ for each $\a_1,\a_2\in G$. But if there exists $(i_{\a})_{\a\in G}$ satisfying these conditions, then $G$ acts on $H_1$ by conjugation by $i_{\a}$: if we set $\phi(\a)(h)=m_{\a,1,\a^{-1}}(\ia\ot h\ot \iam)$ for all $\a\in G, h\in H_1$, then $\phi:G\to\Aut(H_1)$ is a group homomorphism. Moreover, the map $\Ha\to (H_1)_{\phi(\a)}, h\mapsto \phi(\a)\cdot (i_{\a^{-1}}h)$ is a Hopf morphism from the original Hopf $G$-algebra to the Hopf $\Aut(H_1)$-algebra built from $\kk[\Aut(H_1)]\ltimes H_1$ and this map preserves the cointegrals. It follows that the scalar invariant obtained from $\uH$ and $(\ia)_{\a\in G}$ can be obtained through the latter semidirect product, that is $I_{\uH}^{\rho}(M,\c)=I_{H_1}^{\phi\circ\rho}(M,\c)$, thus we do not gain anything new.

\subsection{Relative integrals and Hopf $G$-(co)algebras}\label{subs: rel int and Hopf G-alg} \def\wtint{\wt{\int}}\def\wtH{\wt{H}} \def\wtHb{\wtH_{\b}}
 
We now explain the relation between the relative integral setting of \cite{LN:Kup} and Hopf group-(co)algebras. Suppose $\wt{H}$ is equipped with a relative right cointegral $\iota_r:A\to\wt{H}$ and a relative right integral $\wtint:\wtH\to B$, where $A\sb\wtH$ is a cocentral subalgebra and $B\sb\wtH$ is a central subalgebra. Let's assume a while that $A=\kk$. Then it is well-known that the subalgebra $B$ induces a Hopf group-coalgebra with group $G\eq\Hom_{alg}(B,\kk)$ as follows: for each $\b\in G$, let $I_{\b}$ be the two-sided ideal of $\wtH$ generated by elements of the form $b-\b(b)1_{\wtH}$ with $b\in B$ and set $\wtHb\eq \wtH/I_{\b}$, so each $\wtHb$ is an algebra. It is easy to see that the coproduct $\De_{\wtH}$ descends to $\De_{\b_1\b_2}:\wtH_{\b_1\b_2}\to \wtH_{\b_1}\ot \wtH_{\b_2}$ and similarly the antipode descends to $S_{\b}:\wtHb\to\wtH_{\b^{-1}}$. Hence, $\{\wtHb\}_{\b\in G}$ is a Hopf group-coalgebra with the structure maps induced from $\wt{H}$. In these terms, a relative integral $\wtint:\wtH\to B$ as in \cite{LN:Kup} induces a Hopf group-coalgebra integral $\{\int_{\b}\}_{\b\in G}$ as in \cite{Virelizier:Hopfgroup}. Indeed, it is easy to see that for each $\b\in\Hom_{alg}(B,\kk)$ the map $\b\circ \wtint:\wt{H}\to\kk$ descends to a map $\int_{\b}:\wtHb\to\kk$. In other words, whenever $A=\kk$, the relative integral approach of \cite{LN:Kup} is an explicit form of the Hopf group-coalgebra approach of \cite{Virelizier:flat}, extended to sutured manifolds. Dually, the cocentral subalgebra $A\sb \wtH$ induces a Hopf $G(A)$-algebra structure and the relative cointegral is equivalent to a Hopf group-algebra cointegral (this generalizes Example \ref{example: Hopf group-algebra from semidirect product}). Whenever both $A,B$ are non-trivial, we are naturally led to an algebra-graded and coalgebra-cograded object, which is not a particular case of \cite{Virelizier:flat}, but we do not have interesting concrete examples.

\appendix

\section{The non-abelian case}
\label{subs: non-abelian case}
 In \cite{LN:Kup} we supposed $\rho$ is an abelian representation, though this is not strictly necessary. Everything works for non-abelian $\rho$, but we have to take care of the basepoint $p$ when we isotope the arcs of an extended Heegaard diagram. We now work out the necessary extra details. 
\medskip
 
Let $(M,\c)$ be a balanced sutured 3-manifold with connected $R_-(\c)$ and let $\HH=\HDD$ be an oriented, extended Heegaard diagram of it. Let $\aaa=\aa\cup\bolda$ where $\aa=\{\a_1,\dots,\a_d\}$ are the closed curves and $\bolda=\{\a_{d+1},\dots,\a_{d+l}\}$ are the arcs. Then the dual curves $\a^*\in\pi_1(M,p)$ are invariant under isotopy of $\aaa$ except for an overall conjugation if an arc $\a\in\bolda$ is isotoped along $\p\S$ past the basepoint $p$. More precisely, suppose the arc $\a$ is oriented so that $\a\cdot\d=+1$ where $\d$ is the oriented boundary component of $\p\S$ containing $p$. Suppose further that the arc $\a$ is isotoped past $p$ in the opposite direction of $\d$, and denote by $\a'$ the new arc and by $\HH'$ the oriented extended Heegaard diagram obtained by replacing $\a$ with $\a'$. If we denote by $\a'_i$ the curves $\a_i$ in $\HH'$, then $$(\a'_i)^*=\a^*\a^*_i(\a^*)^{-1}$$ in $\pi_1(M,p)$ for all $i=1,\dots,d+l$, see Figure \ref{fig: dual curves}.

\begin{figure}[h]
\centering
\begin{pspicture}(0,-2.3)(11.909935,1.4150095)
\definecolor{colour0}{rgb}{0.0,0.6,1.0}
\psline[linecolor=black, linewidth=0.04, arrowsize=0.05291667cm 2.0,arrowlength=1.4,arrowinset=0.0]{<-}(0.169935,1.0050094)(4.569935,1.0050094)
\psline[linecolor=red, linewidth=0.04](0.569935,1.0050094)(0.569935,0.20500946)
\psline[linecolor=red, linewidth=0.04](2.569935,1.0050094)(2.569935,0.20500946)
\psline[linecolor=red, linewidth=0.04, linestyle=dotted, dotsep=0.10583334cm, arrowsize=0.05291667cm 2.0,arrowlength=1.4,arrowinset=0.0]{<-}(0.569935,0.20500946)(0.569935,-0.59499055)
\psline[linecolor=red, linewidth=0.04, linestyle=dotted, dotsep=0.10583334cm, arrowsize=0.05291667cm 2.0,arrowlength=1.4,arrowinset=0.0]{<-}(2.569935,0.20500946)(2.569935,-0.59499055)
\psdots[linecolor=black, dotsize=0.014](2.969935,1.0050094)
\psdots[linecolor=black, dotsize=0.014](2.969935,1.0050094)
\psdots[linecolor=black, dotsize=0.014](2.969935,1.0050094)
\psdots[linecolor=black, dotsize=0.014](2.969935,1.0050094)
\psline[linecolor=black, linewidth=0.04, arrowsize=0.05291667cm 2.0,arrowlength=1.4,arrowinset=0.0]{<-}(6.969935,1.0050094)(11.369935,1.0050094)
\psdots[linecolor=black, dotsize=0.014](9.769935,1.0050094)
\psdots[linecolor=black, dotsize=0.014](9.769935,1.0050094)
\psdots[linecolor=black, dotsize=0.014](9.769935,1.0050094)
\psdots[linecolor=black, dotsize=0.014](9.769935,1.0050094)
\psbezier[linecolor=red, linewidth=0.04](9.369935,0.20500946)(9.369935,1.0050094)(10.969935,0.20500946)(10.969935,1.0050094604492188)
\psline[linecolor=red, linewidth=0.04, linestyle=dotted, dotsep=0.10583334cm, arrowsize=0.05291667cm 2.0,arrowlength=1.4,arrowinset=0.0]{<-}(9.369935,0.20500946)(9.369935,-0.59499055)
\psline[linecolor=blue, linewidth=0.04](3.569935,1.0050094)(3.569935,0.20500946)
\psline[linecolor=blue, linewidth=0.04, arrowsize=0.05291667cm 2.0,arrowlength=1.4,arrowinset=0.0]{->}(0.769935,0.40500945)(0.369935,0.40500945)
\psbezier[linecolor=blue, linewidth=0.04, linestyle=dotted, dotsep=0.10583334cm](0.369935,0.40500945)(-0.030065002,0.40500945)(-0.230065,-0.79499054)(0.569935,-0.7949905395507812)(1.369935,-0.79499054)(1.169935,0.40500945)(0.769935,0.40500945)
\psbezier[linecolor=blue, linewidth=0.04, linestyle=dotted, dotsep=0.10583334cm](3.569935,-0.79499054)(3.569935,-1.5949905)(0.569935,-1.5949905)(0.569935,-0.7949905395507812)
\psline[linecolor=blue, linewidth=0.04, linestyle=dotted, dotsep=0.10583334cm](10.369935,0.0050094603)(10.369935,-0.79499054)
\psline[linecolor=blue, linewidth=0.04](10.169935,0.20500946)(10.169935,0.0050094603)
\psdots[linecolor=black, dotsize=0.2](3.569935,1.0050094)
\psline[linecolor=blue, linewidth=0.04, linestyle=dotted, dotsep=0.10583334cm](3.569935,-0.79499054)(3.569935,0.20500946)
\psline[linecolor=red, linewidth=0.04](7.369935,1.0050094)(7.369935,0.20500946)
\psline[linecolor=red, linewidth=0.04, linestyle=dotted, dotsep=0.10583334cm, arrowsize=0.05291667cm 2.0,arrowlength=1.4,arrowinset=0.0]{<-}(7.369935,0.20500946)(7.369935,-0.59499055)
\psline[linecolor=blue, linewidth=0.04, arrowsize=0.05291667cm 2.0,arrowlength=1.4,arrowinset=0.0]{->}(7.569935,0.40500945)(7.169935,0.40500945)
\psbezier[linecolor=blue, linewidth=0.04, linestyle=dotted, dotsep=0.10583334cm](7.169935,0.40500945)(6.769935,0.40500945)(6.569935,-0.79499054)(7.369935,-0.7949905395507812)(8.169935,-0.79499054)(7.969935,0.40500945)(7.569935,0.40500945)
\psline[linecolor=colour0, linewidth=0.04, arrowsize=0.05291667cm 2.0,arrowlength=1.4,arrowinset=0.0]{->}(9.569935,0.40500945)(9.169935,0.40500945)
\psbezier[linecolor=colour0, linewidth=0.04, linestyle=dotted, dotsep=0.10583334cm](9.169935,0.40500945)(8.769935,0.40500945)(8.569935,-0.79499054)(9.369935,-0.7949905395507812)(10.169935,-0.79499054)(9.969935,0.40500945)(9.569935,0.40500945)
\psbezier[linecolor=blue, linewidth=0.04, linestyle=dotted, dotsep=0.10583334cm](10.369935,-0.79499054)(10.369935,-1.5949905)(7.369935,-1.5949905)(7.369935,-0.7949905395507812)
\psline[linecolor=colour0, linewidth=0.04, arrowsize=0.05291667cm 2.0,arrowlength=1.4,arrowinset=0.0]{->}(2.769935,0.40500945)(2.369935,0.40500945)
\psbezier[linecolor=colour0, linewidth=0.04, linestyle=dotted, dotsep=0.10583334cm](2.369935,0.40500945)(1.9699349,0.40500945)(1.769935,-0.79499054)(2.569935,-0.7949905395507812)(3.369935,-0.79499054)(3.169935,0.40500945)(2.769935,0.40500945)
\psbezier[linecolor=blue, linewidth=0.04, linestyle=dotted, dotsep=0.10583334cm](8.769935,0.40500945)(8.369935,-0.19499055)(8.569935,-0.9949905)(9.369935,-0.9949905395507812)(10.169935,-0.9949905)(10.169935,-0.19499055)(10.169935,0.0050094603)
\psbezier[linecolor=blue, linewidth=0.04](10.369935,1.0050094)(10.369935,0.60500944)(9.169935,1.0050094)(8.769935,0.40500946044921876)
\psbezier[linecolor=blue, linewidth=0.04](10.169935,0.20500946)(10.169935,0.40500945)(10.369935,0.40500945)(10.369935,0.20500946044921875)
\psline[linecolor=blue, linewidth=0.04](10.369935,0.20500946)(10.369935,0.0050094603)
\psdots[linecolor=black, dotsize=0.2](10.369935,1.0050094)
\rput[bl](2.569935,1.2050095){$\a$}
\rput[bl](10.969935,1.2050095){$\a'$}
\rput[bl](0.569935,1.2050095){$\a_i$}
\rput[bl](7.369935,1.2050095){$\a_i'$}
\rput[bl](3.569935,1.2050095){$p$}
\rput[bl](10.369935,1.2050095){$p$}
\rput[bl](4.969935,1.0050094){$\delta$}
\rput[bl](11.769935,1.0050094){$\delta$}
\end{pspicture}
\caption{An arc $\a$ is slided past the basepoint $p\in\d\sb \p\S$. On the left we draw $\a^*_i$ as $c A_ic^{-1}$, where $A_i$ is a circle in $\inte(\S)$ intersecting $\aaa$ once at $\a_i$ and $c$ is an arc disjoint from $\aaa$ from the basepoint to a point in $A_i$. On the right we see that the arc $c$ has to be slided over $\a^*$ to avoid an intersection with $\a$. This shows that $(\a'_i)^*=\a^* a^*_i(\a^*)^{-1}$.}\label{fig: dual curves}

\end{figure}

\medskip

Thus, to guarantee that $I_H^{\rho}$ is a topological invariant for non-abelian $\rho$ we need to show that its defining formula is invariant under conjugation of $\rho$ by an element of $\Aut(H)$. This is done in the next lemma.

\begin{lemma}
\label{prop: invariance under conjugation of rho}
Given $\phi\in\Aut(H)$, let $\rhophi:\pi_1(M,p)\to\Aut(H)$ be the homomorphism defined by $\rhophi(x)\eq\phi\circ\rho(x)\circ\phi^{-1}$ for any $x\in\pi_1(M,p)$. Then $Z^{\rhophi}_H(\HH)=Z^{\rho}_H(\HH)$. Hence, $I_H^{\rho}$ is a topological invariant for non-abelian $\rho$ and depends on $\rho$ only up to conjugation.
\end{lemma}
\begin{proof}
\def\Tg{T_g}\def\caa{c_{\aa}}\def\iaa{i_{\aa}}
Since $\phi$ is a Hopf automorphism, we have
\begin{align*}
K_H^{\rhophi}(\HH)=\phi^{\ot d}\circ K_H^{\rho}(\HH)\circ (\phi^{-1})^{\ot d}.
\end{align*}
Now, since $\phi(\coint)=\lH(\phi)\coint$ and $\int\circ\phi=\lH(\phi)\int$ we get
\begin{align*}
Z^{\rhophi}_H(\HH)=\lH(\phi)^d\lH(\phi^{-1})^dZ^{\rho}_H(\HH)=Z^{\rho}_H(\HH).
\end{align*}
The second assertion follows from this and from the fact that $\lH\circ\rho=\lH\circ\rhophi$, which holds since the target of this map is an abelian group.
\end{proof}
\medskip

This lemma also implies that $I_H^{\rho}$ is independent of the basepoint $p$ chosen. More precisely, let $p_1,p_2\in s(\c)$ be basepoints and let $\d:[0,1]\to M$ be a path from $p_1$ to $p_2$. Denote by $C_{[\d]}$ the isomorphism $\pi_1(M,p_1)\to \pi_1(M,p_2)$ defined by $C_{[\d]}(\a)=[\ov{\d}]\a [\d]$ for $\a\in\pi_1(M,p_1)$.
\medskip

\begin{corollary}
Let $p_1,p_2\in s(\c)$ be two basepoints and let $\rho_i:\pi_1(M,p_i)\to \Aut(H)$ for $i=1,2$ be group homomorphisms related by $\rho_1=\rho_2\circ C_{[\d]}$ for some path $\d$ from $p_1$ to $p_2$. Then $I^{\rho_1}_H(M,\c,\ss,\o)=I^{\rho_2}_H(M,\c,\ss,\o)$.
\end{corollary}
\begin{proof}
It suffices to show that $Z^{\rho_1}_H(\HH)=Z^{\rho_2}_H(\HH)$. Note that changing the path $\d$ by another path changes $C_{[\d]}$ by an inner automorphism of $\pi_1(M,p_2)$. Therefore, by Proposition \ref{prop: invariance under conjugation of rho}, it suffices to prove the corollary for a specific path $\d$. Let $\a^*_{p_i}\in \pi_1(M,p_i)$ be the dual curves of the $\a$'s coming from the diagram $\HH$ with basepoint $p_i$, $i=1,2$. If we just let $\d$ be a path from $p_1$ to $p_2$ contained in $\S\sm \aaa$, which is connected since we assume $R_-(\c)$ connected, then $\a_{p_2}^*=[\ov{\d}]\a^*_{p_1}[\d]$.   Using $\rho_1=\rho_2\circ C_{[\d]}$ we get $\rho_2(\a_{p_2}^*)=\rho_2(\ov{\d}\a^*_{p_1}\d)=\rho_1(\a_{p_1}^*)\in\Aut(H)$. It follows that $Z^{\rho_2}(\HH)=Z^{\rho_1}(\HH)$.
\end{proof}

When $R_-(\c)$ is disconnected, the above assertions also hold by our definition of $I_H^{\rho}$ in this case.

\bibliographystyle{amsplain}
\bibliography{/Users/daniel/Desktop/TEX/bib/referencesabr}

\end{document}